\documentclass[a4paper, 11pt]{amsart}
\usepackage{enumerate,subfigure}
\usepackage{amsmath}
\usepackage{amssymb}
\usepackage{epsfig}
\usepackage{eepic}
\usepackage{latexsym
}
\usepackage{color}

\usepackage{tabularx}
\usepackage{tikz}
\usepackage{dsfont}
\newtheorem{theorem}{Theorem}[section]

\newtheorem{definition}[theorem]{Definition}

\newtheorem{assumption}[theorem]{Assumption}

\newtheorem{remark}{Remark}[section]

\usepackage[latin1]{inputenc}

\usepackage[top=2 cm,bottom=2 cm,left=2. cm, right=2. cm]{geometry}

\newcommand {\sgn} { {\rm sgn} }

\newcommand{\beq}{\begin{equation}}
\newcommand{\eeq}{\end{equation}}
\newcommand{\beqa}{\begin{eqnarray}}
\newcommand{\eeqa}{\end{eqnarray}}

\numberwithin{equation}{section}

\begin{document}

\title[From individual-based mechanical models to free-boundary problems]{From individual-based mechanical models of multicellular systems to free-boundary problems} 
\author[T.~Lorenzi, P.J.~Murray, M.~Ptashnyk]{{ Tommaso Lorenzi, Philip J. Murray, Mariya Ptashnyk}\medskip\\
  University of St Andrews, UK \\
University of Dundee, UK \\
Heriot-Watt University, Edinburgh,  UK   }  \footnotetext[1]{accepted in {\it  Interfaces and Free Boundaries} \\
 tl47@st-andrews.ac.uk,  p.murray@dundee.ac.uk,   m.ptashnyk@hw.ac.uk
}

\date{}

\maketitle

%%%%%%%%

\begin{abstract}
In this paper we present an individual-based mechanical model that describes the dynamics of two contiguous cell populations with different proliferative and mechanical characteristics. An off-lattice modelling approach is considered whereby: (i) every cell is identified by the position of its centre;  (ii) mechanical interactions between cells are described via generic nonlinear force laws; and (iii) cell proliferation is contact inhibited. We formally show that the continuum counterpart of this discrete model is given by a free-boundary problem for the  cell densities. The results of the derivation demonstrate how the parameters of continuum mechanical models of multicellular systems can be related to biophysical cell properties. We prove an existence result for the free-boundary problem and construct travelling-wave solutions. Numerical simulations are performed in the case where the cellular interaction forces are described by the celebrated Johnson-Kendall-Roberts model of elastic contact, which has been previously used to model cell-cell interactions. The results obtained indicate excellent agreement between the simulation results for the individual-based model, the numerical solutions of the corresponding free-boundary problem and the travelling-wave analysis.

%\subjclass{35Bxx; 35C07; 35Q92; 35R35; 92-08; 92C10.}

%\keywords{cell mechanics; individual-based models; free-boundary problems; local existence; travelling waves; JKR model.}
\end{abstract}

\section{Introduction}
Continuum mechanical models of multicellular systems have been increasingly used to achieve a deeper understanding of the underpinnings of tissue development, wound healing and tumour growth~\cite{ambrosi2002mechanics, ambrosi2002closure, araujo2004history, bresch2010computational, byrne1995growth, byrne2003modelling, chen2001influence, ciarletta2011radial, greenspan1976growth, lowengrub2009nonlinear, perthame2014some, preziosi2003cancer, sherratt2001new, ward1999mathematical, ward1997mathematical}. These models are formulated in terms of nonlinear partial differential equations for  cell densities (or cell volume fractions) and, as such, are amenable to both numerical and analytical approaches that enable insight to be gained into the biological system under study. From a mathematical perspective, over the past few years particular attention has been paid to the existence of travelling-wave solutions with composite shapes~\cite{bertsch2013modeling, chaplain2019bridging, lorenzi2017interfaces, perthame2014traveling, tang2014composite} and to the convergence to free-boundary problems in the asymptotic limit whereby cells are represented as an incompressible fluid~\cite{bubba2019hele, kim2016free, mellet2017hele, perthame2014derivation, perthame2014hele}. 

Whilst continuum mechanical models of multicellular systems are usually defined on the basis of tissue-scale phenomenological considerations, off-lattice individual-based models enable representation of cell mechanics at the level of individual cells~\cite{drasdo2005coarse,van2015simulating}. However, as the numerical exploration of such individual-based models requires large computational times for biologically relevant cell numbers and the models are not analytically tractable, it is desirable to derive continuum models in an appropriate limit~\cite{baker2018free, byrne2009individual, deroulers2009modeling,drasdo2005coarse, fozard2009continuum, johnston2012mean,johnston2015modelling, middleton2014continuum, motsch2018short, Murray, murray2009discrete,simpson2007simulating}. Although mechanical interactions between interfacing cell populations with different characteristics arise in many biological contexts (\emph{e.g.} tumour growth, development), relatively little prior work has explored the connection between off-lattice individual-based models and continuum models in such situations. 

In this paper we propose an individual-based mechanical model for the dynamics of two contiguous cell populations with different proliferative and mechanical characteristics. In our model: (i) every cell is identified by the position of its centre; (ii) mechanical interactions between cells are described via generic nonlinear force laws; and (iii) cell proliferation is contact inhibited.  Formally deriving a continuum counterpart of the discrete model, we obtain a free-boundary problem with nonstandard transmission conditions that governs the dynamics of the  cell densities. Our derivation extends a previous method developed for the case of a single cell population \cite{Murray,murray2009discrete}. 

To prove an existence result for the free-boundary problem, a novel extension of methods previously developed for related free-boundary problems~\cite{DuLin2010, Evans_1, Evans_2,  Visintin}  is required due to the specific structure of our boundary and transmission conditions. In particular, the jump in the density and in the flux across the moving interface between the two cell populations, along with the fact that there are no conditions of Dirichlet type, prevent us from using existing ideas which are partially based on continuity arguments~\cite{Evans_1, Evans_2} and from applying the enthalpy method~\cite{Crowley, Visintin}. 

Moreover, building on a recently presented method  for a related system of nonlinear partial differential equations~\cite{chaplain2019bridging,lorenzi2017interfaces}, we also construct travelling-wave solutions for the free-boundary problem. In this respect, the novelty of our work lies in the fact that we consider a free-boundary problem subject to biomechanical transmission conditions which are different from those considered in~\cite{chaplain2019bridging,lorenzi2017interfaces}.   This  requires a different approach when studying the properties of the solution at the interface between the two cell populations and introduces a significant difference in the qualitative properties of the travelling wave.

Numerical simulations are performed in the case where the cellular interaction forces are described by the celebrated Johnson-Kendall-Roberts (JKR) model of elastic contact~\cite{johnson1971surface}, which has been shown to be experimentally accurate in some cases~\cite{chu2005johnson}, and has been previously used to approximate mechanical interactions between cells~\cite{drasdo2005single, drasdo2012modeling}. The results obtained support the findings of the travelling-wave analysis, and demonstrate excellent agreement between the individual-based model and the corresponding free-boundary problem.

The  paper is organised as follows. In Section~\ref{DiscreteToContinuum}, we present our individual-based mechanical model and formally derive the corresponding free-boundary problem. In Section~\ref{well-posedness}, we prove the existence result for the free-boundary problem. In Section~\ref{travelling-waves}, we develop the travelling-wave analysis. In Section~\ref{numerics}, we compare  simulation results for the individual-based model and numerical solutions of the free-boundary problem. Section~\ref{sec:discussion} concludes the paper and provides a brief overview of possible research perspectives.

\section{Formulation of the individual-based model and derivation of the corresponding free-boundary problem}
\label{DiscreteToContinuum}
\subsection{Formulation of the individual-based model}
We consider a one-dimensional multicellular system that consists of two populations of cells that are arranged along the real line $\mathbb{R}$ and characterised by different proliferative and mechanical properties. We label the two cell populations by the letters $A$ and $B$ and make the assumption that, during the considered time interval, the cells in population $A$ can proliferate, whereas the cells in population $B$ cannot. We denote the number of cells in population $B$ by $M>0$. Moreover, at  time  $\tau \geq 0$ we let the function $m(\tau)$ represent the number of cells in population $A$ and compute the total number of cells inside the system as $n(\tau)=m(\tau) + M$.

We adopt a discrete off-lattice modelling approach whereby every cell is identified by the position of its centre~\cite{van2015simulating}. Building upon the ideas presented in~\cite{Murray,murray2009discrete}, we model the two cell populations as a chain of masses and springs with the masses corresponding to the cell centres, and assume the cell order to be fixed. We label each cell by an index $i=1, \ldots, n(\tau)$ and describe the position of the $i^{th}$ cell's centre at time $\tau$ by means of the function $r_i(\tau)$. Without loss of generality, we let the cells of population $A$ be on the left of the cells of population $B$. 

We assume that the centre of the first cell of population $A$ 
is pinned at a point $s_0 \in \mathbb{R}$, i.e. 
\beq
\label{ass:firstcellpinned}
r_1(\tau) = s_0, \quad \text{for all } \; \tau \geq 0.
\eeq
We describe the effect of cell proliferation and mechanical interactions between cells on the dynamics of the multicellular system using the modelling strategies and the assumptions described hereafter.

\paragraph{Mathematical modelling of cell proliferation.} We assume that cell proliferation is contact dependent such that the proliferation rate $g$ of the $j^{th}$ cell in population $A$ depends on the position of neighbouring cells, \emph{i.e.} $g \equiv g\left(r_{j}(\tau) - r_{j-1}(\tau)\right)$ with $j = 2, \ldots, m(\tau) -1$. Hence in a time interval $\Delta \tau$ the proliferation of cells in population $A$ will result in a reindexing $\Delta i$ of the $i^{th}$ cell by an amount
\begin{align} \label{Deltaii}
(\Delta i)_i = \left \lceil \sum_{j=2}^i g(r_j-r_{j-1}) \, \Delta \tau \right \rceil \; \text{ for } \; i=2, \ldots, m -1
\end{align}
such that
\begin{align}
r_i(\tau+\Delta \tau) = r_{i-(\Delta i)_i}(\tau) \quad  \text{ for } \, i=2, \ldots, m-1,
\label{DiscreteModelReindexingEqn}
\end{align}
where $m\equiv m(\tau)$.
\paragraph{Mathematical modelling of mechanical interactions between cells.} We make the assumption that mechanical interactions between nearest neighbour cells depend on the distance between their centres. We denote the force exerted on the $i^{th}$ cell of population $l$ by its left and right neighbours by $F_{l}(r_{i}-r_{i-1})$ and $F_{l}(r_{i+1}-r_{i})$, respectively, and introduce the parameter $\eta_l >0$ to model the damping coefficient of cells in population $l$, where $l=A,B$. With this notation and neglecting cell-cell friction, the dynamics of the positions of the cell centres are described via the following system of differential equations
\begin{equation}\label{discrete_11}
\begin{aligned}
\frac{dr_i}{d\tau}  &=  \frac{1}{\eta_A} \, \big(F_A(r_i-r_{i-1}) -  F_A(r_{i+1} - r_i)\big), \,  \quad  i= 2, \ldots, m-1,  \\
 %&&  - \alpha \left(\sum_{j=2}^{i} \, g\left(r_{j} - r_{j-1}\right)\right) \, \left(r_{i} - r_{i-1}\right), \quad  i= 2, \ldots, m-1, \\
\frac{dr_i}{d\tau}  &= \frac{1}{\eta_B} \, \big( F_B(r_i-r_{i-1}) -  F_B(r_{i+1} - r_i)\big),  \,   \quad  i= m+1, \ldots, n-1,
% &&  - \alpha  \left(\sum_{j=2}^{m} \, g\left(r_{j} - r_{j-1}\right)\right) \, \left(r_{i} - r_{i-1}\right), \quad  i= m+1, \ldots, n, \nonumber
\end{aligned}
 \end{equation}
where $m\equiv m(\tau)$ and $n\equiv n(\tau)$. We complete   system  \eqref{discrete_11} with the following differential equations
\beqa\label{boundary_c_11}
\frac{d r_1}{d\tau} &=& 0, \nonumber \\
\frac{dr_m}{d\tau} &=&  \frac 1 {\eta_A} F_A(r_m-r_{m-1}) -  \frac1 {\eta_B} F_B(r_{m+1} - r_m),  \\ 
\frac{d r_n}{d\tau} &=&  \frac 1 { \eta_B} F_B(r_{n} - r_{n-1}). \nonumber 
\eeqa

\subsection{Derivation of the corresponding free-boundary problem}
In order to formally derive a continuum version of our individual-based mechanical model \eqref{discrete_11} and \eqref{boundary_c_11}, considering the scenario where the number of cells in both populations is large, we introduce the continuous variable $y \in \mathbb{R}$ so that, for some $\delta>0$ sufficiently small,  
$$
r_i(\tau)=r(\tau, y_i) \quad \text{with} \quad y_i=i \, \delta,
$$
and 
$$
r_{i\pm1}(\tau)=r(\tau, y_{i\pm1})=r(\tau, y_{i}\pm\delta), \quad  
r_{i-(\Delta i)_i}(\tau) = r(\tau, y_{i-(\Delta i)_i})=r(\tau, y_{i} - (\Delta i)_i \, \delta).
$$
Moreover, we use the notation
\beq
\label{eq:nots1s2}
r(\tau, y_1) = s_0, \quad r(\tau, y_m) = s_1(\tau), \quad r(\tau, y_n) = s_2(\tau), \quad \text{for  } \; \tau>0.
\eeq
We assume the function $r(\tau,y)$ to be continuously differentiable with respect to the variable $\tau$ and twice continuously differentiable with respect to the variable $y$.  Under these assumptions, letting $\Delta \tau$ and $\delta$ be sufficiently small, and using the Taylor expansions
$$
%\begin{equation}\label{eq:Taylor_r}
\begin{aligned}
& r(\tau + \Delta \tau,y_i) = r(\tau,y_i) + \frac{ \partial r(\tau,y_i)}{\partial \tau} \, \Delta \tau  + o(\Delta \tau), 
\\
& r(\tau, y_{i} - (\Delta i)_i \, \delta) = r(\tau,y_i) -\frac{ \partial r(\tau,y_i)}{\partial y} \,  (\Delta i)_i \, \delta + o(\delta),
\end{aligned} 
$$
along with the approximation
$$
\left \lceil \sum_{j=2}^i g(r_j-r_{j-1}) \, \Delta \tau \right \rceil \approx \int_{y_1}^{y_i} g\left(\frac{ \partial r(\tau,y')}{\partial y'} \, \delta \right) dy' \, \frac{\Delta \tau}{\delta},
$$
from~\eqref{DiscreteModelReindexingEqn} we obtain
\begin{align}
\frac{\partial r(\tau,y_{i})}{\partial \tau}  \approx- \left(\int_{y_1}^{y_i} g\left(\frac{ \partial r(\tau,y')}{\partial y'} \, \delta \right) dy'\right) \, \frac{ \partial r(\tau,y_{i})}{\partial y} \quad \text{for } i=2,\ldots, m -1. \label{e.approxdiv}
\end{align}
Moreover, using the Taylor expansions 
$$
\begin{aligned}
& r(\tau,y_i + \delta) = r(\tau,y_i) +\frac{ \partial r(\tau,y_i)}{\partial y} \, \delta  + \frac 12 \frac{\partial^2 r(\tau, y_i)}{\partial y^2} \, \delta^2 + o(\delta^2), 
\\
& r(\tau,y_i - \delta) = r(\tau,y_i) -\frac{ \partial r(\tau,y_i)}{\partial y} \, \delta  + \frac 12 \frac{\partial^2 r(\tau,y_i)}{\partial y^2} \, \delta^2+ o(\delta^2), 
\end{aligned} 
$$
and making the additional assumption that the functions $F_A$ and $F_B$ are twice continuously differentiable, we approximate the force terms in~\eqref{discrete_11} for $i = 2, \ldots, n-1$ as
\beq
\label{e.approxmech}
F_l(r_i-r_{i-1}) -  F_l(r_{i+1} - r_i) \approx - F^\prime_l\left(\frac{ \partial r}{\partial y} \delta \right) \frac{\partial^2 r}{\partial y^2} \delta^2, \quad l=A,B.
\eeq
Using the approximations  \eqref{e.approxdiv} and \eqref{e.approxmech}, and combining proliferation~\eqref{DiscreteModelReindexingEqn} and mechanical inte\-raction~\eqref{discrete_11} processes, we obtain 
\beq\label{eq:r_A}
 \frac{\partial r}{\partial \tau} = -  \frac{1}{\eta_A} \, F^\prime_A\left(\frac{ \partial r}{\partial y} \, \delta \right) \frac{\partial^2 r}{\partial y^2} \, \delta^2  -  \left(\int_{y_1}^{y} g\left(\frac{ \partial r}{\partial y'} \, \delta \right) dy'\right) \frac{ \partial r}{\partial y} \, \; \text{ for } \; y \in (y_1,y_m),
\eeq
and
\beq\label{eq:r_B}
 \frac{\partial r}{\partial \tau} = -  \frac{1}{\eta_B} \, F^\prime_B\left(\frac{ \partial r}{\partial y} \, \delta \right) \frac{\partial^2 r}{\partial y^2} \, \delta^2 \, \; \text{ for } \; y \in (y_m,y_n).
\eeq
Similarly, mechanical interaction processes~\eqref{boundary_c_11} for $i=1$, $i=m$ and $i=n$ yield, respectively,
\begin{equation}
\label{boundary_c_1}
\begin{aligned}
  \frac{\partial r }{\partial \tau}  &= 0 \hspace{7.9cm}  \text{ at } y = y_1,\\ 
  \frac{\partial r }{\partial \tau}  &=  \frac 1 {\eta_A}\left[ F_A\left(\frac{\partial r}{\partial y} \, \delta \right)  - \frac 12 F^\prime_A\left(\frac{ \partial r}{\partial y} \, \delta\right) \frac{\partial^2 r}{\partial y^2} \, \delta^2 \right]  \\
  & \quad - \frac 1 { \eta_B} \left[F_B\Big(\frac{\partial r}{\partial y} \, \delta  \Big)  +  \frac 12   F^\prime_B\left(\frac{ \partial r}{\partial y} \, \delta   \right) \frac{\partial^2 r}{\partial y^2} \, \delta^2  \right] \;\;\quad \text{ at } \; y = y_m,
\end{aligned}
\end{equation}
and
\begin{equation}\label{boundary_c_2}
\frac{\partial r}{\partial \tau} = \frac 1 {\eta_B} \left[F_B\left(\frac{\partial r}{\partial y} \, \delta \right)    -  \frac 12 F^\prime_B\left(\frac{ \partial r}{\partial y} \, \delta \right) \frac{\partial^2 r}{\partial y^2} \, \delta^2 \right] \; \quad \; \; \text{ at } \; y = y_n.
\end{equation}
Based on the ideas presented in~\cite{Murray,murray2009discrete} and according to the considerations given in Remark~\ref{remceldens}, we define the cell number densities of populations $A$ and $B$ as
\beq\label{eq:approxcelldensity}
\rho_A(\tau,y) = \left(\frac{\partial r}{\partial y} \, \delta\right)^{-1} \mbox{for } \; y \in [y_1,y_m], \quad \rho_B(\tau,y) = \left(\frac{\partial r}{\partial y} \, \delta\right)^{-1} \mbox{for } \; y \in [y_m,y_n].
\eeq
\begin{remark}
\label{remceldens}
{\it 
The definitions of the cell densities given by~\eqref{eq:approxcelldensity} are based on the observation that, at any time  $\tau$, the quotient of the number of cells in a generic interval $[r_i,r_j]$, with $j>i$, and the length of the interval is 
$$
\frac{j-i}{r_j(\tau)-r_i(\tau)} = \frac{j-i}{r(\tau,y_j)-r(\tau,y_i)}.
$$
From the above relation, choosing $j=i+1$ and using the fact that $\delta$ is small, we obtain the following approximate expression for the cell density
\beq\label{eq:approxcelldensity1}
%\displaystyle{\rho_l(\tau,y_i) = \frac{1}{r_{i+1}(\tau)-r_i(\tau)} = \frac{1}{r(\tau,y_{i+1})-r(\tau,y_i)} = \frac{1}{\delta} \, \frac{1}{\frac{r(\tau,y_{i}+\delta)-r(\tau,y_i)}{\delta}}} \approx \frac{1}{\delta} \frac{1}{\frac{\partial r(\tau, y_i)}{\partial y}}.
\displaystyle{\rho_l(\tau,y_i) = \frac{1}{r_{i+1}(\tau)-r_i(\tau)} } \approx \frac{1}{\delta} \frac{1}{\frac{\partial r(\tau, y_i)}{\partial y}}.
\eeq
}
\end{remark}
The change of coordinates $(\tau,y) \mapsto (t,r)$, with $t=\tau$ and $r=r(\tau,y)$ yields~\cite{Murray}
$$
\frac{\partial r}{\partial \tau} = - \frac{\partial r}{\partial y}\frac{\partial y}{\partial t} = - \frac 1 \delta \frac 1 \rho_l \frac{\partial y}{\partial t}.
$$
Substituting this relation along with the expressions 
$$
\begin{aligned} 
&\frac{\partial r}{\partial y} = \frac{1}{\delta}  \frac{1}{\rho_l}, \qquad \qquad  \frac{\partial^2 r}{\partial y^2} = \frac{1} \delta \, \frac{\partial}{\partial r} \left(\frac 1 \rho_l\right) \,  \frac{ \partial r}{\partial y} = - \frac 1 {\delta^2}  \frac 1 {\rho_l^3} \frac{\partial \rho_l}{\partial r},
\\
& \left(\int_{y_1}^{y} g\left(\frac{ \partial r}{\partial y'} \delta \, \right) dy'\right) \frac{ \partial r}{\partial y} \ = \left(\int_{s_0}^{r} g\left(1/\rho_A \right) \, \rho_A \, dr'\right) \frac{1}{\rho_A}
\end{aligned}
$$
into equations \eqref{eq:r_A} and \eqref{eq:r_B} yields
\begin{equation}\label{eq:r_Ainter}
\frac 1 \delta \frac{1}{\rho_A} \frac{\partial y}{\partial t} = - \frac{ F_A^\prime(1/\rho_A) }{\eta_A \, \rho_A^3} \, \frac{\partial \rho_A}{\partial r} +  \, \int_{s_0}^{r} g\left(1/\rho_A \right) \, \rho_A \, dr' \frac{1}{\rho_A}
\end{equation}
and
\begin{equation}\label{eq:r_Binter}
\frac 1 \delta \frac{1}{\rho_B} \frac{\partial y}{\partial t} = - \frac{ F_B^\prime(1/\rho_B) }{\eta_B \, \rho_B^3} \, \frac{\partial \rho_B}{\partial r} +  \, \int_{s_0}^{s_1} g\left(1/\rho_A \right) \, \rho_A \, dr' \frac{1}{\rho_B}.
\end{equation}
Multiplying equations \eqref{eq:r_Ainter} and \eqref{eq:r_Binter} by $\rho_A$ and $\rho_B$, respectively, differentiating with respect to $r$, using the fact that
$$
\frac{\partial}{\partial r} \left(\frac 1 \delta   \frac{\partial y}{\partial t}\right) = \frac{\partial}{\partial t} \left(\frac 1 \delta   \frac{\partial y}{\partial r}\right) = \frac{\partial}{\partial t} \left(\frac{\partial r}{\partial y} \, \delta \right)^{-1} =  \frac{\partial \rho_l}{\partial t},
$$
$$
 \, \frac{\partial}{\partial r} \int_{s_0}^{r} g\left(1/\rho_A \right) \, \rho_A \, dr' = G\left(\rho_A \right) \rho_A
\quad \text{ and  }
\quad 
 \, \frac{\partial}{\partial r} \int_{s_0}^{s_1} g\left(1/\rho_A \right) \, \rho_A \, dr' = 0,
$$
with the growth rate $G$ of population $A$  defined as
\beq
\label{eq:defG}
G\left(\rho_A \right) =  \, g(1/\rho_A),
\eeq
and renaming $r$ to $x$, we obtain the following  equations for the cell densities $\rho_A(t,x)$ and $\rho_B(t,x)$
\begin{eqnarray}
 \partial_t \rho_A &=& \partial_x (D_A(\rho_A) \,  \partial_x \rho_A) + G(\rho_A) \rho_A \quad    \text{ for } x \in  (s_0, s_1(t)),  \;\;\;\;\; t>0,  \label{eq:contin_110}
\\
 \partial_t \rho_B &=& \partial_x (D_B(\rho_B) \, \partial_x \rho_B) \quad   \hspace{2 cm }  \text{ for } x \in (s_1(t), s_2(t)),  \; t>0,   \label{eq:contin_11}
\end{eqnarray}
where 
\begin{equation}\label{eq:Dl}
D_l(\rho_l) = -\frac{ F_l^\prime(1/\rho_l) }{\eta_l \, \rho_l^2} \; \text{ for } \; l = A, B.
\end{equation}
Similarly, the evolution equations for the positions of the free boundaries $s_1(t)$ and  $s_2(t)$ are obtained from equations \eqref{boundary_c_1} and \eqref{boundary_c_2}, respectively, yielding
\begin{equation}\label{boundary_c_31}
\begin{aligned}
\frac{ds_1}{dt} &= \frac 1 {\eta_A} F_A\Big(\frac 1{\rho_A}\Big) -\frac 1 {\eta_B}  F_B\Big(\frac 1{\rho_B} \Big) - \frac 12 \Big[\frac {D_A(\rho_A)}{\rho_A} \, \partial_x \rho_A  + \frac{D_B(\rho_B)}{\rho_B} \, \partial_x \rho_B \Big]  &&\text{at } \; x = s_1(t),  \\
\frac{ds_2}{dt} & =   \frac 1 {\eta_B} F_B\Big(\frac 1{\rho_B}\Big) - \frac 12 \frac {D_B(\rho_B)}{\rho_B} \,\partial_x \rho_B\;  && \text{at } \; x = s_2(t).
\end{aligned}
\end{equation}
In order to obtain the boundary conditions (\emph{i.e.} the conditions at $s_0$ and $s_2(t)$) and the transmission conditions (\emph{i.e.} the conditions at $s_1(t)$), that are needed to complete the problem, we consider the mass balance equations
$$
\begin{aligned} 
& \int_{s_0}^{s_1(t)} G(\rho_A) \rho_A \,dx  = \frac{d}{dt} \Big(\int_{s_0}^{s_1(t)} \rho_A \, dx + \int_{s_1(t)}^{s_2(t)} \rho_B \,dx \Big), 
\\
& \int_{s_0}^{s_1(t)} G(\rho_A) \rho_A \, dx = \frac{d}{dt} \int_{s_0}^{s_1(t)} \rho_A \, dx. 
\end{aligned} 
$$
Using the fact that $\displaystyle{\frac{d s_0}{dt}=0}$, together  with equations \eqref{eq:contin_110} and \eqref{eq:contin_11}, yields 
\begin{eqnarray*} 
\begin{aligned} 
\int_{s_0}^{s_1(t)} \hspace{-0.2 cm }  G(\rho_A) \rho_A dx  &= \int_{s_0}^{s_1(t)} \hspace{-0.1 cm }  \partial_t \rho_A \, dx + \rho_A(t,s_1) \frac{d s_1}{dt} - \rho_A(t,s_0) \frac{d s_0}{dt} \\
& \quad + \int_{s_1(t)}^{s_2(t)} \hspace{-0.1 cm }  \partial_t \rho_B  dx + \rho_B(t,s_2) \frac{d s_2}{dt} - \rho_B(t,s_1) \frac{d s_1}{dt} \\
&= \int_{s_0}^{s_1(t)} \hspace{-0.2 cm } G(\rho_A) \rho_A  dx + D_A(\rho_A)  \partial_x \rho_A \big|_{x=s_1} -  D_A(\rho_A)  \partial_x \rho_A \big|_{x=s_0} +  \rho_A(t,s_1) \frac{d s_1}{dt}  \\
&\quad +  D_B(\rho_B) \, \partial_x \rho_B \big|_{x=s_2} -  D_B(\rho_B)  \partial_x \rho_B \big|_{x=s_1} + \rho_B(t,s_2) \frac{d s_2}{dt} - \rho_B(t,s_1) \frac{d s_1}{dt}
\end{aligned} 
\end{eqnarray*}
and 
\begin{eqnarray*} 
\begin{aligned} 
\int_{s_0}^{s_1(t)} \hspace{-0.2 cm } G(\rho_A) \rho_A  dx &=  \int_{s_0}^{s_1(t)} \partial_t \rho_A  dx + \rho_A(t,s_1) \frac{d s_1}{dt} - \rho_A(t,s_0) \frac{d s_0}{dt}  = \rho_A(t,s_1) \frac{d s_1}{dt} \\
& \quad + \int_{s_0}^{s_1(t)} \hspace{-0.1 cm } G(\rho_A) \rho_A  dx + D_A(\rho_A) \, \partial_x \rho_A \big|_{x=s_1} - D_A(\rho_A) \, \partial_x \rho_A \big|_{x=s_0}.
\end{aligned} 
\end{eqnarray*}
Hence
\begin{eqnarray*} 
 0 &=& \frac{d s_1}{dt} \left(\rho_A(t,s_1) - \rho_B(t,s_1) \right) + \left(D_A(\rho_A) \, \partial_x \rho_A - D_B(\rho_B) \, \partial_x \rho_B \right)\big|_{x=s_1}   \\
 &&  + \frac{d s_2}{dt} \rho_B(t,s_2) +  D_B(\rho_B) \, \partial_x \rho_B \big|_{x=s_2}  - D_A(\rho_A) \, \partial_x \rho_A \big|_{x=s_0}
\end{eqnarray*}
and 
\begin{eqnarray*} 
0 &=&  \frac{d s_1}{dt} \rho_A(t,s_1) +  D_A(\rho_A) \, \partial_x \rho_A \big|_{x=s_1} - D_A(\rho_A) \, \partial_x \rho_A \big|_{x=s_0}.
\end{eqnarray*}
The above equations along with  equations \eqref{boundary_c_31} give
\begin{equation}\label{eq:two_popul_model12}
\begin{aligned}
&  D_A(\rho_A) \partial_x \rho_A  =0  && \text{ at } x= s_0, \\
& \frac 1 { \eta_A} F_A\Big(\frac 1{\rho_A}\Big) =   \frac 1 { \eta_B} F_B \Big(\frac 1{\rho_B}\Big)  \quad  && \text{ at }x=s_1(t), \\
& \frac{ds_1}{dt}  \rho_A  =- D_A(\rho_A) \partial_x \rho_A  && \text{ at } x=s_1(t),  \\
& \frac{ds_1}{dt} ( \rho_A - \rho_B) =-\left( D_A(\rho_A) \partial_x \rho_A - D_B(\rho_B) \partial_x \rho_B\right)  && \text{ at } x=s_1(t),  \\
&\frac{ds_2}{dt}  = \frac 1 {\eta_B} F_B\left(\frac 1{\rho_B}\right)  - \frac 12 \frac 1{ \rho_B} D_B(\rho_B) \partial_x \rho_B \quad  && \text{ at }   x=s_2(t), \\
& \frac{ds_2}{dt} \rho_B =   -D_B(\rho_B) \partial_x \rho_B   && \text{ at } x=s_2(t).
\end{aligned}
\end{equation}
%\begin{remark}
%Notice that amongst the boundary and transmission conditions given by equation \eqref{eq:two_popul_model12} there are two additional conditions at $x=s_1(t)$ and $x=s_2(t)$, which govern the 
%\end{remark}
%
We complement the free-boundary problem  \eqref{eq:contin_110}, \eqref{eq:contin_11} and \eqref{eq:two_popul_model12} with the following initial conditions for the moving boundaries $s_{1}(t)$ and $s_{2}(t)$, and the cell densities $\rho_{A}(t,x)$ and $\rho_{B}(t,x)$:
\begin{equation}\label{initial_c}
\begin{aligned} 
& s_1(0) = s_1^\ast, \quad  &&\;  s_2(0) = s_2^\ast, \\
& \rho_A(0,x) = \rho_A^0(x)  \quad && \text{ for } \; x\in (s_0, s_1^\ast), \quad\\
&  \rho_B(0,x) = \rho_B^0(x)  \quad && \text{ for } \; x\in (s_1^\ast, s_2^\ast), 
\end{aligned} 
\end{equation} 
where $s_0 < s_1^\ast< s_2^\ast$.  Letting $\rho^{{\rm eq}}_l>0$ denote the equilibrium cell density (\emph{i.e.} the density below which intercellular forces are zero) and $\rho^{\rm M}> \rho^{{\rm eq}}_A$ a critical cell density above which cells stop dividing due to contact inhibition,  throughout the rest of the paper we will make the following assumptions:
\beq
\label{ass:initial_c}
\rho_{A}^0(x) \geq \rho^{{\rm eq}}_A \; \text{ for all } x \in (s_0, s_1^\ast), \quad \rho_{B}^0(x) \geq \rho^{{\rm eq}}_B  \; \text{ for all } x \in (s_1^\ast, s_2^\ast),
\eeq
\beq
\label{e.FBo}
F_l(1/\rho_l) = 0, \; F_l^\prime(1/\rho_l) = 0 \; \text{ for } \; \rho_l \leq  \rho^{\rm eq}_l, \quad  F_l(1/\rho_l) > 0, \; F^\prime_l(1/\rho_l) < 0 \; \text{ for } \; \rho_l >  \rho^{\rm eq}_l, 
\eeq
where  $ l=A,B$,   and
\beq
\label{e.gBo}
 g(\cdot) > 0   \; \;   \text{ in } \; (1/\rho^{\rm M}, \infty), \; \;   g(\cdot) = 0 \; \text{ in } \; (0, 1/  \rho^{\rm M}] , \; \;  g'(\cdot) < 0  \; \;  \text{ in } \; [1/ \rho^{\rm M}, \infty).
\eeq
Assumptions \eqref{e.FBo} and \eqref{e.gBo}, together with notations \eqref{eq:defG}  and   \eqref{eq:Dl},  imply that the nonlinear diffusion coefficient $D_l(\rho_l)$ and the growth rate $G(\rho_A)$ are such that
\beq
\label{e.DFBo}
D_l(\rho_l) = 0 \; \text{ for } \; \rho_l \leq  \rho^{\rm eq}_l, \qquad D_l(\rho_l) > 0 \; \text{ for } \; \rho_l > \rho^{\rm eq}_l, \qquad l=A,B,
\eeq
and
%G'(\cdot) < 0, \quad G(\rho^{\rm M}) = 0 \; \text{ with } \; \rho^{\rm M}>\rho^{{\rm eq}}_A.
\beq
\label{e.GFBo}
 G(\cdot) > 0   \; \;   \text{ in } \; \; (0,  \rho^{\rm M}), \; \;  
G(\cdot) = 0 \; \;   \text{ in } \;   [\rho^{\rm M}, \infty) , \;   \;  G'(\cdot) < 0 \; \;  \text{ in  } \;  \; (0,  \rho^{\rm M}].  
\eeq

\section{An existence result for the free-boundary problem}  
\label{well-posedness}
Due to the specific structure of our boundary and transmission conditions, the existing well-posedness results for one-dimensional free-boundary problems, such as those presented in~\cite{DuLin2010, Evans_1, Evans_2, Visintin},  are not directly applicable to our problem. Therefore, in this section we prove an existence result for the free-boundary problem \eqref{eq:contin_110}, \eqref{eq:contin_11}, \eqref{eq:two_popul_model12} and \eqref{initial_c}. 

\begin{assumption} \label{assum_1}
We make the following assumptions on the force terms $F_A$ and $F_B$, the diffusion coefficients $D_A$ and $D_B$, the growth rate $G$, and the initial conditions $\rho_A^0$ and $\rho_B^0$.
\begin{itemize} 
\item[(i)] The force terms $F_l \in H^1(0, \infty) \cap C^3(\rho_l^{\rm eq}, \infty)$, with $l=A,B$,  and satisfy~\eqref{e.FBo}.
%with  $F_l(1/\xi) =0$ for $0< \xi \leq \rho_l^{\rm eq}$,   $F_l$ is  monotone for $\xi \geq \rho_l^{\rm eq}$ with $F'_l(1/\xi) < 0$  for $\xi >  \rho^{\rm eq}_l$, and $F_l^\prime(1/\xi) = 0$  for  $0<\xi \leq  \rho^{\rm eq}_l$,    where $l=A,B$.
\item[(ii)] The diffusion coefficients $D_l \in C^2(\rho_l^{\rm eq}, \infty)\cap L^\infty(0, R)$, for any $R>0$, with $l=A,B$,  are defined by~\eqref{eq:Dl}, satisfy  assumptions~\eqref{e.DFBo},  and $D_l(\xi) \ge d_l >0$ for $\xi >\rho_l^{\rm eq}$.
%Diffusion coefficients  $D_l(\rho_l) = -  F^\prime_l(1/\rho_l)/(\eta_l \rho_l^2) $ satisfy  $D_l \in C^2(\rho_l^{\rm eq}, \infty)$,  with  $D_l(\xi) \geq d_l>0$ for $\xi > \rho_l^{\rm eq}$  and  $D_l(\xi) = 0$ for  $0<\xi \leq  \rho_l^{\rm eq}$,  where $l=A,B$ and $\rho_A^{\rm eq} >0$, $\rho_B^{\rm eq} >0$ are given equilibrium densities.  
\item[(iii)]  The growth rate $G\in C^3(\mathbb R)$  and satisfies  assumptions~\eqref{e.GFBo}.
%satisfies $G(\xi)>0$ for $\xi >0$,  $G(0) = 0$,  $G(\rho^{\rm M}) =0$ for some maximal density $\rho^{\rm M}_A > \rho_A^{\rm eq}$,  and  $G(\xi) < 0$ for  $\xi> \rho^{\rm M}_A$ and $\xi< 0$.
\item[(iv)] The initial conditions $\rho_A^0, \rho_B^0 \in C^3_0(\mathbb R)$ and satisfy  assumptions~\eqref{ass:initial_c}.
%$\rho_{A}^0(x) > \rho^{{\rm eq}}_A>0$ for all $x \in (s_0, s_1^\ast)$ while $\rho_{B}^0(x) > \rho^{{\rm eq}}_B>0$ for all $x \in (s_1^\ast, s_2^\ast)$ .
\end{itemize}  
\end{assumption} 

Throughout this section we use the notation 
$$
\Omega_A(t) = (s_0, s_1(t)) \quad \text{and} \quad \Omega_B(t) = (s_1(t), s_2(t)),
$$
 for $t \in [0,T]$, with $T>0$, and consider solutions in the sense specified by the following definition.

\begin{definition} \label{def:weak_sol1}
A solution of the free-boundary problem~\eqref{eq:contin_110}, \eqref{eq:contin_11}, \eqref{eq:two_popul_model12}, \eqref{initial_c} is given by functions  $s_1, s_2\in W^{1,3}(0,T)$ and $\rho_{l} \in H^1(0, T; L^2(\Omega_l(t)))\cap L^2(0, T; H^2(\Omega_l(t)))$, with  $\rho_l \geq \rho^{{\rm eq}}_l$ and $\rho_l \in L^\infty(0, T; L^\infty(\Omega_l(t)))$ for $l=A,B$, that satisfy equations \eqref{eq:contin_110} and \eqref{eq:contin_11}, the following boundary and transmission conditions 
\begin{equation}\label{transmis_1}
\begin{aligned}
&\partial_x  F_A\Big(\frac 1{\rho_A}\Big) = 0  && \text{ at }  x=s_0, \\
&\frac 1 {\eta_A} F_A\Big(\frac 1{\rho_A}\Big)   =  \frac 1 {\eta_B} F_B\Big(\frac 1{\rho_B} \Big) && \text{ at }  x=s_1(t), \\
&\frac {\partial_x F_A\big(1/{\rho_A}\big)}{\eta_A\rho_A}   =  \frac {\partial_x  F_B\big(1/{\rho_B} \big)}{\eta_B\rho_B} && \text{ at }  x=s_1(t), \\
& \frac {\partial_x  F_B\big(1/{\rho_B} \big)}{\eta_B\rho_B} = -\frac 2 {\eta_B} F_B\Big(\frac 1{\rho_B} \Big) && \text{ at }  x=s_2(t), 
\end{aligned} 
\end{equation}
the  equations for the free boundaries 
\begin{equation}\label{weak_sol_free_bp_22}
\begin{aligned}
& \frac{d s_1}{dt}  = -  \frac {\partial_x F_A\big(1/{\rho_A}\big)}{\eta_A\rho_A}&& \text{ at }  x=s_1(t), \\
 &  \frac{d s_2}{dt}  = \frac 2 {\eta_B} F_B\Big(\frac 1{\rho_B} \Big)  && \text{ at } x=s_2(t),
\end{aligned} 
\end{equation}
and the initial conditions~\eqref{initial_c}.   
\end{definition} 

%\begin{definition} \label{def:weak_sol2}
%A weak solution of  \eqref{eq:contin_110}, \eqref{eq:contin_11}, \eqref{eq:two_popul_model12}, \eqref{initial_c} are  functions $s_1, s_2\in W^{1,3}(0,T)$, 
%$\rho_l \in H^1(0, T; L^2(\Omega_1(t))\cap L^2(0, T; H^2(\Omega_l(t)))$,   with  $\rho_l \in L^\infty(0, T; L^\infty(\Omega_l(t)))$,   for $l=A,B$,  satisfying
%\begin{equation}\label{weak_sol_free_bp}
%\begin{aligned} 
%& \int_0^T \int_0^{s_1(t)} \big[\partial_t \rho_A \, \varphi  +  D_A(\rho_A)\partial_x \rho_A \, \partial_x \varphi  - G(\rho_A) \varphi \big]dx dt \\
%& + 
%\int_0^T \int_{s_1(t)}^{s_2(t)}  \big[\partial_t \rho_B \, \psi  +  D_B(\rho_B)\partial_x \rho_B\,  \partial_x \psi \big]dx dt 
%\\
%& = \int_0^T \frac{d s_1}{dt}\,  \big[\rho_B(t, s_1(t))\,  \psi(t, s_1(t)) -  \rho_A(t, s_1(t)) \, \varphi(t, s_1(t))\big] dt \\
%& - \int_0^T \frac{d s_2}{dt} \rho_B(t, s_2(t)) \, \psi(t, s_2(t)) dt, 
%\end{aligned} 
%\end{equation}
%for $\varphi \in L^\infty(0,T; H^1(s_0, s_1(t)))$, $\psi \in L^\infty(0,T; H^1(s_1(t), s_2(t)))$,

\begin{theorem}\label{th_existence}
Under Assumptions~\ref{assum_1}   there exists a  solution of the free-boundary problem~\eqref{eq:contin_110}, \eqref{eq:contin_11}, \eqref{eq:two_popul_model12} and \eqref{initial_c}.  
\end{theorem}
\begin{proof}   
In order to prove the existence of a solution of the free-boundary problem  \eqref{eq:contin_110}, \eqref{eq:contin_11}, \eqref{eq:two_popul_model12}  we consider iterations over successive time intervals and use a fixed point argument. In particular, we first show the existence of a solution on a time interval  $[0, T_1]$ such that 
\begin{equation}\label{cond_sj}
|s_1(t) - s_1^\ast| \leq \frac {s_1^\ast - s_0} 8 \quad \text{and} \quad |s_2(t) - s_2^\ast| \leq \frac {s_2^\ast - s_1^\ast} 8, \quad \text{ for } \; \; t \in [0,T_1]. 
\end{equation} 
Subsequently,  the boundedness of $ s_1^\prime$ and $s_2^\prime$,  shown at the end of the proof,  will allow iteration over successive time intervals in order to obtain an existence result for  $t \in (0,  T]$.  

We begin by making the change of variables 
\begin{equation}\label{transform_1}
(t,x) \mapsto (t,y), \quad \text{ with } \quad  x = y + \zeta(y)(s_1(t) - s_1^\ast) + \xi(y) (s_2(t) - s_2^\ast), 
\end{equation}
with $\zeta, \xi \in C^2_0(\mathbb R)$ such that $\zeta(y) = 1$ for $ |y- s_1^\ast| <  \alpha$ and $\zeta(y) = 0$ for $ |y- s_1^\ast| > 2  \alpha$, while 
$\xi(y) = 1$ for $ |y- s_2^\ast| <  \alpha$ and   $\xi(y) = 0$ for $ |y- s_2^\ast| > 2  \alpha$, where $\alpha =\min\{(s_1^\ast-s_0)/4,  (s_2^\ast -s_1^\ast)/4\}$. 
The change of variables~\eqref{transform_1} transforms the time-dependent domains $\Omega_A(t)=(s_0, s_1(t))$ and $\Omega_B(t)=(s_1(t), s_2(t))$ into the fixed intervals $\Omega_A^\ast = (s_0,s_1^\ast)$ and $\Omega_B^\ast= (s_1^\ast, s_2^\ast)$, respectively. A similar change of variables was considered in~\cite{DuLin2010}. Notice that, for $s_1(t)$ and $s_2(t)$ satisfying  conditions~\eqref{cond_sj},  such a change of variables defines a diffeomorphism from 
$[0, + \infty)$ into $[0, + \infty)$. Hence  we obtain 
$$
\rho_l (t,x) = \rho_l\big(t,  y + \zeta(y)(s_1(t) - s_1^\ast) + \xi(y) (s_2(t) - s_2^\ast)\big) = w_l(t,y) \; \; \; \text{ for } \; l = A,B, 
$$
where $w_A$ and $w_B$ satisfy the  reaction-diffusion-convection equations 
\begin{equation}\label{transformed_free_bp_11}
\begin{aligned} 
&J_A(s_1) \partial_t w_A -  \partial_y\left( R^2_A(s_1) D_A(w_A)  \partial_y  w_A\right) -  Q_A(s^\prime_{1}) \partial_y w_A - J_A(s_1) G_A(w_A) =0, \\
  & J_B(s_1, s_2) \partial_t w_B -  \partial_y\left(R^2_B(s_1, s_2)  D_B(w_B) \partial_y  w_B \right) - Q_B(s_1^\prime,  s_2^\prime) \partial_y w_B = 0, 
 \end{aligned} 
\end{equation} 
complemented with the  nonlinear transmission and boundary conditions 
\begin{equation}\label{transformed_free_bp_transm_12}
\begin{aligned} 
&\partial_y F_A(w_A) = 0  \qquad  && \text{ at } \; y= s_0, \\
& F_A(w_A) =    F_B (w_B)  \quad  && \text{ at } y=s_1^\ast , \\
&\frac{\partial_y F_A(w_A)}{ w_A}  = \frac{\partial_y F_B(w_B)}{w_B}  \quad  && \text{ at } y=s_1^\ast ,\\
&  \frac{\partial_y F_B(w_B)}{w_B}  = - 2    F_B({w_B} )   \qquad && \text{ at } \; y= s_2^\ast, 
\end{aligned} 
\end{equation}
and the equations for the velocities $s^\prime_{1}$ and $s^\prime_2$
\begin{equation}\label{eq_dsj_11}
\begin{aligned} 
& \frac{ds_{1}}{dt} w_A =-  \partial_y F_A({w_A})   && \text{ at } y= s_1^\ast, \\
& \frac{ds_{2}}{dt}  w_B  = -   \partial_y F_B({w_B})     && \text{ at } y= s_2^\ast. 
\end{aligned}  
\end{equation}
Notice that for ease of notation we denote
 $$
 F_l( w_l) \equiv \frac 1 \eta_l F_l(1/w_l)\;  \text{ and } \;  F_l^\prime( w_l) \equiv -\frac 1 \eta_l \frac{F_l^\prime(1/w_l)}{w_l^2},   \qquad l=A,B,
 $$
$G_A(w_A) = G(w_A) w_A$, and the functions $D_l(w_l)$ are defined in terms of $F_l(w_l)$ according to~\eqref{eq:Dl}.

In equations \eqref{transformed_free_bp_11},   since $\xi(y) = 0$ for $y< s_2^\ast - 2\alpha$  and $s_{1}(t) < s_2^\ast - 2\alpha $ for $t\in [0, T_1]$, we have 
$$
\begin{aligned} 
&R_A(s_{1}) = \frac{dy}{dx} =  \frac 1 {1+ \zeta^\prime(y)( s_{1}(t) - s_1^\ast)} \;  && \text{for }\;   s_0 < x < s_1^\ast,     \\
& R_B(s_{1}, s_{2}) = \frac{dy}{dx} =  \frac 1 {1+ \zeta^\prime(y)( s_{1}(t) - s_1^\ast) + \xi^\prime(y)( s_{2}(t) - s_2^\ast) } \; && \text{for }\;   s_1^\ast < x < s_2^\ast,   \\
& J_A(s_1) = 1+ \zeta^\prime(y)( s_{1}(t) - s_1^\ast) \; &&  \text{for }\;   s_0 < x < s_1^\ast, \\
& J_B(s_1, s_2)= 1+ \zeta^\prime(y)( s_{1}(t)  - s_1^\ast) + \xi^\prime(y)( s_{2}(t) - s_2^\ast)  \quad && \text{for }   s_1^\ast < x < s_2^\ast, \\
& Q_A(s^\prime_{1})  =  \zeta(y) \, s^\prime_{1}(t), \qquad Q_B(s^\prime_1, s^\prime_2) =  \zeta(y) s_{1}^\prime(t) + \xi(y) s^\prime_{2}(t), \\
&  \frac{dy}{dt} =  \frac{ Q_A(s^\prime_1) }{J_A(s_1)} \quad \;\;     \text{for }\;   s_0 < x < s_1^\ast , \qquad\quad
   \frac{dy}{dt} =   \frac{ Q_B(s^\prime_1, s^\prime_2) }{J_B(s_1, s_2)}   \quad && \text{for }   s_1^\ast < x < s_2^\ast.
\end{aligned} 
$$

The assumptions on $F_l$ and  $D_l$, for $l=A,B$, ensure that
\begin{equation}\label{eq_bound_equil}
\begin{aligned} 
&D_A(w_A)R_A( s_1)\partial_y w_A =0  \quad && \text{ for }  0< w_A \leq \rho_A^{\rm eq}, \\
&D_B(w_B)R_B( s_1,  s_2)\partial_y w_B =0  \quad  && \text{ for } 0<  w_B \leq \rho_B^{\rm eq}. 
\end{aligned} 
\end{equation}

Notice that, without loss of generality, we can focus on the case where $\rho_l^0 >  \rho_l^{\rm eq}$ for $l=A,B$. In fact, if $\rho_l^0 =  \rho_l^{\rm eq}$  the growth term in the equation for $w_A$ would  result into $w_A(t,y)>\rho_A^{\rm eq}$ and $F_A(w_A)>0$,  thus  ensuring that  $w_B(t,y)>\rho_B^{\rm eq}$ due to the transmission conditions at $s_1^\ast$ and the convection term in the equation for $w_B$. 
  In the case where $\rho_l^0 >  \rho_l^{\rm eq}$ for $l = A,B$, using the maximum principle and relations \eqref{eq_bound_equil} we obtain that $w_l(t,x) > \rho_l^{\rm eq}$ for $(t,x) \in (0,T)\times \Omega_{l}^\ast$. Therefore, we conclude that   system  \eqref{transformed_free_bp_transm_12},  \eqref{eq_dsj_11}, or  equivalently  system  \eqref{eq:contin_110},  \eqref{eq:contin_11},  is nondegenerate. 
Notice also that assuming   $s_j(t)\in H^2(0,T)$ with $s_j^\prime(t) \geq 0$ for $t\in [0,T]$ and $j=1,2$, and considering $F_l(w_l)$, $\partial_y^2 \partial_t F_l(w_l)$ and $\partial_t^2 w_l$ as test functions in equations \eqref{transformed_free_bp_transm_12} and \eqref{eq_dsj_11}, one can prove that $w_l$ is continuous in $\overline \Omega_{l, T}^\ast$, while $\partial_t w_l$ and $\partial_y^2 F_l(w_l)$ are  continuous in $(0, T)\times \Omega_l^\ast$ for $l=A,B$, which is the regularity required to apply the maximum principle. A similar approach was previously used in the analysis of the porous medium equation~\cite{Vazquez}.    

The assumptions on $F_B$ imply that
$$
D_B(w_B) \partial_y w_B  = -2 w_B F_B({w_B}) \leq 0\; \;  \text{ at } \; y = s_2^\ast, \;  \; t\geq 0, 
$$
and, applying the maximum principle to the equation for $w_B$, we find that $w_B$ has a minimum at $s_2^\ast$ and a maximum at $s_1^\ast$. Hence, $\partial_y w_B(t,y) \leq 0$ at $y=s_1^\ast$ for $t> 0$ and, therefore, $\partial_y F_A(w_A) \leq 0$ at $y=s_1^\ast$.
Applying the comparison principle and using the fact that $F^\prime_A(w_A) =0$ for $0<w_A(t,y) \leq \rho_A^{\rm eq}$, along with the assumptions on $G$ and on the initial conditions, we obtain 
\begin{equation}\label{bound_wA}
\rho_A^{\rm eq} \leq w_A(t, y) \leq  \rho^{\rm M}_A \; \; \text{ in }\; \; [s_0, s_1^\ast], \; \ t\geq 0, 
\end{equation}
where $\rho^{\rm M}_A = \max \{ \rho^{\rm M} , \max\limits_{x \in [s_0, s_1^\ast]} \rho_A^0(x)\}$.
Moreover, applying the maximum principle to the equation for $w_B$ and using the assumptions on $F_B$ and on the initial data, along with  the boundedness of $w_A$ and the transmission conditions at $y=s_1^\ast$, yield   
 \begin{equation}\label{bound_wB}
 \rho_B^{\rm eq} \leq w_B(t,y) \leq \rho_B^{\rm M}  \; \; \text{ in }\; \; [s_1^\ast, s_2^\ast], \; \ t\geq 0, 
 \end{equation}
where  $\rho_B^{\rm M}  = \max\{F_B^{-1}(F_A(\rho^{\rm M}_A)), \max\limits_{x \in [s_1^\ast, s_2^\ast]}  \rho_B^0(x)\}$. Using these results along with the change of variables given by equation~\eqref{transform_1} we conclude that   
$$
\begin{aligned} 
&\rho_A^{\rm eq} \leq \rho_A(t,x) \leq \rho^{\rm M}_A \qquad  && \text{ for } x\in [s_0, s_1(t)], \; \; t\geq 0,  \\
& \rho_B^{\rm eq} \leq \rho_B(t,x)   \leq  \rho_B^{\rm M}  \qquad && \text{ for } x\in [s_1(t), s_2(t)], \; \; t\geq 0. 
\end{aligned}
$$
If $w_B$ is nonconstant in $(s_1^\ast, s_2^\ast)$ and $w_B(t, s_j^\ast)\neq  \rho_B^{\rm eq}$ for $j=1,2$ and $t\ge 0$, the maximum principle yields $\partial_y w_B(t, s_2^\ast)  < 0$ and $\partial_y w_B(t, s_1^\ast)  < 0$. This along with the assumptions on $F_B$ ensures the monotonicity of the free boundaries $\{x= s_1(t)\}$ and $\{x= s_2(t)\}$, \emph{i.e.}
$$ 
\begin{aligned} 
& \frac{ d s_2(t)}{dt} >0  \;   \;  \text{ if } \; w_B(t, s_2^\ast) > \rho_B^{\rm eq},  \qquad  \frac{ d s_2(t)}{dt}  = 0 \;  \;   \text{ if } \; w_B(t, s_2^\ast) = \rho_B^{\rm eq}, \;  \;\quad   t\geq 0,  
\\ 
& \frac{ d s_1(t)}{dt} >0  \;  \;   \text{ if } \; w_B(t, s_1^\ast) > \rho_B^{\rm eq},  \qquad  \frac{ d s_1(t)}{dt}  = 0 \;\;     \text{ if } \; w_B(t, s_1^\ast) = \rho_B^{\rm eq}, \;  \quad  t\geq 0.  
\end{aligned} 
$$
 
%%%%%%%%%%%%%%%%%%%%%
To prove the existence of a solution of  problem \eqref{transformed_free_bp_11}-\eqref{eq_dsj_11} we  use a fixed point argument.   Let 
\begin{equation}\label{eq_s_ast_1}
s_1^{\ast, 1} = - \frac 1{  \rho_A^0 \eta_A} \partial_x F_A\Big(\frac 1{\rho_A^0(s_1^\ast)}\Big), \qquad  
s_2^{\ast, 1} =  - \frac 1{\rho_B^0\eta_B}  \partial_x F_B\Big(\frac 1 {\rho_B^0(s_2^\ast)}\Big),
\end{equation}
which are both well-defined quantities due to the assumptions on $F_l$ and $\rho_l^0$,  for $l=A,B$. Moreover, consider 
$$
\begin{aligned} 
&\mathcal W_l  =  \big\{u \in L^6(0, T_1; W^{1,4}(\Omega_l^\ast))  \, : \,   \rho_l^{\rm eq} \leq u(t,x) \leq \rho^{\rm M}_l \; \text{ for } \;  (t,x) \in \Omega_{l, T_1}^\ast, 
\; \|\partial_t u\|_{L^2(\Omega_{l, T_1}^\ast)} \leq \mu \big\}, \\
&\mathcal W_s  =  \big\{ (s_1, s_2) \in W^{1,3}(0,T_1)^2: \;  \|s^\prime_j  -s_j^{\ast,1}\|_{L^{3}(0, T_1)}\leq 1, \; \text{ for } \;  j=1,2 \big\},   
\end{aligned} 
$$
 for $l=A,B$, some constant $\mu>0$, and $T_1 >0$.  Notice that for $(s_1,  s_2) \in \mathcal W_s$ we have 
 $$
\sup\limits_{(0,T_1)}| s_j(t) - s_j^\ast| \leq \int_0^{T_1}\Big |\frac {d s_j}{dt}  \Big| dt \leq T_1^{\frac 23 }  \| s^\prime_j(t)\|_{L^{3}(0, T_1)} \leq T_1^{\frac 23}(1+ \|s_j^{\ast, 1}\|_{L^3(0,T_1)}), \qquad j=1,2.
$$
Therefore, choosing 
$$
T_1=\min\big\{ \big((s^\ast_1-s_0)/ 8\big)^{\frac 3 2} \big(1+ \|s_1^{\ast, 1}\|_{L^3(0,T_1)})^{-\frac 3 2}, \big((s^\ast_2-s^\ast_1)/ 8\big)^{\frac 3 2} \big(1+ \|s_2^{\ast, 1}\|_{L^3(0,T_1)})^{-\frac 3 2} \big\}  ,
$$
we find that  $s_j$ satisfies the conditions~\eqref{cond_sj}, for $j=1,2$, and the change of coordinates~\eqref{transform_1} is well-defined for all $(s_1, s_2) \in \mathcal W_s$.

For some given  $(\tilde s_1, \tilde  s_2) \in \mathcal W_s$ and $\widetilde w_l \in \mathcal W_l$, with $l= A,B$,  we first consider the problem given by the following equations for $w_A$ and $w_B$ 
\begin{equation}\label{transformed_free_bp_1}
\begin{aligned} 
&  J_A(\tilde s_1)\partial_t w_A -   \partial_y\big( R^2_A(\tilde s_1)  \partial_y F_A(w_A)\big) -  Q_A(\tilde s^\prime_{1}) \partial_y w_A  && \hspace{-2 cm} =  J_A(\tilde s_1) G_A(\widetilde w_A)   \quad \text{in } \Omega_A^\ast , \, t>0,  \\
& J_B(\tilde s_1, \tilde s_2)\partial_t w_B -  \partial_y\big(R^2_B(\tilde s_1, \tilde s_2)  \partial_y F_B(w_B)\big) -  Q_B(\tilde s_1^\prime,  \tilde s_2^\prime) \partial_y w_B 
&& \hspace{-0.3 cm}= 0   \quad \quad \; \text{ in } \Omega_B^\ast, \, t>0,  \\
 &\partial_y F_A(w_A)  = 0   \qquad &&  \text{ at }\;  y=s_0 ,\, t>0, \\
&  \frac{ \partial_y F_A( w_A)}{\widetilde w_A}  =    \frac{ \partial_y F_B(w_B)}{\widetilde w_B},   \qquad     F_A(w_A)=  F_B(w_B) && \text{ at } \;  y = s_1^\ast, \, t>0, \\
 & \frac{\partial_y F_B(w_B)}{\widetilde w_B}  =  -  2  F_B\left(w_B\right) &&  \text{ at } y= s_2^\ast, \, t>0, \\
 & w_A(0) = \rho_A^0  \qquad  && \text{ in }  (s_0, s_1^\ast), \\
 & w_B(0) = \rho_B^0   && \text{ in }  (s_1^\ast, s_2^\ast).
 \end{aligned} 
\end{equation}
For  $(\tilde s_1, \tilde  s_2) \in \mathcal W_s$ and $\widetilde w_l \in L^2(0,T_1; H^1(\Omega_l^\ast))$, with $l=A,B$, applying the Rothe-Galerkin method and using the a priori estimates obtained by considering $F_l(w_l)/{\widetilde w_l}$ as  a test function in the equations for $w_l$, we obtain the existence of a weak solution $F_l(w_l) \in L^2(0, T_1; H^1(\Omega_l^\ast))$, with $\partial_t w_l \in L^2(0, T_1; H^{-1}(\Omega_l^\ast))$, of problem~\eqref{transformed_free_bp_1}.  Notice that for $\tilde s_1, \tilde  s_2 \in H^2(0,T_1)$, in the same way as below,  we can show that the solutions of \eqref{transformed_free_bp_1} satisfy the regularity properties required by the maximum principle, and obtain that the solutions are bounded and satisfy \eqref{bound_wA} and \eqref{bound_wB}, and  the equations in \eqref{transformed_free_bp_1} are nondegenerate.

To derive  a priori estimates for $\partial_t w_A$ and  $\partial_t w_B$, we consider $\phi = \partial_t F_A(w_A)/ \widetilde w_A$ and $\psi = \partial_t F_B(w_B)/ \widetilde w_B$ as test functions for the equations in problem~\eqref{transformed_free_bp_1}. In this way, we obtain 
\begin{equation}\label{estim_apriori_dt}
\begin{aligned} 
\sum_{l=A,B}
\Big\{\int_0^\tau \hspace{-0.1 cm } \int_{\Omega_l^\ast} \frac{J_l(\tilde s) D_l(w_l)}{\widetilde w_l} |\partial_t w_l|^2 dy dt
+\frac 12  \int_0^\tau \hspace{-0.1 cm } \frac { d} {dt}  \int_{\Omega_l^\ast}  \frac{R_l(\tilde s)^2}{\widetilde w_l} \,   |\partial_y  F_l(w_l)|^2 dy dt
\\ -\int_0^\tau  \hspace{-0.1 cm }\int_{\Omega_l^\ast}  \Big[ Q_l(\tilde s^\prime)\,   \frac{\partial_y F_l(w_l)}{\widetilde w_l} \partial_t w_l 
+  \frac{R_l(\tilde s)^2}{\widetilde w_l^2} \,   \partial_y  F_l(w_l)\partial_t F_l(w_l) \partial_y \widetilde w_l  \Big]dy dt
\\  
+   \int_0^\tau  \hspace{-0.1 cm } \int_{\Omega_l^\ast}  \Big[ \frac 12 \frac{R_l^2(\tilde s)}{\widetilde w_l^2 } \,   |\partial_y  F_l(w_l)|^2 \partial_t \widetilde w_l 
-  \frac{R_l(\tilde s) \partial_t R_l(\tilde s)}{\widetilde w_l} |\partial_y  F_l(w_l)|^2 \Big] dy dt \Big\}
\\  -  \int_0^\tau \hspace{-0.1 cm } \int_{\Omega_A^\ast}   J_A(\tilde s_1)  G_A(\widetilde w_A) \frac{\partial_t F_A(w_A)}{\widetilde w_A}  dy dt  
 + 2 \int_0^\tau \hspace{-0.1 cm }   F_B(w_B) \partial_t F_B(w_B)  \Big|_{y=s_2^\ast} dt
\\  =   \int_0^\tau  \Big[\frac{\partial_y F_A(w_A)}{\widetilde w_A} \partial_t F_A(w_A)  - \frac{\partial_y F_B(w_B) }{\widetilde w_B}   \partial_t F_B(w_B)  \Big]_{y=s_1^\ast}  dt, 
\end{aligned} 
\end{equation}
for $\tau \in (0, T_1]$, where $J_A(\tilde s) = J_A(\tilde s_1)$, $J_B(\tilde s) = J_B(\tilde s_1, \tilde s_2)$,  $R_A(\tilde s) = R_A(\tilde s_1)$, $R_B(\tilde s) = R_B(\tilde s_1, \tilde s_2)$, 
$Q_A(\tilde s^\prime) = Q_A(\tilde s_1^\prime)$, and $Q_B(\tilde s^\prime) = Q_B(\tilde s_1^\prime, \tilde s_2^\prime)$. 

The transmission conditions in problem \eqref{transformed_free_bp_1} ensure that the integral at $y = s_1^\ast$ is equal to zero, while for the integral at $y= s_2^\ast$ we have 
 $$
 2 \int_0^\tau  F_B(w_B) \partial_t F_B(w_B)  \Big|_{y=s_2^\ast} dt   = |F(w_B(\tau, s_2^\ast))|^2 -  |F(w_B(0, s_2^\ast))|^2. 
 $$
From the equation for $w_A$ in problem \eqref{transformed_free_bp_1} we obtain 
$$
\begin{aligned} 
\|R_A^2\partial_y^2 F_A(w_A) \|^2_{L^2(\Omega_{A,\tau}^\ast)} 
\leq \|2R_A \partial_y R_A \partial_y  F_A(w_A) \|^2_{L^2(\Omega_{A,\tau}^\ast)}  +  \|J_A(\tilde s_1)\partial_t w_A\|_{L^2(\Omega_{A,\tau}^\ast)}^2 \\ +  \|Q_A(\tilde s_1^\prime) \partial_y w_A\|_{L^2(\Omega_{A,\tau}^\ast)}^2  + 
 \|J_A(\tilde s_1)G_A(\widetilde w_A)\|^2_{L^2(\Omega_{A,\tau}^\ast)}.
\end{aligned} 
$$
Using the definition of $Q_A$ and H\"older inequality,  the third term on the right-hand side is estimated as 
$$
 \|Q_A(\tilde s_1^\prime) \partial_y w_A\|_{L^2(\Omega_{A,\tau}^\ast)}^2
 \le C_1 \int_0^\tau |\tilde s_1^\prime|^2\| \partial_y w_A\|_{L^2(\Omega_{A}^\ast)}^2 dt 
 \le C_\delta \int_0^\tau |\tilde s_1^\prime|^3 dt + \delta \int_0^\tau \| \partial_y w_A\|_{L^2(\Omega_{A}^\ast)}^6 dt, 
 $$
for any fixed $\delta >0$. The assumptions on $G_A$ and $F_A$, the boundedness of   $J_A(\tilde s_1)$,  $R_A$, and $\partial_y R_A$, and   the fact that 
$R_A^2(\tilde s_1) \ge 4/9$ and  $F_A^\prime(w_A) = D_A(w_A)  \ge  d_A  >0$ for $w_A> \rho_A^{\rm eq}$, imply 
$$
\begin{aligned} 
\|\partial_y^2 F_A(w_A) \|^2_{L^2(\Omega_{A,\tau}^\ast)} \leq C_1\Big[ \|\partial_t w_A\|_{L^2(\Omega_{A,\tau}^\ast)}^2 +  \|\partial_y F_A(w_A)\|_{L^2(\Omega_{A,\tau}^\ast)}^2 \Big] + 
C_\delta \|\tilde s_1^\prime\|^3_{L^3(0,\tau)} \\ + \delta \int_0^\tau \|\partial_y F_A(w_A)\|_{L^2(\Omega_{A}^\ast)}^6 dt + C_2\tau,   
\end{aligned} 
$$
for $\tau \in (0, T_1]$.  Notice that for  $s_1^\prime \in L^\infty(0,T_1)$ we would have the $L^2$-norm of $\partial_y F_A(w_A)$  on the right-hand side of the last inequality. A similar inequality for $\|\partial_y^2 F_B(w_B) \|^2_{L^2(\Omega_{B,\tau}^\ast)}$ follows from the equation for $w_B$ in problem~\eqref{transformed_free_bp_1}.  The Gagliardo-Nirenberg inequality gives  
\begin{equation}\label{estim_GN_L6}
 \|\partial_y F_l(w_l)\|_{L^2(\Omega_{l}^\ast)}^6 \leq C \big(\|\partial_y^2 F_l(w_l) \|^2_{L^2(\Omega_{l}^\ast)}\|F_l(w_l)\|_{L^\infty(\Omega_{l}^\ast)}^4 +   \|F_l(w_l)\|^6_{L^2(\Omega_{l}^\ast)} \big), 
 \end{equation}
and, using the fact that $w_l$ is uniformly bounded and choosing $\delta >0$ sufficiently small, we obtain    
\begin{equation}\label{estim_H2}
\begin{aligned}
\|\partial_y^2 F_l(w_l) \|^2_{L^2(\Omega_{l,\tau}^\ast)} \leq C_1 \|\partial_t w_l\|_{L^2(\Omega_{l,\tau}^\ast)}^2
 + C_2 \|\partial_y F_l(w_l)\|_{L^2(\Omega_{l,\tau}^\ast)}^2 + 
C_\delta \|\tilde s^\prime\|^3_{L^3(0,\tau)} + C_3\tau, 
\end{aligned} 
\end{equation}
for $\tau \in (0, T_1]$, where $|\tilde s^\prime| = |\tilde s_1^\prime|$ if $l=A$ and   $|\tilde s^\prime| = |\tilde s_1^\prime|+ |\tilde s_2^\prime|$ if $l=B$.  In a similar way, we also obtain the following pointwise in the time variable estimate 
\begin{equation}\label{estim_H2_point}
\begin{aligned}
\|\partial_y^2 F_l(w_l(t)) \|^2_{L^2(\Omega_{l}^\ast)} \leq C_1 \|\partial_t w_l(t)\|_{L^2(\Omega_{l}^\ast)}^2
 + C_2 \|\partial_y F_l(w_l(t))\|_{L^2(\Omega_{l}^\ast)}^2  + C_\delta |\tilde s^\prime(t)|^3 + C_3 
\end{aligned} 
\end{equation}
for a.e. $t \in [0, T_1]$. Additionally, using the Gagliardo-Nirenberg inequality we have
\begin{equation} \label{estim_GN_L4}
\begin{aligned}
 & \|\partial_y F_l(w_l)\|_{L^4(\Omega_{l}^\ast)}^4 \leq C\big( \|\partial_y^2 F_l(w_l) \|^2_{L^2(\Omega_{l}^\ast)}\|F_l(w_l)\|_{L^\infty(\Omega_{l}^\ast)}^2
 +  \|F_l(w_l)\|^4_{L^2(\Omega_{l}^\ast)} \big), \\
&  \|\partial_y F_l(w_l)\|_{L^4(\Omega_{l}^\ast)}^4 \leq C\big(\|\partial_y^2 F_l(w_l) \|_{L^2(\Omega_{l}^\ast)}\|\partial_y F_l(w_l)\|_{L^2(\Omega_{l}^\ast)}^3
 +  \|F_l(w_l)\|^4_{L^2(\Omega_{l}^\ast)} \big). 
 \end{aligned} 
  \end{equation} 
We shall estimate each term in equation \eqref{estim_apriori_dt} separately. Using  estimates  \eqref{estim_GN_L6} and  \eqref{estim_H2} yields  
$$
\begin{aligned} 
 \int_0^\tau  \int_{\Omega_l^\ast}  \Big| Q_l(\tilde s^\prime)\,   \frac{\partial_y F_l(w_l)}{\widetilde w_l} \partial_t w_l \Big| dy dt 
 \leq  \delta \int_0^\tau  \Big[  \| \partial_y F_l(w_l)\|^6_{L^2(\Omega_l^\ast)} + \| \partial_t w_l\|^2_{L^2(\Omega_l^\ast)} \Big] dt 
\\ + C_\delta \| \tilde s^\prime\|^3_{L^3(0, \tau)} 
\leq \delta \|\partial_t w_l\|_{L^2(\Omega_{l,\tau}^\ast)}^2
 + \delta  \|\partial_y F_l(w_l)\|_{L^2(\Omega_{l,\tau}^\ast)}^2  + 
C_\delta \|\tilde s^\prime\|^3_{L^3(0,\tau)} + C\tau.
\end{aligned} 
$$   
Notice  that assuming the boundedness of $\tilde s^\prime_j$, with $j=1,2$, we would have the $L^2(0,T; L^2(\Omega_l^\ast))$-norm instead of the $L^6(0,T; L^2(\Omega_l^\ast))$-norm of $\partial_y F_l(w_l)$ in the last inequality. Using the results in \eqref{estim_H2} and \eqref{estim_GN_L4}, along with the boundedness of $w_l$ and $\widetilde w_l$, we estimate the next term in \eqref{estim_apriori_dt} as  
$$
\begin{aligned} 
& \int_{\Omega_{l,\tau} ^\ast}  \Big|\frac{R_l^2(\tilde s)}{\widetilde w_l^2} \,   \partial_y  F_l(w_l)\partial_t F_l(w_l) \partial_y \widetilde w_l\Big|  dy dt
\leq   \delta\big[  \|\partial_y  F_l(w_l)\|^4_{L^4(\Omega_{l, \tau}^\ast)} + \|\partial_t F_l(w_l)\|^2_{L^2(\Omega_{l, \tau}^\ast)}\big] \\
& + C_\delta   \|\partial_y  \widetilde w_l\|^4_{L^4(\Omega_{l, \tau}^\ast)}
\leq  \delta \big[\|\partial_y  F_l(w_l)\|^2_{L^2(\Omega_{l, \tau}^\ast)} 
+ \|\partial_t w_l\|^2_{L^2(\Omega_{l, \tau}^\ast)} \big]+ C_1 \|\tilde s^\prime\|_{L^3(0,\tau)}^3  
 + C_\delta   \|\partial_y  \widetilde w_l\|^4_{L^4(\Omega_{l, \tau}^\ast)} .
\end{aligned} 
$$
For the fourth integral in  \eqref{estim_apriori_dt} we have 
$$
\begin{aligned} 
   \int_0^\tau \hspace{-0.1 cm }  \int_{\Omega_l^\ast}  \Big|\frac 12 \frac{R_l^2(\tilde s)}{\widetilde w_l^2 }    |\partial_y  F_l(w_l)|^2 \partial_t \widetilde w_l 
-  \frac{R_l(\tilde s) \partial_t R_l(\tilde s)}{\widetilde w_l} |\partial_y  F_l(w_l)|^2 \Big| dy dt  \hspace{2.5 cm }  
\\ \leq C_\delta \| \tilde s^\prime\|^3_{L^3(0,\tau)} 
 +    \int_0^\tau \hspace{-0.1 cm }  \Big[ C_1\| \partial_y  F_l(w_l)\|^2_{L^4(\Omega_l^\ast)}\|\partial_t \widetilde w_l\|_{L^2(\Omega_l^\ast)} + \delta  \|\partial_y  F_l(w_l)\|^3_{L^2(\Omega_l^\ast)} \Big] dt  
 \\  \leq C_\delta \| \tilde s^\prime\|^3_{L^3(0,\tau)} 
    +   C_\delta \int_0^\tau  \hspace{-0.1 cm } \| \partial_y  F_l(w_l)\|^{2}_{L^2(\Omega_l^\ast)} \|\partial_t \widetilde w_l\|^{\frac 43}_{L^2(\Omega_l^\ast)} dt    + C_2 \tau^{\frac 12 }  \|\partial_t \widetilde w_l\|_{L^2(\Omega_{l, \tau}^\ast)} 
   \\ + \delta \big[\|\partial_y  F_l(w_l)\|^2_{L^2(\Omega^\ast_{l, \tau})}  
  +  \|\partial_t  w_l\|^2_{L^2(\Omega_{l, \tau}^\ast)}\big] + C_3\tau, 
\end{aligned} 
$$
for $\tau \in (0, T_1]$ and any fixed $\delta >0$. Here we used the following estimate 
$$
\begin{aligned} 
\int_0^\tau \hspace{-0.2 cm }  \| \partial_y  F_l(w_l)\|^2_{L^4(\Omega_l^\ast)}\|\partial_t \widetilde w_l\|_{L^2(\Omega_l^\ast)} dt 
\leq C  \hspace{-0.1 cm }   \int_0^\tau  \hspace{-0.15 cm } [ \| \partial^2_y  F_l(w_l)\|^{\frac 12}_{L^2(\Omega_l^\ast)} \| \partial_y  F_l(w_l)\|^{\frac 32}_{L^2(\Omega_l^\ast)} +1]\|\partial_t \widetilde w_l\|_{L^2(\Omega_l^\ast)} dt
\\ 
\leq \delta  \| \partial^2_y  F_l(w_l)\|^2_{L^2(\Omega_{l, \tau}^\ast)}  + C_\delta\int_0^\tau  \hspace{-0.1 cm } \Big[ \| \partial_y  F_l(w_l)\|^{2}_{L^2(\Omega_l^\ast)} \|\partial_t \widetilde w_l\|^{\frac 43}_{L^2(\Omega_l^\ast)} + \|\partial_t \widetilde w_l\|_{L^2(\Omega_{l}^\ast)} \Big]dt,  
 \end{aligned} 
$$
along with estimate \eqref{estim_H2}. Using \eqref{estim_GN_L4} and the boundedness of $w_l$ we also obtain 
$$ 
\int_0^\tau \hspace{-0.2 cm }  \| \partial_y  F_l(w_l)\|^2_{L^4(\Omega_l^\ast)}\|\partial_t \widetilde w_l\|_{L^2(\Omega_l^\ast)} dt 
\leq \delta    \| \partial_y^2  F_l(w_l)\|^2_{L^2(\Omega_{l, \tau}^\ast)} + C_\delta \|\partial_t \widetilde w_l\|^2_{L^2(\Omega_{l,\tau}^\ast)} + C\tau. 
$$
The boundedness of $\widetilde w_A$, along with the assumptions on $G_A$,  implies 
$$
\begin{aligned} 
 \int_{\Omega_{A, \tau}^\ast}  \Big| J_A(\tilde s_1)  G_A(\widetilde w_A) \frac{\partial_t F_A(w_A)}{\widetilde w_A}\Big|  dy dt 
 \leq C_\delta\tau + \delta \| \partial_t  w_A\|^2_{L^2(\Omega_{A, \tau}^\ast)} , 
 \end{aligned} 
$$
for $\tau \in (0, T_1]$ and  any fixed $\delta >0$.
 
Thus  for $\partial_t \widetilde w_l \in L^2((0, T_1)\times \Omega_{l}^\ast)$ and $\partial_y \widetilde w_l \in L^6(0, T_1; L^4(\Omega_l^\ast))$, combining the estimates from above, choosing  $\delta>0$ sufficiently small,   and applying the Gronwall inequality  yields   
\begin{eqnarray}\label{estim_22}
\sum_{l=A,B}\Big[ \|\partial_t  w_l\|^2_{L^2(\Omega_{l, T_1}^\ast)} + \|\partial_y F_l( w_l) \|^2_{L^\infty(0, T_1; L^2(\Omega_{l}^\ast))}\Big]
    \leq C_1\big(1+\| \tilde s_1^\prime\|^3_{L^3(0, T_1)} + \| \tilde  s_2^\prime\|^3_{L^3(0, T_1)} \big)    \nonumber
 \\  + C_2 \sum_{l=A,B} \Big[ T_1^{\frac 13}   \|\partial_y \widetilde w_l\|_{L^6(0, T_1; L^4(\Omega_l^\ast))}^4 + T_1^{\frac 12 } 
 \|\partial_t \widetilde w_l\|_{L^2(\Omega_{l, T_1}^\ast)}
+    \exp{(T_1^{\frac 13} \|\partial_t \widetilde w_l\|^{\frac 43}_{L^2(\Omega_{l, T_1}^\ast)})} \Big].
\end{eqnarray} 
In a similar way we also obtain 
\begin{eqnarray}\label{estim_222}
\sum_{l=A,B}\Big[ \|\partial_t w_l\|^2_{L^2(\Omega_{l, T_1}^\ast)} + \|\partial_y F_l(w_l)\|^2_{L^\infty(0, T_1; L^2(\Omega_{l}^\ast))}\Big]
    \leq C_1\big(1+\| \tilde s_1^\prime\|^3_{L^3(0, T_1)} + \| \tilde  s_2^\prime\|^3_{L^3(0, T_1)} \big)    \nonumber
 \\  + C_2 \sum_{l=A,B} \Big[   \|\partial_y F_l(\widetilde w_l)\|_{L^4(\Omega_{l, T_1}^\ast)}^4 +   \|\partial_t \widetilde w_l\|^{2}_{L^2(\Omega_{l, T_1}^\ast)} \Big].
\end{eqnarray} 
Thus  using  \eqref{estim_22} and \eqref{estim_222},  along with   \eqref{estim_H2} and \eqref{estim_GN_L4}, and considering $T_1$ sufficiently small,  we find that the map  $K: \mathcal W_A\times \mathcal W_B \to \mathcal W_A\times \mathcal W_B$, where 
 $(w_A, w_B)=K(\widetilde w_A, \widetilde w_B)$ is defined as a solution of problem \eqref{transformed_free_bp_1} for a given $(\tilde s_1, \tilde  s_2) \in \mathcal W_s$,   is continuous. 
  
 Considering $w_l$ in equation \eqref{estim_apriori_dt} instead of $\widetilde w_l$ and using the boundedness of $w_l$ yield
$$
\begin{aligned} 
 & \int_{\Omega_{l, \tau}^\ast} \hspace{-0.2 cm}   \Big| Q_l(\tilde s^\prime)\,   \frac{\partial_y F_l(w_l)}{w_l} \partial_t w_l \Big| dy dt 
 \leq   \int_0^\tau \hspace{-0.1 cm} \delta \Big[   \| \partial_y F_l(w_l)\|^6_{L^2(\Omega_l^\ast)} +  \| \partial_t w_l\|^2_{L^2(\Omega_l^\ast)} \Big] dt  +  C_\delta \| \tilde s^\prime\|^3_{L^3(0, \tau)} \\
 & \hspace{ 5.2 cm }  \leq  \delta \Big[   \| \partial_y F_l(w_l)\|^2_{L^2(\Omega_{l,\tau}^\ast)} +  \| \partial_t w_l\|^2_{L^2(\Omega_{l,\tau}^\ast)} \Big]  +  C_\delta \| \tilde s^\prime\|^3_{L^3(0, \tau)}, \\
& \int_{\Omega_{A, \tau}^\ast}  \Big| J_A(\tilde s_1)  G_A(w_A) \frac{\partial_t F_A(w_A)}{w_A}\Big|  dy dt 
 \leq  C_\delta \tau + \delta \| \partial_t w_A\|^2_{L^2(\Omega_{A, \tau}^\ast)} ,  
\end{aligned} 
$$ 
and 
$$
\begin{aligned} 
&  \int_{\Omega_{l,\tau}^\ast}  \Big|\frac 12 \frac{R_l^2(\tilde s)}{w_l^2 }    |\partial_y  F_l(w_l)|^2 \partial_t  w_l 
+ \frac{R_l(\tilde s) \partial_t R_l(\tilde s)}{w_l} |\partial_y  F_l(w_l)|^2 \Big| dy dt
\leq C_\delta \| \tilde s^\prime\|^3_{L^3(0,\tau)}   \\ 
&\qquad \;  + \delta \big[ \|\partial_t w_l \|^2_{L^2(\Omega_{l,\tau}^\ast)}  +   \|\partial_y  F_l(w_l)\|^2_{L^2(\Omega_{l,\tau}^\ast)}+1 \big] 
 +  \int_0^\tau \| \partial_y  F_l(w_l)\|^2_{L^4(\Omega_{l}^\ast)} \|\partial_t w_l \|_{L^2(\Omega_{l}^\ast)}  dt 
\\  
&  \leq  
 C_\delta \| \tilde s^\prime\|^3_{L^3(0,\tau)} 
 + \delta \big[ \|\partial_t w_l \|^2_{L^2(\Omega_{l,\tau}^\ast)} +   \|\partial_y  F_l(w_l)\|^2_{L^2(\Omega_{l,\tau}^\ast)}+1\big] 
 +\widetilde C_\delta  \int_0^\tau \| \partial_y  F_l(w_l)\|^6_{L^2(\Omega_{l}^\ast)} dt, 
\end{aligned} 
$$
for $\tau \in (0, T_1]$.   
Choosing  $\delta>0$ sufficiently small, applying the Gronwall inequality, and considering $T_1$ such that  
$$
T_1 \leq \min\limits_{l=A,B} \frac{\eta_l}{8 C_\delta \rho_l^{\rm M}(\|\partial_y  F_l(1/\rho_l^0)\|^4_{L^2(\Omega^\ast_l)}+1)}, 
$$
we obtain the following estimates for $w_A$ and $w_B$ 
 \begin{equation}\label{estim_23}
\sum_{l=A,B}\Big[ \|\partial_t w_l\|^2_{L^2(\Omega_{l, T_1}^\ast)} + \|\partial_y F_l(w_l)\|^2_{L^\infty(0, T_1; L^2(\Omega_{l}^\ast))}\Big]
    \leq  C+  C_\delta \big[\| \tilde s_1^\prime\|^3_{L^3(0, T_1)}   +  \| \tilde  s_2^\prime\|^3_{L^3(0, T_1)}\big] . 
  \end{equation} 
  The estimate for  $\|\partial_t \partial_y w_l\|_{L^2(0,T_1; H^{-1}(\Omega_l^\ast))}$ in terms of
$ \|\partial_y F_l(\widetilde w_l)\|_{L^4(\Omega_{l, T_1}^\ast)}$ and  $\|\partial_t \widetilde w_l\|_{L^2(\Omega_{l, T_1}^\ast)}$
  follows  directly from  differentiating the equation for $w_l$ with respect to $y$  and using  the boundedness of  $\|\partial_y^2 w_l\|_{L^2(\Omega_{l, T_1}^\ast)}$ and  $\|\partial_y F_l( w_l) \|_{L^\infty(0, T_1; L^2(\Omega_{l}^\ast))}$,  which is ensured by  \eqref{estim_22} and \eqref{estim_H2}.
Using \eqref{estim_23}  and  \eqref{estim_H2} and differentiating  the equation for $w_l$ in  \eqref{transformed_free_bp_1} with respect to $y$, while considering $w_l$ instead of $\widetilde w_l$, gives 
\begin{equation}\label{estim_25}
 \sum_{l=A,B}\Big[ \|\partial_t\partial_y w_l\|^2_{L^2(0,T_1; H^{-1}(\Omega_{l}^\ast))} + \|\partial_y^2 w_l\|^2_{ L^2(\Omega_{l,T_1}^\ast)}\Big]
    \leq  C+  C_\delta \big[\| \tilde s_1^\prime\|^3_{L^3(0, T_1)}   +  \| \tilde  s_2^\prime\|^3_{L^3(0, T_1)}\big].
\end{equation} 
Thus for a sufficiently small $T_1$, or small initial data,  $(w_A, w_B) = K(w_A, w_B)$ is uniformly bounded in  $\mathcal W_A\times \mathcal W_B$ and $\partial_y w_l$  in $\mathcal V_l$, for $l=A,B$, where 
$$\mathcal V_l=\big\{ u\in L^2(0,T_1; H^1(\Omega_l^\ast))\cap L^\infty(0, T_1; L^2(\Omega_l^\ast)), \; \partial_t  u \in L^2(0,T_1; H^{-1}(\Omega_l^\ast))\big\}, \quad l=A,B.$$ 
  The Aubin-Lions lemma, along with the Sobolev embedding theorem, ensures that $\mathcal V_l$
is a compact  subset of $L^2(\Omega_{l, T_1}^\ast)$ and of $L^2(0, T_1; C(\overline\Omega_{l}^\ast))$, for $l= A,B$.  Using inequality \eqref{estim_GN_L4} we also obtain that the embedding $\mathcal V_l \subset L^6(0, T_1; L^{4}(\Omega_l^\ast))$ is compact.  Thus applying the Schauder fixed point theorem, see \emph{e.g.}\ \cite{Schauder}, gives that for a given pair $(s_1, s_2) \in \mathcal W_s$ there exists  a solution of problem  \eqref{transformed_free_bp_11} and \eqref{transformed_free_bp_transm_12}  for $t\in(0, T_1]$, with an appropriate choice of  $T_1>0$.

To complete the proof  we shall show that $\mathcal M:  L^{3}(0,T_1)^2 \to L^{3}(0,T_1)^2$ given by 
$$
\mathcal M (r_1, r_2) = \Big( -  \frac{\partial_y F_A(w_A( t, s_1^\ast))}{w_A}, -  \frac{\partial_y F_B(w_B( t, s_2^\ast))}{w_B}\Big), 
$$
where  $s_j(t) = s_j^\ast + \int_0^t r_j(\tau) d\tau$  for $j=1,2$, maps 
$\mathcal W^\prime_s= \{ (r_1, r_2)\in L^{3}(0,T_1)^2 :  \|r_j -  s_j^{\ast, 1} \|_{L^3(0,T_1)} \leq 1\}$ into itself and is precompact.  Considering $(r_1, r_2) \in \mathcal W^\prime_s$ we have
$$
\begin{aligned} 
 \int_0^{T_1} |\mathcal M (r_1, r_2)- (s_1^{\ast, 1}, s_2^{\ast, 1})|^3 dt \leq& \sum_{l=A, B} \delta \int_0^{T_1} \|\partial_y^2  F_l( w_l)\|^2_{L^2(\Omega_{l}^\ast)} dt
 \\ & + \sum_{l=A, B}  T_1 \big[C_\delta \sup_{(0,T)}\|\partial_y F_l( w_l) \|^6_{L^2(\Omega_l^\ast)}+C\big] \leq 1,  
\end{aligned} 
$$
for  an appropriate choice of $\delta>0$ and $T_1>0$.  
To show that $\mathcal M: L^{3}(0,T_1)^2 \to L^{3}(0,T_1)^2$ is precompact we consider two sequences $\{r_1^n\}$ and $\{r_2^n\}$ bounded in $L^3(0,T_1)$ and obtain 
\begin{equation}\label{estim_wn_1}
\begin{aligned} 
\|\partial_y F_l(w_l^n)\|_{L^\infty(0,T_1; L^2(\Omega_{l}^\ast))}+ \|\partial_y^2 F_l(w_l^n)\|_{L^2((0,T_1)\times\Omega_{l}^\ast)}\qquad\\ + \|\partial_t w_l^n\|_{L^2(\Omega_{l, T_1}^\ast)} + \|\partial_t \partial_y F_l(w_l^n)\|_{L^2(0,T_1; H^{-1}(\Omega_{l}^\ast))}   \leq C, 
\end{aligned} 
\end{equation}
for $l=A,B$, with a constant $C$  independent of $n$. Using the fact that the embedding $H^1(\Omega_l^\ast) \subset C (\overline \Omega_l^\ast)$ is compact  and applying the Aubin-Lions lemma we obtain  the strong convergence  $w_l^n \to w_l$ in $L^p(0,T_1; C(\overline \Omega_l^\ast))$, for any $1<p< \infty$, and   $ \partial_y F_l(w_l^n) \to  \partial_y F_l(w_l)$ in $L^2(0,T_1; C (\overline \Omega_l^\ast))$ as $n \to \infty$. 
This combined with the estimate \eqref{estim_wn_1} ensures that
\begin{eqnarray}\label{conveg}
&& \int_0^{T_1} \Big\|\frac {\partial_y F_l(w_l^n)} {w_l^n} - \frac {\partial_y F_l(w_l)} {w_l}\Big\|_{L^\infty(\Omega_l^\ast)}^3 dt 
\leq C_1  \int_0^{T_1} \big\|\partial_y F_l(w_l^n)- \partial_y F_l(w_l) \big\|_{L^\infty(\Omega_l^\ast)}^3 dt 
\nonumber \\ 
&& \hspace{5. cm } + C_2 \int_0^{T_1} \|\partial_y F_l(w_l)\|_{L^\infty(\Omega_l^\ast)}^3\|w_l^n- w_l \big\|_{L^\infty(\Omega_l^\ast)}^3 dt 
\\ 
&& \leq C_3
\Big[ \int_0^{T_1} \big\|\partial_y F_l(w_l^n)- \partial_y F_l(w_l) \big\|_{L^\infty(\Omega_l^\ast)}^2 dt\Big]^{\frac 12 } 
 + C_4 \Big[ \int_0^{T_1}\|w_l^n- w_l \big\|_{L^\infty(\Omega_l^\ast)}^{12} dt \Big]^{\frac 14} \to 0, \nonumber
\end{eqnarray} 
as $n \to \infty$, where we used the fact that
$$
\begin{aligned}
\int_0^{T_1} \|\partial_y F_l(w_l^n) - \partial_y F_l(w_l)\|_{L^\infty(\Omega_l^\ast)}^4 dt 
\leq  & \int_0^{T_1} \Big[\|\partial_y^2 F_l(w_l^n)\|_{L^2(\Omega_l^\ast)}^2  \|\partial_y F_l(w_l^n)\|_{L^2(\Omega_l^\ast)}^2 \\ 
& + \|\partial_y^2 F_l(w_l)\|_{L^2(\Omega_l^\ast)}^2\|\partial_y F_l(w_l)\|_{L^2(\Omega_l^\ast)}^2 \Big]dt 
\leq C. 
\end{aligned} 
$$
The  convergence in \eqref{conveg} implies the strong  convergence $\mathcal M(r^n_1, r^n_2) \to \mathcal M(r_1, r_2)$  in $L^3(0,T_1)^2$ as $n \to \infty$. Hence, we have proved the existence of a solution of  problem~\eqref{transformed_free_bp_11}-\eqref{eq_dsj_11} in $(0, \hat T)\times \Omega_l^\ast$ with 
$
\hat T =\min\{ T_1,  T_2\},
$
where 
$$
\begin{aligned} 
T_1& = \min \big\{ \big[ \big((s^\ast_1-s_0)/ 8\big)^{\frac 32} \big(1+ \|s_1^{\ast, 1}\|_{L^2(0,T_1)})^{-\frac 32}\big ] , \big[ \big((s^\ast_2-s_1^\ast)/ 8\big)^{\frac 32} \big(1+ \|s_2^{\ast, 1}\|_{L^2(0,T_1)})^{-\frac 32}\big ] \big \}, \\
 T_2& = \min\limits_{l=A,B} \frac{\eta_l}{ 8 C_\delta \rho_l^{\rm M} (\|\partial_x  F_l(1/\rho_l^0)\|^4_{L^2(\Omega^\ast_l)}+1)}.
 \end{aligned}
$$ 
Now we  show that $s^\prime_1(t)$ and $s^\prime_2(t)$  are uniformly bounded, which will allow us to iterate  over successive time intervals and obtain that $\hat T \leq T_2$.  The uniform boundedness of $\rho_B$,  the assumptions on $F_B$, and equations \eqref{weak_sol_free_bp_22} ensure that $s^\prime_2(t)$ is uniformly bounded.
To show the boundedness of  $s_1^\prime(t)$,  we  consider the original problem  \eqref{eq:contin_110}, \eqref{eq:contin_11}, \eqref{eq:two_popul_model12}, and \eqref{initial_c} and  apply the comparison principle to the following problem for   $v= \eta_A^{-1} F_A(1/\rho_A)$: 
\begin{equation}\label{prob_v}
\begin{aligned} 
&\partial_t v = \tilde D_A(v) \partial_x^2 v + \tilde D_A(v)\tilde G_A(v) && \text{ in } \; (s_0, s_1(t)),  \; && t>0,    \\
& \partial_x v(t, x) = 0 \;  && \text{ for}\;  x= s_0, \; \; && t >0,   \\
& v (t, x) =  \frac 1 {\eta_B}  F_B\Big(\frac 1{\rho_B}\Big)  \; && \text{ for} \;  x= s_1(t), \; \; &&  t>0, \\
& v(0, x) =\frac 1{ \eta_A} F_A\Big(\frac 1{\rho_A^0} \Big) && \text{ in } \; (s_0, s_1^\ast), 
\end{aligned} 
\end{equation}
where $\tilde G_A(v) = G_A([F_A^{-1}(\eta_A v)]^{-1})$ and $ \tilde D_A(v)   = D_A([F_A^{-1}(\eta_A v)]^{-1})$. 

Consider the interval  $(s_1(t) -\delta,  s_1(t))$,  with $t>0,$ and the function  
$$\omega(t,x) = \frac 1 {\eta_B} F_B\Big(\frac 1{\rho_B^{\rm M}}\Big) +  \frac 1 {\eta_A}  F_A\Big(\frac 1{\rho^{\rm M}_A}\Big) \Big[\frac 2 {\delta} (s_1(t)  - x) - \frac 1 {\delta^2}(s_1(t) - x)^2 \Big], $$ 
for some $\delta>0$. A similar idea was used in~\cite{DuLin2010}. Since $F_A$ and $F_B$ are monotonically decreasing functions, we obtain  
$$
\begin{aligned} 
& \omega(t, s_1(t)) = \frac 1 {\eta_B}  F_B\Big(\frac 1{\rho_B^{\rm M}}\Big) \geq  \frac 1 {\eta_B}  F_B\Big(\frac 1{\rho_B(t, s_1(t))}\Big) = v(t, s_1(t)), 
\\
& \omega(t, s_1(t)-\delta) = \frac 1 {\eta_B} F_B\Big(\frac 1{\rho_B^{\rm M}}\Big) + \frac 1 {\eta_A} F_A\Big(\frac 1{\rho_A^{\rm M}} \Big) \geq 
 \frac 1 {\eta_A} F_A\Big(\frac 1{\rho_A(t,s_1(t) - \delta)} \Big) = v(t, s_1(t) - \delta). 
\end{aligned} 
$$
For the derivatives of $\omega$, since $s^\prime_1(t) \geq 0$,  we have 
$$
\begin{aligned} 
& \partial_t \omega(t,x) = \frac 2 \delta \frac 1 {\eta_A}  F_A\Big(\frac 1{\rho_A^{\rm M}} \Big) s^\prime_1(t) \Big[1- \frac {s_1(t)-x} \delta\Big] \geq 0 \quad \text{ for } x \in [s_1(t) - \delta, s_1(t)], \; t \geq 0,  \\
& \partial_x \omega(t,x) =   \frac 2 \delta \frac 1 {\eta_A} F_A\Big(\frac 1{\rho_A^{\rm M}} \Big)\Big[ \frac {s_1(t) - x}\delta - 1\Big], \qquad 
 \partial_x^2 \omega(t,x) = - \frac 2{\delta^2}  \frac 1 {\eta_A} F_A\Big(\frac 1{\rho_A^{\rm M}} \Big). 
\end{aligned} 
$$
Using the assumptions on $G_A$, for $\delta>0$ sufficiently small we obtain 
$$
\begin{aligned} 
\partial_t (\omega - v) - \tilde D_A(v) \partial^2_x(\omega - v)
\geq  \tilde D_A(v) \Big[ \frac 2{\delta^2}  \frac 1 {\eta_A} F_A\Big(\frac 1{\rho_A^{\rm M}} \Big) - \tilde G_A(v)\Big] \geq 0. 
\end{aligned} 
$$
Since  $F_A$ is continuous  and $\eta^{-1}_A F_A(1/\rho_A^0)=  \eta^{-1}_B F_B(1/\rho_B^0)$ at $x = s_1^\ast$, there exists a sufficiently small  $\delta>0$ such that 
$$
\omega(0, x) \geq v(0, x)  \quad \text{ for } \; x\in [s_1^\ast - \delta, s_1^\ast].
$$
Then applying  the comparison principle for parabolic equations gives
$$
\omega(t,x) \geq v(t,x) \quad \text{ for }  t \in (0, T) \text{ and } x\in  [s_1(t) - \delta, s_1(t)].
$$
Hence we have 
$$
- \frac 2 \delta \frac 1 {\eta_A} F_A\Big(\frac 1{\rho_A^{\rm M}} \Big)= \partial_x \omega  \leq   \partial_x v = \frac 1 {\eta_A}  \partial_x F_A\Big(\frac 1{\rho_A} \Big) \leq  0  \quad \text{ at } x=s_1(t)
$$
and for some sufficiently small  fixed  $\delta >0$ 
$$
0\leq \frac {ds_1(t)}{dt} \leq \frac 2 {\delta \rho_A^{\rm eq}}  \frac 1 {\eta_A}  F_A\Big(\frac 1{\rho_A^{\rm M}}\Big)  \; \; \text{ in } \; \;  (0, T). 
$$
Therefore, provided that $\partial_x F_l(1/\rho_l)$ is uniformly bounded in $L^\infty(0, T; L^2(\Omega_l(t)))$, for $l=A,B$,  the uniform boundedness of $s_1^\prime$ and  $s_2^\prime$ allows us to iterate over successive time intervals and prove the existence of a global solution of the free-boundary problem  \eqref{eq:contin_110}, \eqref{eq:contin_11}, \eqref{eq:two_popul_model12} and \eqref{initial_c}.

Thus, as next we prove the uniform boundedness of $\|\partial_x F_l(1/\rho_l)\|_{L^\infty(0, T; L^2(\Omega_l(t)))}$, or equivalently of $\|\partial_y F_l(w_l)\|_{L^\infty(0, T; L^2(\Omega_l^\ast))}$, for $l=A,B$. First we show higher regularity of the solutions of  problem~\eqref{transformed_free_bp_11}-\eqref{eq_dsj_11} by differentiating the equations in   \eqref{transformed_free_bp_11} with respect to the time variable  and considering  $\phi = \partial_t^2 F_A(w_A) /  w_A$  and $\psi = \partial_t^2 F_B(w_B) /  w_B$  as   test functions, respectively, 
$$
\begin{aligned} 
\sum_{l=A,B}\Big\{  \int_{\Omega_{l, \tau}^\ast}  \hspace{-0.2 cm} \Big( \frac {J_l  D_l} {w_l}   |\partial_t^2  w_l|^2 + 
\Big[J_l D^\prime_l \frac{ |\partial_t w_l|^2}{w_l}  + \partial_t J_l D_l \frac{ \partial_t w_l}{ w_l} \Big] \partial_t^2 w_l 
+ \partial_t J_lD_l^\prime  \frac{(\partial_t w_l)^3} { w_l} \Big) dy dt \\
 +\frac 12  \int_0^\tau \hspace{-0.1 cm }  \frac { d} {dt}  \int_{\Omega_l^\ast}  \frac{R_l^2(s) }{w_l}    | \partial_y  \partial_t F_l(w_l)|^2 dy dt - \int_{\Omega_{l, \tau}^\ast}  \hspace{-0.1 cm}\partial_y (R_l \partial_t R_l \partial_y F_l(w_l)) \frac{\partial_t^2 F_l(w_l)}{w_l} dy dt   \\
+\int_{\Omega_{l, \tau}^\ast} \hspace{-0.1 cm} \Big[  |\partial_y \partial_t  F_l(w_l)|^2  \frac 12  \Big(  \frac{R_l^2}{w_l^2} \partial_t  w_l -  \frac{\partial_t R_l^2(s)}{ w_l} \Big)
- \frac{ R_l^2(s) }{w_l^2}  \partial_y\partial_t  F_l(w_l)  \partial_t^2 F_l(w_l) \partial_y  w_l \Big] dy dt
\\ 
 + \int_{\Omega_{l, \tau}^\ast}  \hspace{-0.2 cm} \partial_t \big(  Q_l(s^\prime)  \partial_y F_l(w_l)\big) \frac{ \partial_t^2 F_l(w_l)}{w_l}  dy  dt  \Big\}
   -  \int_{\Omega_{A, \tau}^\ast} \hspace{-0.3 cm} \partial_t G_A(w_A)   \frac{ \partial_t^2 F_A(w_A)}{w_A}  dy  dt 
\\
  = \int_0^\tau \hspace{-0.1 cm } \frac{\partial_t\partial_y F_A(w_A)}{w_A} \partial_t^2 F_A(w_A)\Big|_{y= s_1^\ast} dt 
  + \int_0^\tau  \frac{\partial_t\partial_y F_B(w_B)}{w_B} \partial_t^2 F_B(w_B) \Big|_{s_1^\ast}^{s_2^\ast} dt  .   
\end{aligned} 
$$
The second term in the equation above can be estimated as 
$$
\begin{aligned} 
 \int_{\Omega^\ast_{l,\tau}} |\partial_t w_l|^2 & |\partial_t^2 w_l| dy dt  \leq 
\delta \| \partial_t^2 w_l\|_{L^2(\Omega_{l, \tau}^\ast)}^2 + C_\delta \|\partial_t w_l\|_{L^4(\Omega_{l, \tau}^\ast)}^4 
\\ & \leq \delta \| \partial_t^2 w_l\|_{L^2(\Omega_{l, \tau}^\ast)}^2 + C_\delta \int_0^\tau\Big[ \|\partial_t \partial_y F_l(w_l)\|^2_{L^2(\Omega_{l}^\ast)} \|\partial_t w_l\|^2_{L^1(\Omega_{l}^\ast)} +  \|\partial_t w_l\|^4_{L^1(\Omega_{l}^\ast)}\Big]dt, 
\end{aligned} 
$$
for $\tau \in (0, T]$. For the third and fourth  terms we have  
$$
\begin{aligned} 
 \int_{\Omega_{l, \tau}^\ast}\big[ |\partial_t w_l| |s^\prime| |\partial_t^2 w_l| +  |  \partial_t w_l|^3  | s^\prime|  \big] dy dt \leq 
  \delta \Big[ \|\partial_t \partial_y w_l\|^2_{L^2(\Omega_{l,\tau}^\ast)} 
 +   \|\partial_t^2 w_l\|^2_{L^2(\Omega_{l,\tau}^\ast)} \Big]
 \\ + C_\delta \Big[ \| s^\prime\|^2_{L^\infty(0,\tau)}  \| \partial_t w_l\|^2_{L^2(\Omega_{l, \tau}^\ast)} +  \| \partial_t w_l\|^{\frac 5 2}_{L^2(\Omega_{l,\tau}^\ast)}  \| s^\prime\|^3_{L^3(0,\tau)} \Big]. 
\end{aligned} 
$$
Moreover,
$$
\begin{aligned} 
 \int_{\Omega_{l, \tau}^\ast}  \hspace{-0.1 cm} \Big( |\partial_y \partial_t  F_l(w_l)|^2 \Big| \frac{R_l^2(s)}{w_l^2} \partial_t  w_l -  \frac{\partial_t R^2_l(s)}{ w_l} \Big|  +  \frac{ R_l^2(s) }{w_l^2}  \big| \partial_y\partial_t  F_l(w_l)   \partial_y  w_l \big| |\partial_t^2 F_l(w_l)| \Big) dy dt  \quad  
\\
 \leq   C_\delta \int_0^\tau  \hspace{-0.1 cm} \|\partial_y \partial_t F_l(w_l)\|_{L^2(\Omega_l^\ast)}^{2} 
\big[\|\partial_t w_l\|^{\frac 4 3}_{L^2(\Omega_l^\ast)} 
+  \|  \partial_y  w_l\|_{L^2(\Omega_{l}^\ast)}^2\|\partial_t w_l\|^{\frac 2 3}_{L^2(\Omega_l^\ast)} +  \|  \partial_y  w_l\|_{L^2(\Omega_{l}^\ast)}^{\frac 83}\\ 
 + |s^\prime| +1 \big] dt + \delta  \|\partial_t^2 F_l(w_l) \|_{L^2(\Omega_{l, \tau}^\ast)}^2.
\end{aligned} 
$$
We estimate the next terms as
$$
\begin{aligned} 
 \int_{\Omega_{l, \tau}^\ast} \Big[\big|\partial_y (R_l \partial_t R_l \partial_y F_l(w_l))\big| + 
  \big|\partial_t(  Q_l(s^\prime)  \partial_y F_l(w_l))\big|\Big]
 \frac{|\partial_t^2 F_l(w_l)|}{w_l}  dy dt  
\leq \delta \|\partial_t^2 F_l(w_l)\|_{L^2(\Omega_{l,\tau}^\ast)}^2\\
 + C_\delta \int_0^\tau [|s^\prime|^2 + |s^{\prime \prime}|^2]  \| \partial_y F_l(w_l)\|^2_{L^2(\Omega_l^\ast)} +
|s^\prime|^2 \big( \| \partial_y^2 F_l(w_l)\|^2_{L^2(\Omega_l^\ast)} +  \|\partial_y\partial_t  F_l(w_l) \|_{L^2(\Omega_l^\ast)}^2 \big) dt. 
\end{aligned} 
$$
The reaction term is estimated by 
$$
\begin{aligned}  
\delta \|\partial_t^2 F_A(w_A)\|_{L^2(\Omega_{1, \tau}^\ast)}^2
 +  C_\delta \int_0^\tau \hspace{-0.1 cm } \Big[ |s_1^{\prime}|^2  + (1 +|s_1|^2) \|\partial_t w_A\|^2_{L^2(\Omega_A^\ast)} \Big] dt  . 
\end{aligned} 
$$
For the non-zero contributions from the boundary terms we have 
$$
\begin{aligned} 
\int_0^\tau \partial_y F_l(w_l)  \frac{ \partial_t w_l }{w_l^2} \partial_t^2 F_l(w_l)\Big|_{y=s_1^\ast} dt& =
\frac{\partial_y w_l(t, s_1^\ast)}{2 w_l^2(t, s_1^\ast)} |\partial_t F_l(w_l(t, s_1^\ast))|^2 \Big |_0^\tau 
 \qquad   \\
 -  \frac 12  \int_0^\tau &\Big[\frac{\partial_y\partial_t  w_l} {w_l^2}  - \frac{\partial_y w_l \partial_t w_l}{w_l^3} \Big]   |\partial_t F_l(w_l)|^2  \Big|_{y= s_1^\ast} dt  = J_1 +J_2
\end{aligned} 
$$
for $l=A, B$, where  
$$
\begin{aligned} 
|J_1|& \leq  \delta \big[  \|\partial_t \partial_y F_l(w_l(\tau))\|^2_{L^2(\Omega_l^\ast)}  
+  \|\partial_t^2 w_l\|^2_{L^2(\Omega_{l, \tau}^\ast)} \big] +  C_\delta \big[ \|\partial_y w_l\|^4_{L^\infty(0, \tau; L^2(\Omega_{l}^\ast))}  \|\partial_t w_l\|^6_{L^2(\Omega_{l, \tau}^\ast)}+1\big], 
\\
|J_2|
& \leq  \delta \| \partial_t^2 F_l(w_l) \|^2_{L^2(\Omega_{l, \tau}^\ast)} +  C_\delta  \|\partial_y w_l\|^2_{L^\infty(0, \tau; L^2(\Omega_{l}^\ast))}\big[    \|\partial_t w_l\|^2_{L^2(\Omega_{l, \tau}^\ast)}+  \|\partial_y w_l\|^2_{L^2(\Omega_{l, \tau}^\ast)}+ 1\big]\\
 & \quad +  C \int_0^\tau
    \|\partial_y\partial_t  F_l(w_l)\|^{2}_{L^2(\Omega_{l}^\ast)}\big[\|\partial_t  F_l(w_l)\|^{\frac 4 3}_{L^2(\Omega_{l}^\ast)} +\delta \|\partial_t  F_l(w_l)\|^{2}_{L^2(\Omega_{l}^\ast)}\big]dt , 
\end{aligned} 
$$
and 
$$
\begin{aligned} 
\int_0^\tau \frac{\partial}{\partial t} \frac{\partial_y F_B(w_B)}{w_B} \partial_t^2 F_B(w_B) \Big|_{y= s_2^\ast} dt 
 = -  |\partial_t F_B(w_B(\tau, s_2^\ast))|^2 + |\partial_t F_B(w_B(0, s_2^\ast))|^2 . 
\end{aligned} 
$$
Hence, applying the Gronwall inequality and using the estimates  \eqref{estim_23} and \eqref{estim_25} yields 
\begin{eqnarray}\label{estim_dtt}
&& \sum_{l=A,B}\|\partial_t^2 w_l \|_{L^2(\Omega_{l,\hat T}^\ast)} + \|\partial_y \partial_t F_l(w_l)\|_{L^\infty(0, \hat T; L^2(\Omega_l^\ast))} 
 \leq C_1  
\\ && + C_2\sum_{l=A,B} \big[ \|s_1^{\prime\prime}\|_{L^2(0, \hat T)} +   \|s_2^{\prime\prime}\|_{L^2(0, \hat T)}\big]\exp(C_3[1+ \|\partial_y w_l\|^{\frac 83}_{L^\infty(0, \hat T; L^2(\Omega_{l}^\ast))} + \|\partial_t w_l\|^2_{L^2(\Omega_{l,\hat T}^\ast)}]), 
\nonumber
\end{eqnarray}
with constants $C_j$, for $j=1,2,3$, depending on $\|s_1^\prime\|_{L^\infty(0, \hat T)}$ and $\|s_2^\prime\|_{L^\infty(0, \hat T)}$. 
 
Using the Gagliardo-Nirenberg inequality and boundedness of $\partial_y F_A(w_A(t, s_1^\ast))$ we obtain 
$$
\begin{aligned} 
\|s^{\prime\prime}_1 \|^2_{L^2(0,\tau)}\le C_1\Big[ \|\partial_t\partial_y F_A(w_A)\|^2_{L^2(0,\tau; L^\infty(\Omega_A^\ast))} + 
\|\partial_t w_A \|^2_{L^2(0,\tau; L^\infty(\Omega_A^\ast))} \|\partial_y F_A(w_A(\cdot, s_1^\ast))\|^2_{L^2(0,\tau)}\Big]
\\  \le C_2  \Big[\|\partial_t\partial_y F_A(w_A)\|_{L^2(\Omega_{A, \tau}^\ast)}   \|\partial_t\partial_y^2 F_A(w_A)\|_{L^2(\Omega_{A, \tau}^\ast)}  +  \|\partial_t\partial_y F_A(w_A)\|^2_{L^2(\Omega_{A, \tau}^\ast)}   +  \|\partial_t w_A\|^2_{L^2(\Omega_{A, \tau}^\ast)} \Big]\\ \le \delta  \|\partial_t\partial_y^2 F_A(w_A)\|^2_{L^2(\Omega_{A, \tau}^\ast)} + C_\delta  \|\partial_t\partial_y F_A(w_A)\|^2_{L^2(\Omega_{A, \tau}^\ast)} + C_3  \|\partial_t w_A\|^2_{L^2(\Omega_{A, \tau}^\ast)} .
\end{aligned} 
$$
Differentiating \eqref{transformed_free_bp_11} with respect to the time variable and using definition of $J_A$, $R_A$, and $Q_A$ yields
$$
\begin{aligned} 
\|\partial_t \partial_y^2 F_A(w_A)\|^2_{L^2(\Omega^\ast_{A, \tau})} \le C\Big[ \|\partial_t^2  w_A\|^2_{L^2(\Omega^\ast_{A, \tau})} + 
\|(|s_1|+s^\prime_1) \partial_t \partial_y F_A(w_A)\|^2_{L^2(\Omega^\ast_{A,\tau})} + \|s^\prime_1  w_A\|^2_{L^2(\Omega^\ast_{A,\tau})}  \\ 
+ \|(|s_1|+ s^\prime_1) \partial_t  w_A\|^2_{L^2(\Omega^\ast_{A,\tau})}
 +  \|s^\prime_1 (|\partial_y^2  F_A(w_A)| + |\partial_y F_A(w_A)|)\|^2_{L^2(\Omega^\ast_{A,\tau})}+   \|s^{\prime\prime}_1 \partial_y  w_A \|^2_{L^2(\Omega^\ast_{A,\tau})}\Big].
\end{aligned} 
$$
The last term is  estimated by
$$
\begin{aligned} 
& \|s^{\prime\prime}_1 \partial_y  w_A\|_{L^2(\Omega^\ast_{A,\tau})}\ \le \|\partial_y  w_A\|^2_{L^\infty(0,\tau; L^2(\Omega^\ast_{A}))}
 \|s^{\prime\prime}_1 \|^2_{L^2(0,\tau)} 
 \le \delta  \|\partial_t\partial_y^2 F_A(w_A)\|^2_{L^2(\Omega_{A,\tau}^\ast)} \\ 
& + C_\delta \big[ \|\partial_y  w_A\|^4_{L^\infty(0,\tau; L^2(\Omega^\ast_{A}))}
  +  \|\partial_y  w_A\|^2_{L^\infty(0,\tau; L^2(\Omega^\ast_{A}))} \big]
 \big[ \|\partial_t\partial_y F_A(w_A)\|^2_{L^2(\Omega_{A,\tau}^\ast)} +  \|\partial_t w_A\|^2_{L^2(\Omega_{A, \tau}^\ast)} \big]. 
\end{aligned} 
$$
Similar estimates, using the boundedness of $\partial_y F_B(w_B(t, s_2^\ast))$, hold for $\partial_t \partial_y^2 F_B(w_B)$ and $s^{\prime\prime}_2$. 
Thus choosing $\delta>0$ sufficiently small  and using the boundedness of $s^\prime_1$  and $s^{\prime}_2$ we obtain 
$$
\begin{aligned} 
&\|\partial_t \partial_y^2 F_l(w_l)\|^2_{L^2(\Omega^\ast_{l, \tau})} \le C_1\Big[ \|\partial_t^2  w_l\|^2_{L^2(\Omega^\ast_{l,\tau})} 
+     \|\partial_y  w_l\|^2_{L^\infty(0,\tau; L^2(\Omega^\ast_{l}))}\Big] + C_2 \tau
\\ & \; \; \; \; \;  + C_3(1+ \|\partial_y  w_l\|^4_{L^\infty(0,\tau; L^2(\Omega^\ast_{l}))})
 \big[ \|\partial_t\partial_y F_l(w_l)\|^2_{L^2(\Omega_{l,\tau}^\ast)} +  \|\partial_t  w_l\|^2_{L^2(\Omega^\ast_{l,\tau})} \big] 
  \end{aligned} 
$$
 and, considering the same calculations as in the derivation of \eqref{estim_dtt}, conclude 
 \begin{eqnarray}\label{estim_dtt_second}
 \sum_{l=A,B}\|\partial_t^2 w_l \|_{L^2(\Omega_{l, \hat T}^\ast)} + \|\partial_y \partial_t F_l(w_l)\|_{L^\infty(0, \hat T; L^2(\Omega_l^\ast))} 
 \leq C, 
\end{eqnarray}
where  the constant $C$ depends on $\|\partial_y  w_l\|_{L^\infty(0, \hat T; L^2(\Omega^\ast_{l}))}$,  
 $ \|\partial_t  w_l\|_{L^2(\Omega^\ast_{l, \hat T})}$ for $l=A,B$, and $\|s_j^\prime\|_{L^\infty(0, \hat T)}$  for $j=1,2$. 
 
 Hence estimates  \eqref{estim_25} and  \eqref{estim_dtt_second} imply that $\partial_y F_l(w_l) \in C(\overline \Omega_{l, \hat T}^\ast)$. Since $\partial_y F_A(w_A(t, s_1^\ast))$ and $\partial_y F_B(w_B(t, s_j^\ast))$, for $j=1,2$, are bounded, then  $\partial_y F_A(w_A)$ is  bounded in $[s_1^\ast-\delta, s_1^\ast]$   and  $\partial_y F_B(w_B)$ is  bounded in   $[s_1^\ast, s_1^\ast+\delta]$ and $[s_2^\ast-\delta, s_2^\ast]$, for $t\in [0, \hat T]$ and for a sufficiently small $\delta>0$. 
 Taking $F_l(w_l)$ as test function in \eqref{transformed_free_bp_11} and using boundedness of $\partial_y F_A(w_A(t, s_1^\ast))$ and $\partial_y F_B(w_B(t, s_j^\ast))$, for $j=1,2$, yield
 $$
  \sum_{l=A,B} \Big[\|w_l\|_{L^\infty(0, \hat T; L^2(\Omega^\ast_{l}))} + \|\partial_y F_l(w_l)\|_{L^2(\Omega^\ast_{l, \hat T})} \Big]
 \le C.
 $$ 
 Considering a cut-off function $\zeta_l\in C^2(\overline \Omega_l^\ast)$, such that $\zeta_A(y) = 1$ for $y \in [s_0, s_1^\ast -   \delta/2]$ and $\zeta_A(y) = 0$ for $y \in [s_1^\ast -   \delta/4, s_1^\ast]$ and 
 $\zeta_B(y) = 1$ for $y \in [s_1^\ast +   \delta/2, s_2^\ast - \delta/2]$ and $\zeta_B(y) = 0$ for $y \in [s_1^\ast,   s_1^\ast + \delta/4]$ and 
 $y \in [s_2^\ast- \delta/4,   s_2^\ast]$ 
  and taking $\partial_t F_l(w_l)\zeta_l^2$ as test function in \eqref{transformed_free_bp_11} gives 
 $$
\sum_{l=A,B} \Big[\|\partial_t  w_l\zeta_l\|_{L^2(\Omega^\ast_{l,  \hat T})} + \|\partial_y F_l(w_l)\zeta_l\|_{L^\infty(0, \hat T; L^2(\Omega^\ast_{l}))} \Big]
 \le C_1\big(1+\sum_{l=A,B} \|F_l(w_l)\|_{L^2(0, \hat T; H^1(\Omega^\ast_{l}))} \big)\le C_2. 
 $$
Therefore, $\|\partial_y  F_A(w_A)\|_{L^\infty(0, \hat T; L^2(\Omega^\ast_{A}))}$ and  $\|\partial_y  F_B(w_B)\|_{L^\infty(0, \hat T; L^2(\Omega^\ast_{B}))}$ are uniformly bounded and $\hat T$ can be chosen independently of the initial data.  Then iterating over $\hat T$, we can conclude that there exists a global solution of problem \eqref{eq:contin_110}, \eqref{eq:contin_11}, \eqref{eq:two_popul_model12}, \eqref{initial_c}. 

\end{proof}

%%%%%%%%%%%%%%%%%%%%

\section{Travelling-wave solutions of the free-boundary problem}
\label{travelling-waves}
In this section, we carry out a travelling-wave analysis for the free-boundary problem~\eqref{eq:contin_110}, \eqref{eq:contin_11} and \eqref{eq:two_popul_model12}. 

We begin by noting that under  assumptions \eqref{e.FBo}, \eqref{e.DFBo} and \eqref{e.GFBo},  the free-boundary problem \eqref{eq:contin_110}, \eqref{eq:contin_11} and~\eqref{eq:two_popul_model12} can be written in terms of the cell pressures  $P_A$ and $P_B$ defined according to the following barotropic relation
\beq
\label{e.barotreln}
P_l(\rho_l) = 0 \; \text{ for } \; \rho_l \leq \rho^{\rm eq}_l \quad \text{and} \quad \dfrac{d P_l(\rho_l)}{d\rho_l}  = \frac{D_l(\rho_l)}{\rho_l} \; \text{ for } \; \rho_l >  \rho^{\rm eq}_l, \qquad l=A,B. 
\eeq
Under assumptions  \eqref{e.DFBo}, the barotropic relation \eqref{e.barotreln} is such that
\beq
\label{e.TWASC1FBb}
\dfrac{d P_l(\rho_l)}{d\rho_l} >0 \; \mbox{ for } \; \rho_l > \rho^{\rm eq}_l, \qquad l = A,B.
\eeq
The monotonicity conditions  \eqref{e.TWASC1FBb} allow one to write both the force terms $F_l(1/\rho_l)$ and the growth rate $G(\rho_A)$ 
as functions of the cell pressure $\tilde F_A(P_A)$, $\tilde F_B(P_B)$ and $\tilde G(P_A)$. Moreover, with the notation
\beq
\label{def.PeqPM}
P^{\rm eq}_l=P_l(\rho^{\rm eq}_l) \geq 0 \quad \text{ and }  \quad P^{\rm M}=P_A(\rho^{\rm M}) > P^{\rm eq}_A,
\eeq
the  monotonicity conditions \eqref{e.TWASC1FBb} make it possible to reformulate assumptions \eqref{e.FBo} and \eqref{e.GFBo} on the functions $F_{l}$ and $G$ as
\beq
\label{e.FTW}
\tilde F_l(P_l) = 0  \;  \; \text{ for } \; P_l(\rho_l) \leq P^{\rm eq}_l, \quad \dfrac{d \tilde F(P_l)}{d P_l}  > 0 \; \text{ for } \; P_l > P^{\rm eq}_l, \qquad l=A,B
\eeq
and
% \tilde G'(\cdot) < 0, \quad \tilde G(P^{\rm M}) = 0 \; \text{ with } \; P^{\rm M} > P^{{\rm eq}}_A.
\beq
\label{e.GFB}
\tilde G(\cdot) > 0   \; \;   \text{ in } \; \; (0,  P^{\rm M}), \; \;  
\tilde G(\cdot) = 0 \; \;   \text{ in } \;   [P^{\rm M}, \infty) , \;   \;  \dfrac{d \tilde G(\cdot)}{dP_A} < 0 \; \;  \text{ in  } \;  \; (0,  P^{\rm M}].  
\eeq
Hence, under assumptions  \eqref{e.FBo}, \eqref{e.DFBo} and \eqref{e.GFBo}, using the barotropic relation  \eqref{e.barotreln}, we can rewrite the free-boundary problem  \eqref{eq:contin_110}, \eqref{eq:contin_11} and~\eqref{eq:two_popul_model12} in the following alternative form:
\begin{equation}\label{eq:continTW}
\begin{aligned}
&\partial_t \rho_A - \partial_x \big(\rho_A \, \partial_x P_A\big) = \tilde G(P_A) \rho_A \; &&\text{ for }  x \in (s_0,s_1(t)),  \quad\;\, t>0, \\
&\partial_t \rho_B - \partial_x (\rho_B \, \partial_x P_B) = 0 \; &&\text{ for } x \in (s_1(t),s_2(t)),  \; t>0, \\
&\partial_x P_A  =0 \; && \text{ at } \;  x= s_0, \\
&\frac{1}{\eta_A} \, \tilde F_A\left(P_A\right) = \frac{1}{\eta_B} \, \tilde F_B \left(P_B\right)  \; && \text{ at } \;  x= s_1(t),   \\
&\frac{ds_1}{dt}   =- \partial_x P_A \; && \text{ at } \; x=s_1(t), \\
&\frac{ds_1}{dt} ( \rho_A - \rho_B) =-\big( \rho_A \, \partial_x P_A  - \rho_B \, \partial_x P_B\big) \; && \text{ at } \; x=s_1(t), \\
&\frac{ds_2}{dt}  = \frac{1}{\eta_B} \, \tilde F_B\left(P_B\right)  - \frac 12 \partial_x P_B \; && \text{ at } \; x=s_2(t), \\
&\frac{ds_2}{dt} = - \partial_x P_B  \; && \text{ at } \; x=s_2(t).
\end{aligned}
\end{equation}
Having rewritten the problem in this form allows us to construct travelling-wave solutions using an approach that builds on the method of proof recently presented in~\cite{chaplain2019bridging, lorenzi2017interfaces}.  For the sake of brevity, in this section we  drop the tildes from all the quantities in  problem~\eqref{eq:continTW}.

We construct travelling-wave solutions of the  free-boundary problem~\eqref{eq:continTW} such that both the position of the inner free boundary, $s_1(t)$, and the position of the outer free boundary, $s_2(t)$, move at a given constant speed $c>0$.  Without loss of generality, we let $s_0$ go to $-\infty$ and consider the case where 
\beq
\label{eq:dts1dts2c0}
s_1(t) = (c + o(1))t  \; \text{ and } \; s_2(t) = \ell + (c + o(1))t,
\eeq
for some $\ell >0$, so that 
$$
s_1(0)=s^*_1=0, \quad s_2(0)=s^*_2=\ell  \quad \text{and} \quad \dfrac {ds_{1}}{dt}=\dfrac {ds_{2}}{dt} = c.
$$
We make the following travelling-wave ansatz for the cell densities $\rho_A$ and $\rho_B$
\beq
\label{e.TWansatz}
\rho_A(t,x) = \rho_A(z) \; \mbox{ and } \; \rho_B(t,x) = \rho_B(z) \; \text{ with } \; z=x - c \, t,
\eeq
which are related to the cell pressures $P_A(z)$ and $P_B(z)$ through the barotropic relation~\eqref{e.barotreln}. In this framework, substituting the travelling-wave ansatz~\eqref{e.TWansatz} into problem~\eqref{eq:continTW} we find 
\begin{equation}\label{e.TWsysFB}
\begin{aligned}
&- c \, \rho_A' = \left(\rho_A \, P_A' \right)' + G(P_A) \; \rho_A \; &&\text{ in }  -\infty < z < 0, \\
&- c \, \rho_B' = \left(\rho_B \, P_B' \right)' \; &&\text{ in } 0 < z <  \ell, \\
%& P'_A  =0 \; && \text{ at } \;  z=-\infty, \\
& \frac{1}{\eta_A} \, F_A(P_A) = \frac{1}{\eta_B} \, F_B(P_B)  \; && \text{ at } \; z=0, \\
& P_A' = - c  \; && \text{ at } \; z=0, \\
& c \, (\rho_A - \rho_B) = - \rho_A \, P_A' + \rho_B \, P_B'  \; && \text{ at } \; z=0, \\
& \frac{1}{\eta_B} \, F_B\left(P_B\right) = c + \frac 12 P'_B  \; && \text{ at } \; z=\ell,\\
& P'_B  = -c \; && \text{ at } \;  z=\ell,
\end{aligned}
\end{equation}
where $\rho'_l$ and $P'_l$ denote the derivatives of $\rho_l$ and $P_l$ with respect to the variable $z$, with $l=A,B$. We consider the case where the following condition holds 
\beq
\label{e.TWPmininf}
\rho_A(z) \xrightarrow[ z \rightarrow - \infty]{} \rho^{\rm M},
\eeq
which implies $P_A(z) \xrightarrow[ z \rightarrow - \infty]{} P^{\rm M}$. Moreover, we note that  the principle of mass conservation ensures that
\beq
\label{e.Th1MFB}
\int_0^{\ell} \rho_B(z) \; dz=M,
\eeq 
for some $M>0$. The results of our travelling-wave analysis are summarised in the following theorem.
\begin{theorem}
\label{theorem.TWFB}
Under Assumptions \ref{assum_1}(i)-(iii), for any $M>0$ given there exist $c>0$ and $\ell > 0$ such that the travelling-wave problem defined by  system  \eqref{e.TWsysFB}, complemented with the asymptotic condition~\eqref{e.TWPmininf}, admits a solution whereby:
\begin{itemize}
\item[(i)] $\rho_A(z)$ is decreasing in $(-\infty,0)$ and satisfies the condition
\beq
\label{e.Th1rhoBellFA}
\rho_A^{\rm eq} < \rho_A(0) < \rho_A(z) < \rho^{\rm M} \; \mbox{ for all } \; z \in (-\infty,0);
\eeq
\item[(ii)] $\rho_B(z)$ is decreasing in $(0,\ell)$ and satisfies the condition
\beq
\label{e.Th1rhoBellFB}
\rho^{\rm eq}_B < \rho_B(\ell) < \rho_B(z) < \rho_B(0) \; \mbox{ for all } \; z \in  (0,\ell)
\eeq
along with the condition~\eqref{e.Th1MFB}.
\end{itemize}
Moreover, in the case where $F_A(\cdot)=F_B(\cdot)$, the following jump condition holds:
\beq
\label{e.Th1jump}
\sgn\left(\rho_A(0) - \rho_B(0) \right) = \sgn\left(\eta_A - \eta_B \right). 
\eeq
\end{theorem}

\begin{proof}  
We prove Theorem~\ref{theorem.TWFB} in five steps. 
\\\\
{\it Step 1: existence of a solution of  problem~\eqref{e.TWsysFB}. For $c>0$ given, we have the following problem for $P_B(z)$ 
$$
\begin{aligned}
&- c \, \rho_B' = \left(\rho_B \, P_B' \right)' \; &&\text{ in } 0 < z <  \ell, \\
&  P_B' = - c \; && \text{ at } \; z=0, \\
& \frac{1}{\eta_B} \, F_B\left(P_B\right) = \frac c2  \; && \text{ at } \; z=\ell.
\end{aligned}
 $$
Integrating the equation for $P_B$ over $(0,z)$, with $z<\ell$, and using the condition at $z=0$ we obtain an ordinary differential equation with final condition at $z=\ell$, that is,  
 $$
 \begin{aligned}
& \rho_B \, P_B'  = - c \, \rho_B\; &&\text{ in } 0 < z <  \ell, \\
 & \frac{1}{\eta_B} \, F_B\left(P_B\right) = \frac c2  \; && \text{ at } \; z=\ell,
 \end{aligned} 
 $$
 which can be solved explicitly giving $P_B(z) = c \, (\ell- z) + F_B^{-1}\Big(\dfrac{\eta_B}2 c\Big)$. Notice that since $c>0$ we have $P_B(\ell) > P_B^{\rm eq}$ and $F_B$ is invertible.  Knowing  $P_B(z)$, we have the following problem for $P_A(z)$ 
$$
\begin{aligned}
&- c \, \rho_A' = \left(\rho_A \, P_A' \right)' + G(P_A) \; \rho_A \; &&\text{ in }  -\infty < z < 0, \\
& \frac{1}{\eta_A} \, F_A(P_A)= \frac{1}{\eta_B} \, F_B(P_B)  \; && \text{ at } \; z=0, \\
& P_A' = - c  \; && \text{ at } \; z=0.
\end{aligned}
$$
Integrating the equation for $P_A$ over $(z,0)$, with $z<0$, and using the second condition at $z=0$ we obtain an ordinary differential equation with final condition at $z=0$, that is,
 $$
 \begin{aligned}
& \rho_A \, P_A'  = - c \, \rho_A + \int_z^0  G(P_A) \; \rho_A \, d\xi \; &&\text{ in } -\infty < z <  0, \\
& \frac{1}{\eta_A} \, F_A(P_A)= \frac{1}{\eta_B} \, F_B(P_B)  \; && \text{ at } \; z=0.
 \end{aligned} 
 $$
Under Assumptions~\ref{assum_1}(i), (iii) on the functions $F_A$, $F_B$ and $G$, the above problem  admits a solution. Hence, for a given $c>0$ there exists a solution of  problem~\eqref{e.TWsysFB}.
}
\\\\
{\it Step 2: monotonicity of $\rho_B$ in $(0,\ell)$ and proof of the condition~\eqref{e.Th1rhoBellFB}.} Integrating ~\eqref{e.TWsysFB}$_2$ between a generic point $z \in [0,\ell)$ and $\ell$, and using the boundary condition \eqref{e.TWsysFB}$_7$, we find
\beq
\label{e.Pon0rFB22}
P'_B(z) = - c < 0 \; \mbox{ for all } \; z \in [0,\ell].
\eeq
Moreover, integrating~\eqref{e.Pon0rFB22} between a generic point $z \in [0,\ell)$ and $\ell$ gives
\beq
\label{e:PeqFB}
P_B(z) = c \, (\ell - z) + P_B(\ell) \; \mbox{ for } \; z \in [0,\ell].
\eeq
Therefore,
\beq
\label{e.Pon0rFB12}
P_B(0) - P_B(\ell) = c \, \ell, \quad P_B(\ell) < P_B(z) < P_B(0) \; \mbox{ for all } \; z \in  (0,\ell).
\eeq
Furthermore, note that \eqref{e.Pon0rFB22} allows us to rewrite  the boundary condition \eqref{e.TWsysFB}$_6$  as
\beq
\label{e.BC72}
F_B\left(P_B(\ell)\right) = \eta_B \, \frac{c}{2} > 0. 
\eeq
Since under assumptions~\eqref{e.FTW}  we have that $F_B\left(P_B(\ell)\right)>0$ if and only if $P_B(\ell) > P^{\rm eq}_B$, we conclude that 
\beq
\label{e.PBell}
P_B(\ell) > P^{\rm eq}_B.
\eeq
Finally, since the function $F_B\left(P_B\right)$ is monotone for $P_B > P^{\rm eq}_B$, cf. \eqref{e.FTW}, the value of $P_B(\ell)$ is uniquely determined by~\eqref{e.BC72}.

Using the results~\eqref{e.Pon0rFB22}, \eqref{e.Pon0rFB12} and \eqref{e.PBell} along with the fact that, under assumptions~\eqref{e.barotreln} and \eqref{e.TWASC1FBb}, $P_B>0$ if and only if $\rho_B > \rho^{\rm eq}_B$ and 
 $P_B$ is a monotonically increasing and continuous function of $\rho_B$ for $\rho_B > \rho^{\rm eq}_B$, we conclude that the function $\rho_B$ is continuous in $(0,\ell)$ and satisfies the following conditions
\beq
\label{e.rho'<00rFB}
\rho'_B(z) < 0 \; \mbox{ for all } \; z \in  (0,\ell)
\eeq
and
\beq
\label{e.rhobound0rFB}
\rho^{\rm eq}_B < \rho_B(\ell) < \rho_B(z) < \rho_B(0) \; \mbox{ for all } \; z \in  (0,\ell).
\eeq
\\
{\it Step 3: identification of $\ell$.}
For $M>0$ given, since the value of $\rho_B(z)$ is uniquely determined for all $z \in [0,\ell]$, the value of $\ell$ is uniquely defined by the integral identity~\eqref{e.Th1MFB}.
\\\\
\noindent {\it Step 4: monotonicity of $\rho_A$ in $(-\infty,0]$ and proof of the condition~\eqref{e.Th1rhoBellFA}.} Recalling that $\dfrac{d P_A(\rho_A)}{d\rho_A} > 0$ for $\rho_A > \rho^{\rm eq}_A$,  we multiply both sides of~\eqref{e.TWsysFB}$_1$ by $\dfrac{dP_A}{d\rho_A}$ and use assumption~\eqref{e.TWPmininf}  to obtain the following boundary-value problem for $P_A$ 
\beq
\label{e.P-inf0FB}
- P_A' \, \left(c + P_A' \right) - P_A'' \, \rho_A \, \frac{d P_A}{d \rho_A} = G\left(P_A\right) \, \rho_A \, \frac{d P_A}{d \rho_A} \quad \mbox{in} \quad (-\infty,0]
\eeq
\beq
\label{e.P-inf0BCFB}
P_A(z) \xrightarrow[ z \rightarrow - \infty]{} P^{\rm M} \quad \mbox{and} \quad P_A(0)=P_A^0.
\eeq
Let $z^* \in  (-\infty,0)$ be a critical point of $P_A$. Using~\eqref{e.P-inf0FB} we conclude that 
$$
P_A''(z^*) = - G\left(P_A(z^*)\right).
$$
Therefore, under the conditions~\eqref{e.GFB} and \eqref{e.P-inf0BCFB}, the strong maximum principle ensures that $P_A < P^{\rm M}$ in $(-\infty,0]$ and that $P_A$ cannot have a local minimum in $(-\infty,0)$, \emph{i.e.}
\beq
\label{e.Pon-inf0FB}
P_A'(z) < 0 \; \mbox{ for all } \; z \in (-\infty,0).
\eeq 
Hence the solution $P_A$ of~\eqref{e.P-inf0FB},~\eqref{e.P-inf0BCFB} is a continuous and nonincreasing function that satisfies 
\beq
\label{e.Pon-inf0FB1}
P^{\rm eq}_A <  P_A(0) < P_A(z) < P^{\rm M} \; \mbox{ for all } \; z \in (-\infty,0).
\eeq 
Using the results~\eqref{e.Pon-inf0FB} and \eqref{e.Pon-inf0FB1} along with the fact that, under assumptions \eqref{e.barotreln} and \eqref{e.TWASC1FBb}, $P_A>0$ if and only if $\rho_A > \rho^{\rm eq}$ and $P_A$ is a monotonically increasing and continuous function of $\rho_A$ for $\rho_A > \rho^{\rm eq}$, we conclude that the function $\rho_A$ is continuous in $(-\infty,0)$ and satisfies the following conditions
\beq
\label{e.rho'<0FB}
\rho'_A(z) < 0 \; \mbox{ for all } \; z \in (-\infty,0)
\eeq
and
\beq
\label{e.rhoboundFB}
\rho_A^{\rm eq} < \rho_A(0) < \rho_A(z) < \rho^{\rm M} \; \mbox{ for all } \; z \in (-\infty,0),
\eeq
with  $\rho^{\rm M}$ being related  to $P^{\rm M}$ by \eqref{def.PeqPM}.
\\\\
{\it Step 5: proof of the jump condition \eqref{e.Th1jump}.} %Using~\eqref{e.Pon0rFB22} along with 
The transmission conditions~\eqref{e.TWsysFB}$_4$  and \eqref{e.TWsysFB}$_5$ give
\beq
\label{e.PAPBP0FB}
P_A'(0) = -c = P_B'(0).
\eeq
%As a consequence, the interface condition \eqref{e.TWsysFB}$_6$ is automatically satisfied. 
Furthermore, due to the uniqueness  of the value of $P_B(0)>P^{\rm eq}_B$, under the monotonicity assumptions \eqref{e.FTW}, the transmission condition \eqref{e.TWsysFB}$_3$ allows one to uniquely determine the value of $P_A(0) > P^{\rm eq}_A$. In particular, in the case where $F_A(\cdot)=F_B(\cdot) \equiv F(\cdot)$, the transmission condition \eqref{e.TWsysFB}$_3$ gives 
$$
\frac{F\left(P_A(0)\right)}{F\left(P_B(0)\right)} = \frac{\eta_A}{\eta_B} \quad \Longrightarrow \quad \sgn\left(F\left(P_A(0)\right) - F\left(P_B(0)\right)\right) = \sgn\left(\eta_A - \eta_B \right),
$$
from which, using the monotonicity assumptions~\eqref{e.barotreln} and~\eqref{e.FTW}, one finds the jump condition~\eqref{e.Th1jump}.
\end{proof}

\section{Numerical solutions of the free-boundary problem and computational simulations for the individual-based model} 
\label{numerics}
In this section, we illustrate the results established in Theorem~\ref{theorem.TWFB} by presenting a sample of nume\-rical solutions of the free-boundary problem  \eqref{eq:contin_110}, \eqref{eq:contin_11} and \eqref{eq:two_popul_model12}. Moreover, we compare  numerical solutions of the continuum model with computational simulations for the individual-based model  \eqref{DiscreteModelReindexingEqn}-\eqref{boundary_c_11}.
% complemented with initial conditions corresponding to those used for the free-boundary problem.
%
We focus on the case where the force terms $F_{A}(d_{ij})$ and $F_{B}(d_{ij})$ are both given by the following cubic approximation of the JKR force law~\cite{johnson1971surface}
\beq
\label{e.Fij}
F(d_{ij}) = a_1(d_{ij} - d^{\rm eq}) + a_2 (d_{ij} - d^{\rm eq})^2 + a_3(d_{ij} - d^{\rm eq})^3 \; \text{ for } \; d_{ij} < d^{\rm eq},
\eeq
where $d_{ij} = |r_i - r_j|$ and $d^{\rm eq}$ stands for the equilibrium intercellular distance, which is the distance between cell centres above which cells do not exert any force upon one another (\emph{i.e.} $F(d_{ij}) =0$ for all $d_{ij} \geq d^{\rm eq}$). The equilibrium distance $d^{\rm eq}$ and the coefficients $a_1$, $a_2$ and $a_3$ depend on the biophysical characteristics of the cells and are defined as
\beq
\label{e.a123}
\begin{aligned}
&d^{\rm eq} = 2R - \frac 12 \frac{(\pi\gamma)^{2/3} (3R)^{1/3}}{\tilde E^{2/3}},  &&
a_1 = - \frac 3{5}  (3R \tilde E)^{2/3}  ( \pi \gamma)^{1/3} d^{\rm eq}, \\
&a_2 = \frac{33}{125} \frac{\tilde E^{4/3} (3R)^{1/3} }{(\pi \gamma)^{1/3}}(d^{\rm eq})^2, &&
a_3 =  \frac{209}{3125}  \frac{\tilde E^2}{\pi \gamma}(d^{\rm eq})^3.
\end{aligned}
\eeq
In the formulas \eqref{e.a123}, $R$ is the cell radius, the parameter $\gamma$ measures the strength of cell-cell adhesion and $\tilde E$ is an effective Young's modulus defined as
\beq
\label{e.tildeE}
\tilde E  = \frac{E}{2 (1- \nu^2)},
\eeq
with $E$ and $\nu$ being, respectively, the Young's modulus and the Poisson's ratio of the cells. We refer the interested reader to Appendix~\ref{JKR_approx} for a detailed derivation of the approximate representation of the JKR force law given by~\eqref{e.Fij}-\eqref{e.tildeE}. 

Using the formal relations between the intercellular distance and the cell density \eqref{eq:approxcelldensity} and \eqref{eq:approxcelldensity1}, we compute the cell densities and the equilibrium cell density  as 
\beq
\label{celldensdisc}
\begin{aligned}
& \rho_A(t,r_i) = \frac{1}{r_{i+1} - r_i} \; \;  && \text{for } \; \; i=1, \ldots, m,   \\ 
&  \rho_B(t,r_i) = \frac{1}{r_{i+1} - r_i} \; \;  &&  \text{for }\; \;  i=m+1, \ldots, n,  
\end{aligned} 
\eeq
and 
$
\rho^{\rm eq}_A = \rho^{\rm eq}_B = \rho^{\rm eq} = 1/{d^{\rm eq}}$. 
The approximation of the JKR force law \eqref{e.Fij} can be rewritten in terms of the cell densities $\rho_l$ and $\rho^{\rm eq}$ as follows
\beq
\label{e.Fijrho}
F\left(1 /\rho_l \right) = a_1 \left(\frac{1}{\rho_l} -\frac{1}{\rho^{\rm eq}}\right) + a_2 \left(\frac{1}{\rho_l} -\frac{1}{\rho^{\rm eq}}\right)^2 + a_3 \left(\frac{1}{\rho_l} -\frac{1}{\rho^{\rm eq}}\right)^3. 
\eeq
Inserting the latter expression for $F\left(1 /\rho_l \right)$ into the definition~\eqref{eq:Dl} of the nonlinear diffusion coefficient $D_l(\rho_l)$ yields 
\beq
\label{DiffCoGen}
D_l(\rho_l)=
\begin{cases}
-\dfrac{3a_3 \, (\rho^{\rm eq}- \rho_l)^2+2a_2\, (\rho^{\rm eq}-\rho_l) \, \rho_l \, \rho^{\rm eq} +a_1\, \left(\rho_l \, \rho^{\rm eq}\right)^2}{\eta_l \left(\rho^{\rm eq} \right)^2 \rho_l^4}  & \text{ if }  \rho_l > \rho^{\rm eq}, 
\\
0 \phantom{-\frac{3a_3 \, (\rho^{\rm eq}- \rho_l)^2+2a_2\, (\rho^{\rm eq}-\rho_l)\rho_l+a_1\, \rho_l^2}{\eta_l \,  \rho_l^4}}    &  \text{ if } \; \rho_l \leq \rho^{\rm eq}, 
\end{cases}
\eeq
from which, using~\eqref{e.barotreln}, we obtain the following barotropic relation for the cell pressure $P_l$
\beq
\label{JKRPressureDensity}
P_l(\rho_l)=
\begin{cases}
\dfrac{1}{\eta_l \left(\rho^{\rm eq}\right)^2 \rho_l^4}\left(\dfrac{\alpha_1}{2}\rho_l^2 + \dfrac{2 \alpha_2}{3}\rho_l+ \dfrac{3 \alpha_3}{4}\right) + P^l_0   & \text{ if } \; \rho_l \geq \rho^{\rm eq}, 
\\
0 \phantom{\left(\dfrac{\alpha_1}{2}\rho_l^2 + \dfrac{2 \alpha_2}{3}\rho_l+ \dfrac{3 \alpha_3}{4}\right) + P_0^l}    &  \text{ if } \; \rho_l < \rho^{\rm eq}. 
\end{cases}
\eeq
In~\eqref{JKRPressureDensity}, the term $P^l_0$ is an integration constant such that $P_l(\rho^{\rm eq}) =0$ and
\beq
\label{e.alpha123}
\alpha_1 =  a_1 \left(\rho^{\rm eq}\right)^2 - 2 \, a_2 \rho^{\rm eq} + 3 a_3, \quad \alpha_2 =  a_2 \left(\rho^{\rm eq}\right)^2 - 3 \, a_3 \rho^{\rm eq}, \quad \alpha_3 =  a_3 \left(\rho^{\rm eq}\right)^2.
\eeq

\begin{figure}
\begin{center}
{\includegraphics[scale=0.3] {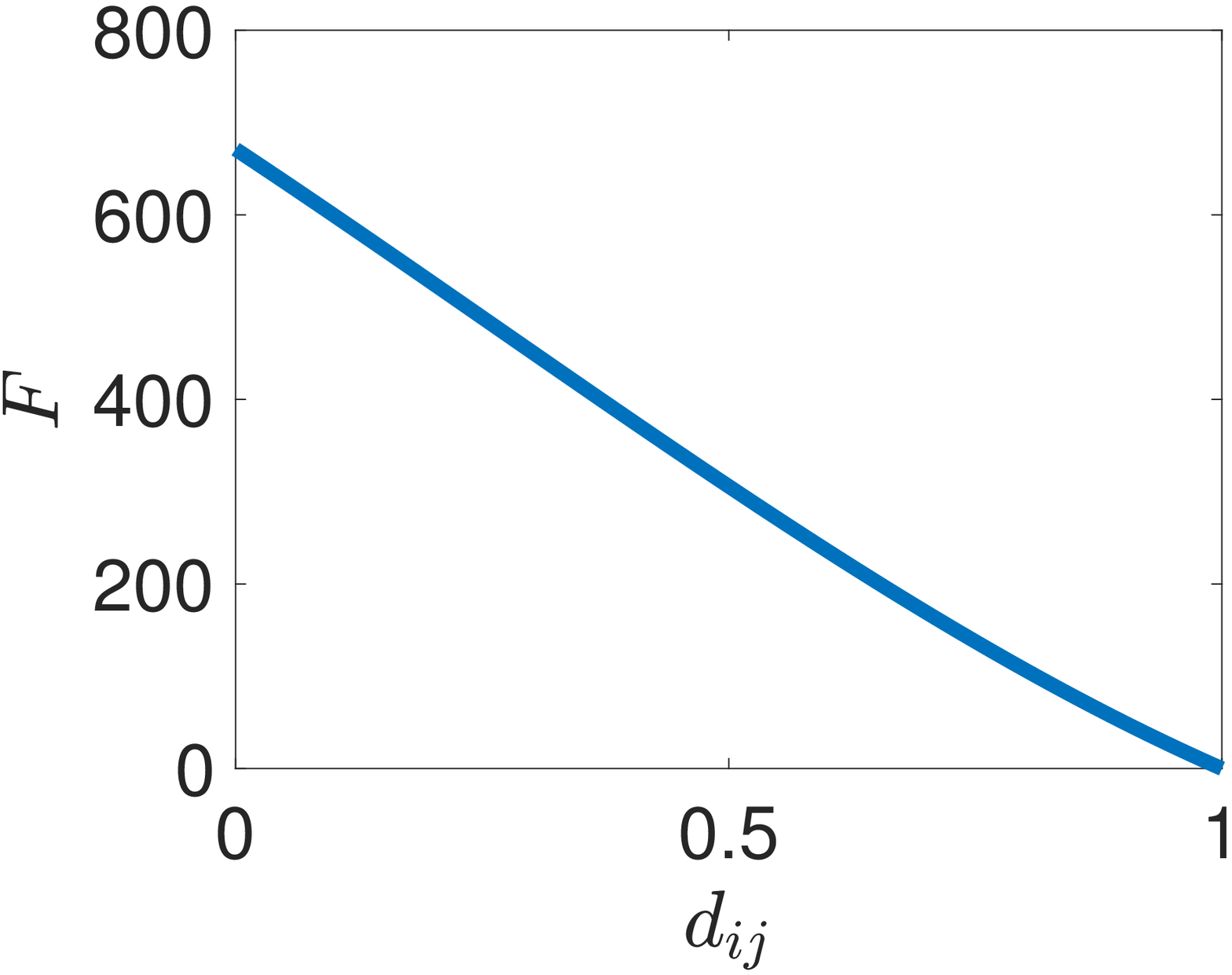}}
{\includegraphics[scale=0.3] {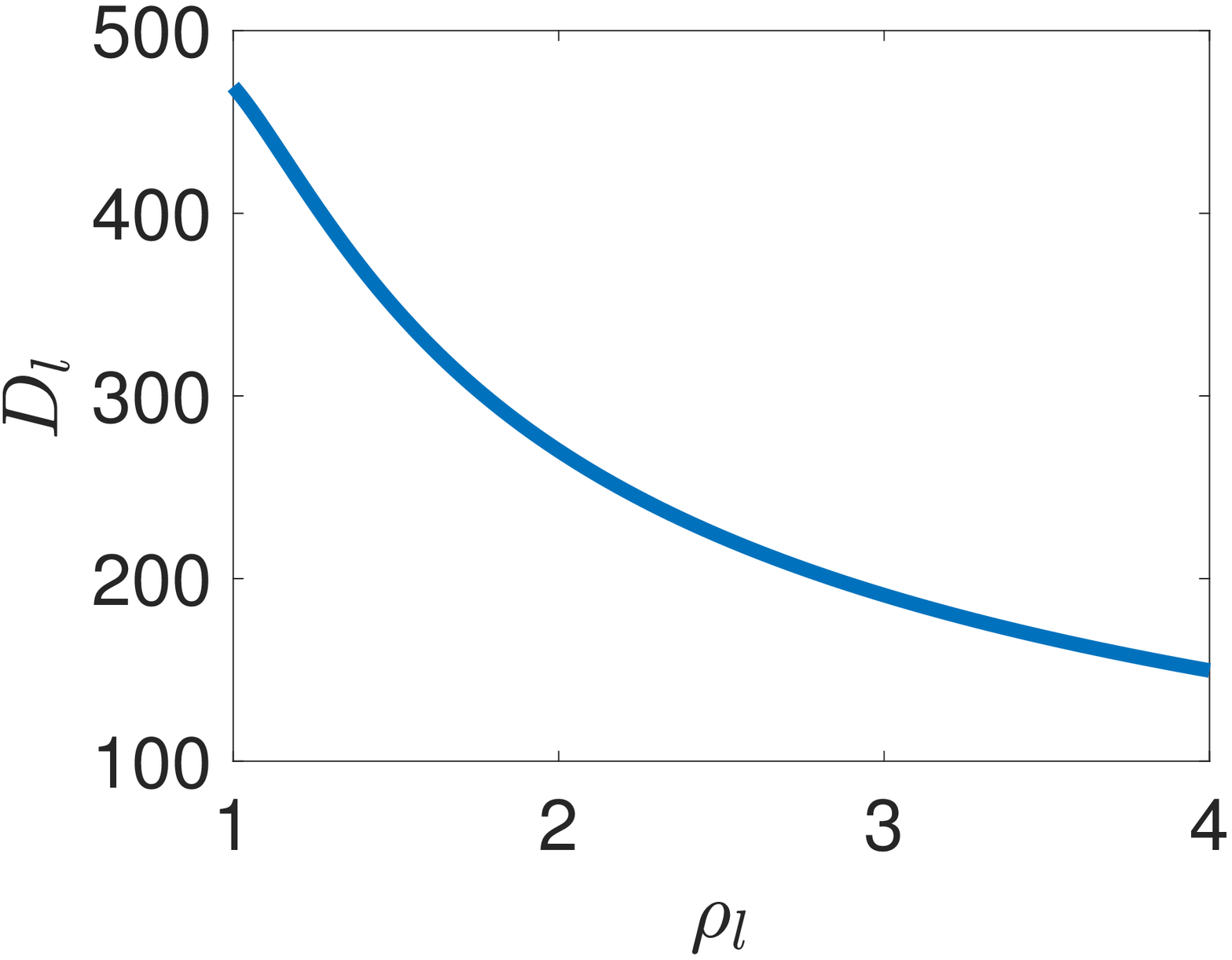}}
{\includegraphics[scale=0.3] {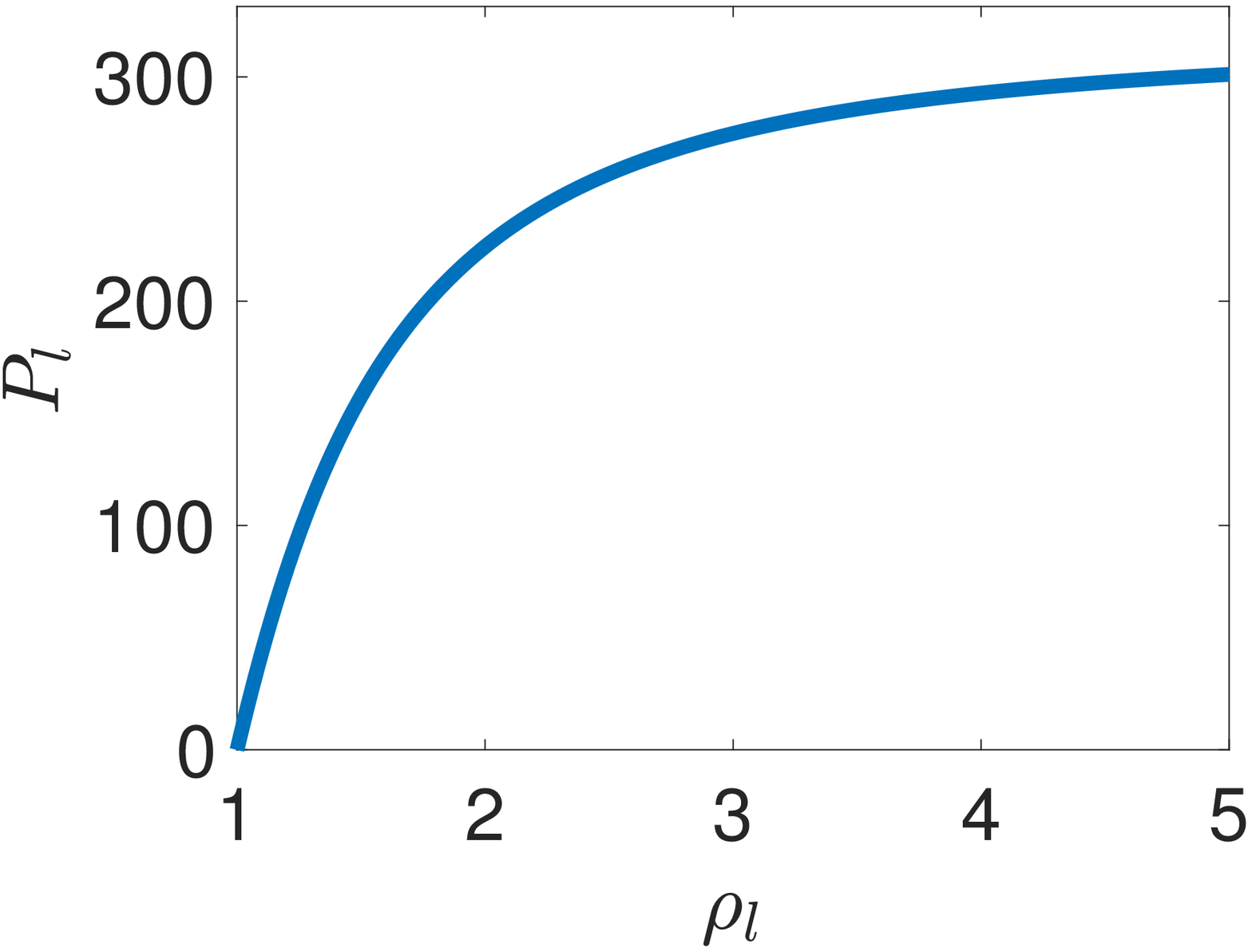}}
\caption{{\it Representations of the JKR model.} {\it Left panel.} The force between neighbouring cells $F$, defined by  \eqref{e.Fijrho}, is plotted against the intercellular distance $d_{ij}$. {\it Central panel.} The nonlinear diffusion coefficient $D_l$, defined by  \eqref{DiffCoGen}, is plotted against the cell density $\rho_l$. {\it Right panel.} The cell pressure $P_l$, defined by  \eqref{JKRPressureDensity}, is plotted against the cell density $\rho_l$. The  parameter values are as in~\eqref{paramval} and $\eta_l=0.5 \times 10^{-2} \, \text{kg} \, \text{s}^{-1}$. The values of $d_{ij}$ and $\rho_l$ are nondimensionalised by $d^{\rm eq}$ and $\rho^{\rm eq}$, respectively.}
\label{ForceDisplacement}
\end{center}
\end{figure}

Numerical simulations were performed using parameter values chosen in agreement with those used in~\cite{drasdo2012modeling}, that is,
\beq \label{paramval}
E= 300 \, \text{Pa}, \quad \nu=0.4, \quad \gamma= \zeta \, k_B \, T, \quad R = 7.5 \times 10^{-6} \, \text{m},
\eeq
where $\zeta =10^{15} \, \text{m}^2$ is the density of cell-cell adhesion molecules in the cell membrane, $k_B$ is the Boltzmann constant and $T = 298 \, \text{K}$ is an absolute temperature. Figure~\ref{ForceDisplacement} displays the plots of the force $F$ between neighbouring cells,  the nonlinear diffusion coefficient $D_l$, and the cell pressure $P_l$, obtained using the parameter values given by~\eqref{paramval}. 

We let the cell damping coefficients of population $A$ be $\eta_A=0.5 \times 10^{-2} \, \text{kg} \, \text{s}^{-1}$, and considered the cases where $\eta_A=\eta_B$ or $\eta_B=2 \, \eta_A$ or $\eta_B=0.5 \, \eta_A$. Moreover, for the cell proliferation term we assumed
$$
g(1/\rho_A) = \tilde{H}(1/\rho_A-1/\rho^{\rm M}) \quad \text{and} \quad G(\rho_A) = \alpha \, \tilde{H}(\rho^{\rm M}-\rho_A), \quad \text{with} \quad  \rho^{{\rm M}}=\frac{4}{3} \, \rho^{\rm eq} \quad \text{and} \quad \alpha=\frac{1}{2},
$$
where $\tilde{H}$ is a smooth approximation to the Heaviside  function. 

To construct numerical solutions, the free-boundary problem~\eqref{eq:contin_110}, \eqref{eq:contin_11} and \eqref{eq:two_popul_model12} was transformed to a Lagrangian reference frame and the method of lines was employed to solve the resultant equations. The resulting system of ordinary differential equations,  as well as the ordinary differential equations \eqref{discrete_11} and \eqref{boundary_c_11} of the individual-based model, were numerically solved using the {\sc Matlab} routine {\sc ode15s}. 

The plots in Figures~\ref{EqualMotilityTwoPopJKRModelDensity}-\ref{HigherMotilityNonProlifTwoPopJKRModelDensity} show sample dynamics of the cell density $\rho$ defined as
\beq
\label{celldens}
\rho(t,x)=\left\{
\begin{array}{ll}
\rho_A(t,x), \quad \text{if } x \leq s_1(t),
\\\\
\rho_B(t,x), \quad \text{if } x > s_1(t),
\end{array}
\right.
\eeq
with $\rho_A$ and $\rho_B$ being either numerical solutions of the free-boundary problem~\eqref{eq:contin_110}, \eqref{eq:contin_11}, \eqref{eq:two_popul_model12} (black lines) or approximate cell densities 
computed from numerical solutions of the individual-based model~\eqref{DiscreteModelReindexingEqn}-\eqref{boundary_c_11}
 using~\eqref{celldensdisc} (red markers). In agreement with the results established in Theorem~\ref{theorem.TWFB}, we observe the emergence of travelling-wave solutions, whereby the positions of the inner  free boundary $s_1(t)$ and the outer free boundary $s_2(t)$ move at the same constant speed, and the cell densities $\rho_A$ and $\rho_B$ are monotonically decreasing. 

Moreover, $\rho$ is continuous at $s_1$ if $\eta_A=\eta_B$, cf. the plots in Figure \ref{EqualMotilityTwoPopJKRModelDensity}, whereas it has a jump discontinuity at $s_1$ both for $\eta_A<\eta_B$ and for $\eta_A>\eta_B$. The sign of the jump $\rho(s_1^+) - \rho(s_1^-)$ satisfies condition~\eqref{e.Th1jump}, cf. the plots in Figures \ref{LowerMotilityNonProlifTwoPopJKRModelDensity}~and ~\ref{HigherMotilityNonProlifTwoPopJKRModelDensity}, and, once that the travelling-wave is formed, the size of the jump is constant and such that the transmission condition~\eqref{e.TWsysFB}$_3$ is met, see Supplementary Figure \ref{Figure_5}(a). 
  As shown by these plots, there is an excellent match between the numerical solutions of the free-boundary problem and the computational simulation results for the individual-based model.
\begin{figure}
\centering
{\includegraphics[scale = 0.31]{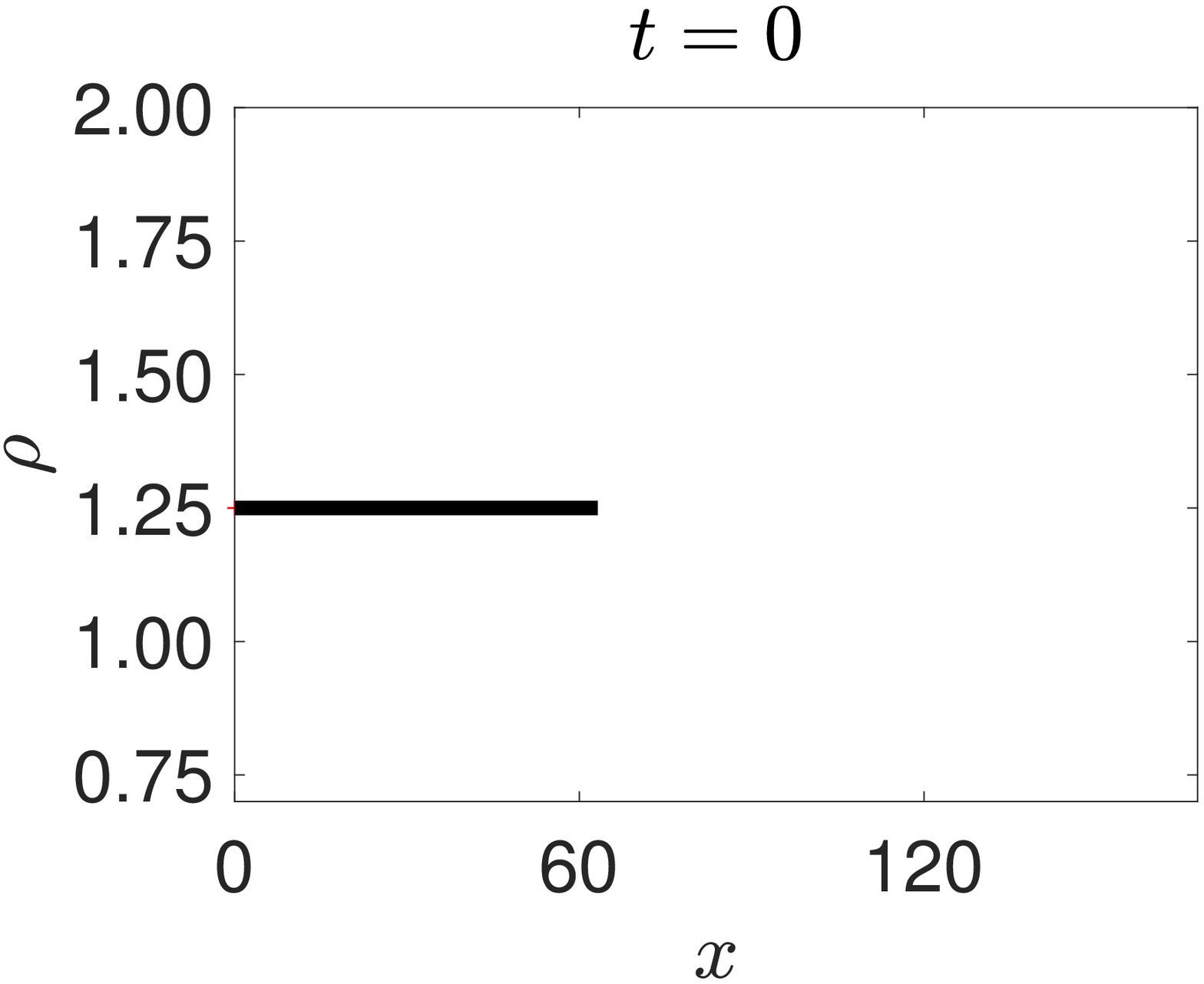}} 
{\includegraphics[scale = 0.31]{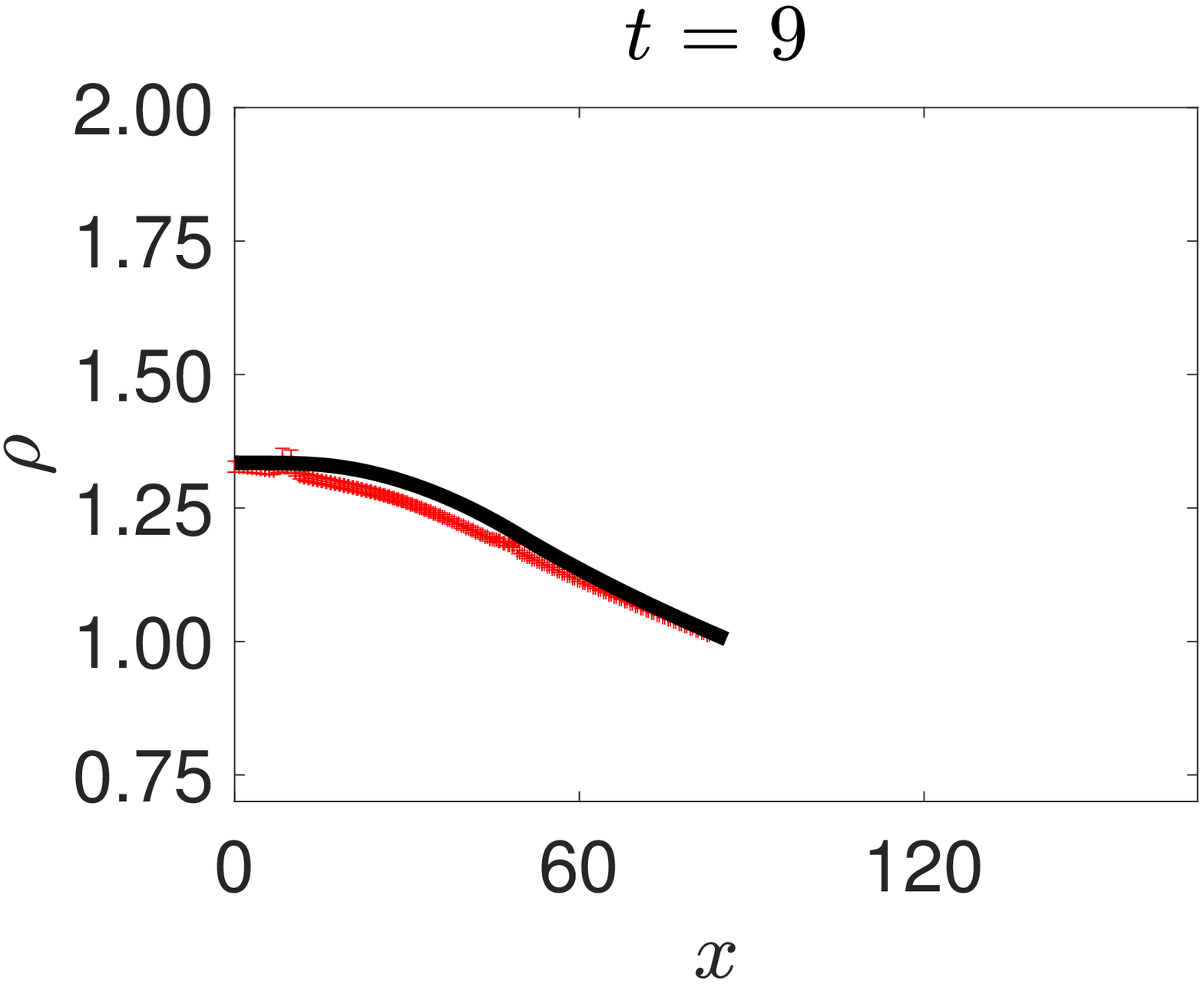}} 
{\includegraphics[scale = 0.31]{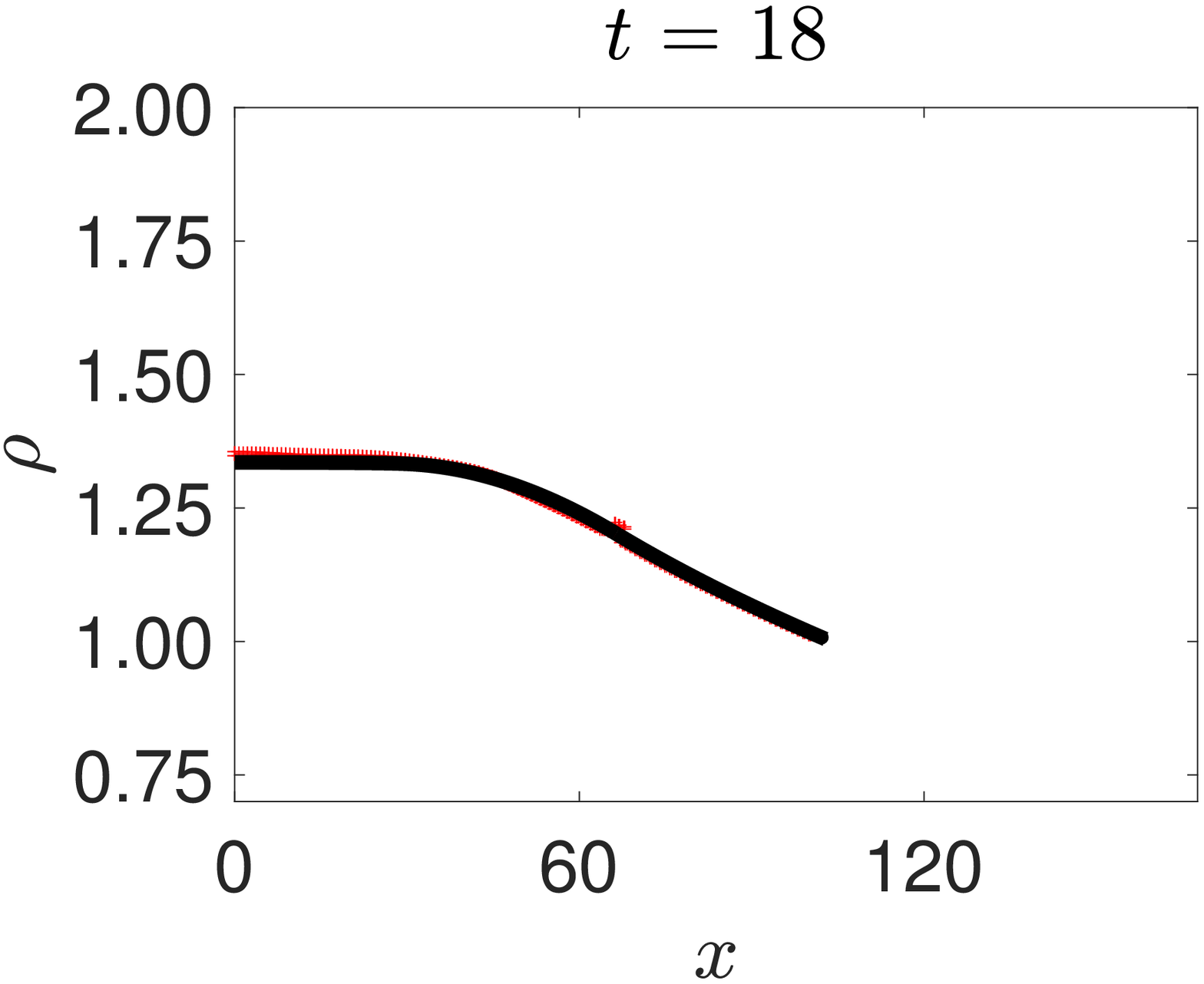}}\\
{\includegraphics[scale = 0.31]{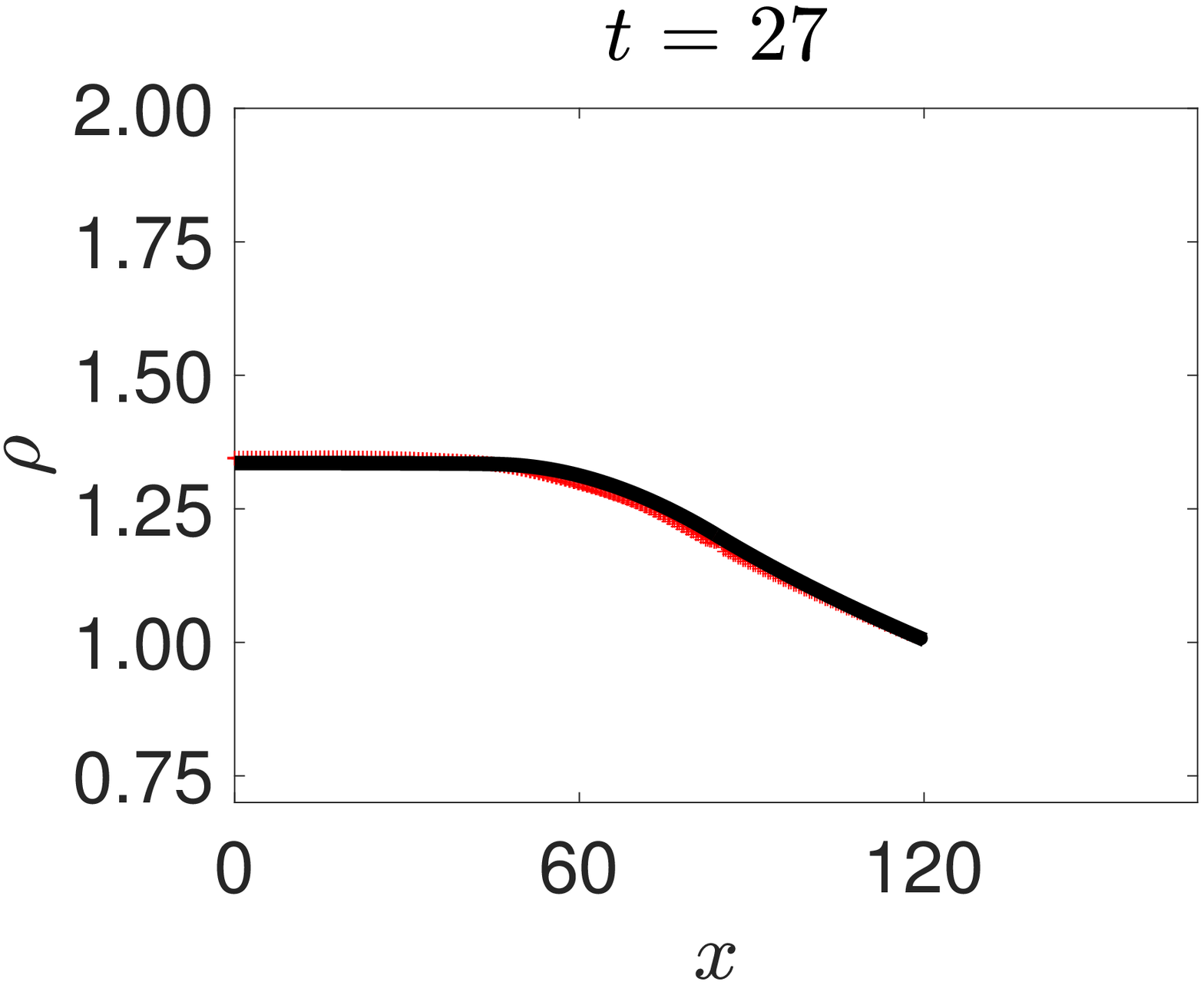}} 
{\includegraphics[scale = 0.31]{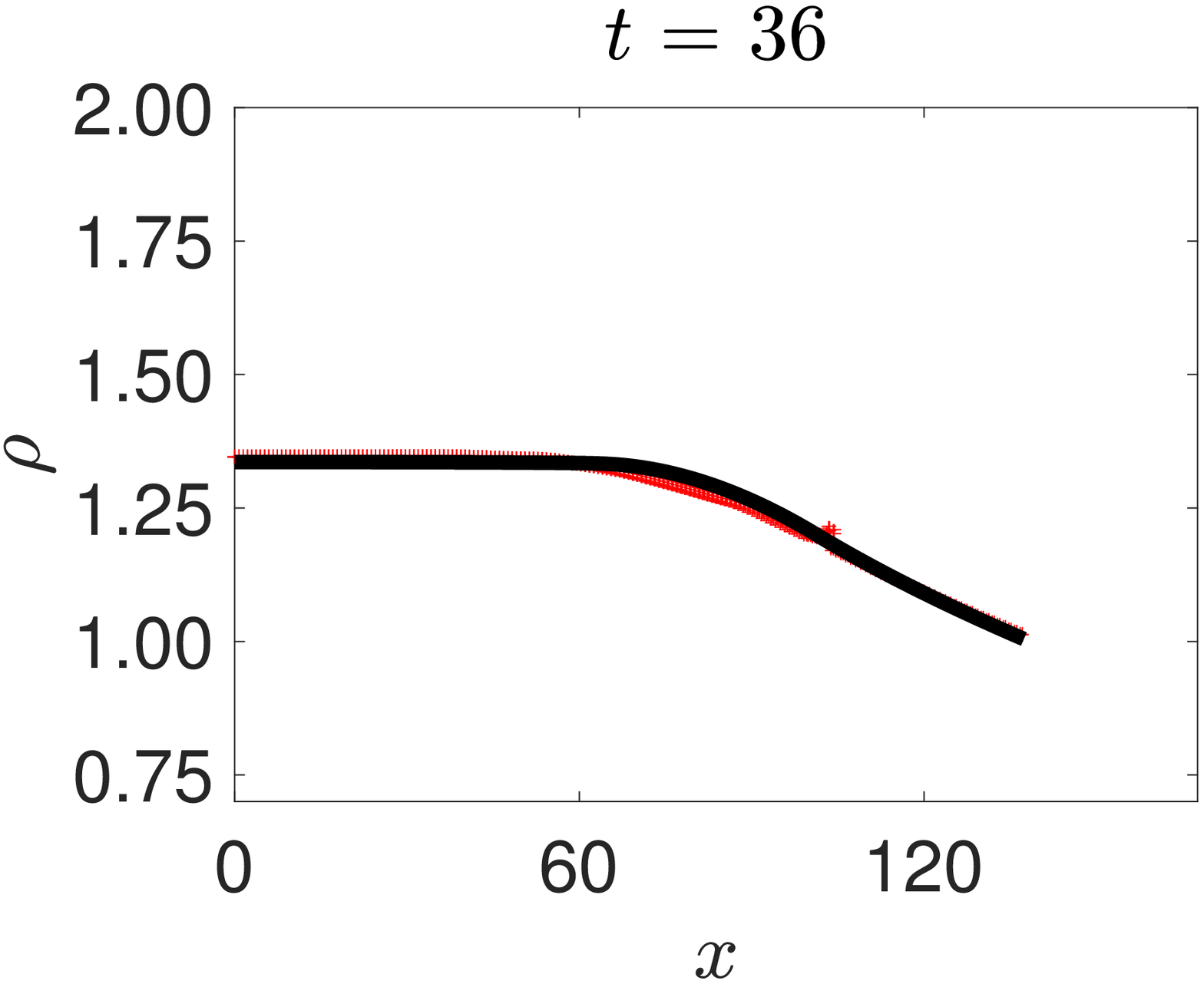}}
{\includegraphics[scale = 0.31]{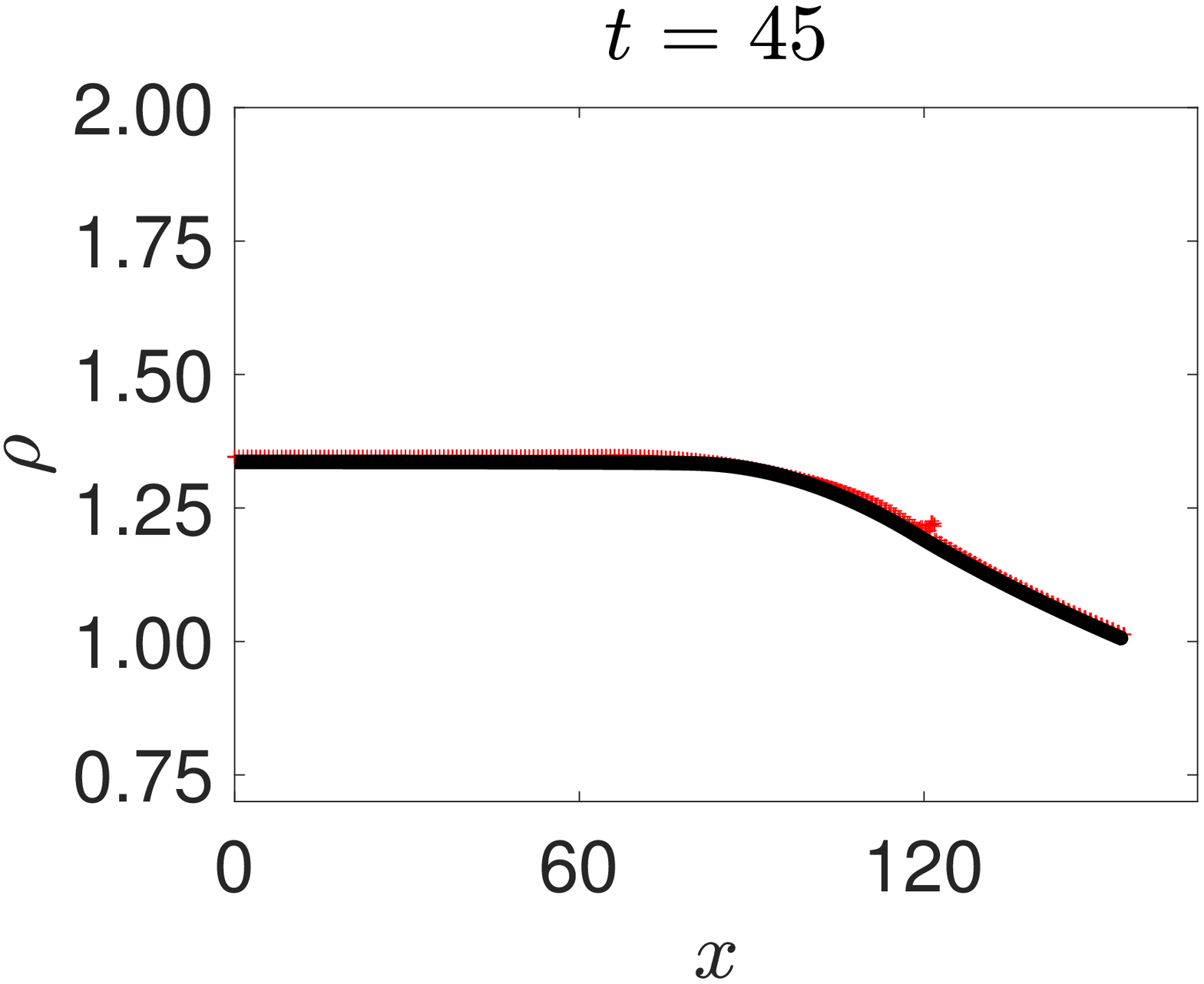}} 
\caption{{\it Comparison between the free-boundary problem and the individual-based model for $\eta_A=\eta_B$.} The cell density, $\rho(t,x)$, given by  \eqref{celldens} is plotted against $x$ for increasing values of $t$. The cell densities $\rho_A$ and $\rho_B$  are either numerical solutions of the free-boundary problem (black lines) or approximate cell densities computed from simulation results for the individual-based model using~\eqref{celldensdisc} (red markers). The values of $x$ are nondimensionalised by $d^{\rm eq}$, while the values of $\rho$ are nondimensionalised by $\rho^{\rm eq}$.}\label{EqualMotilityTwoPopJKRModelDensity}
\end{figure}

\begin{figure}
\centering
{\includegraphics[scale = 0.31]{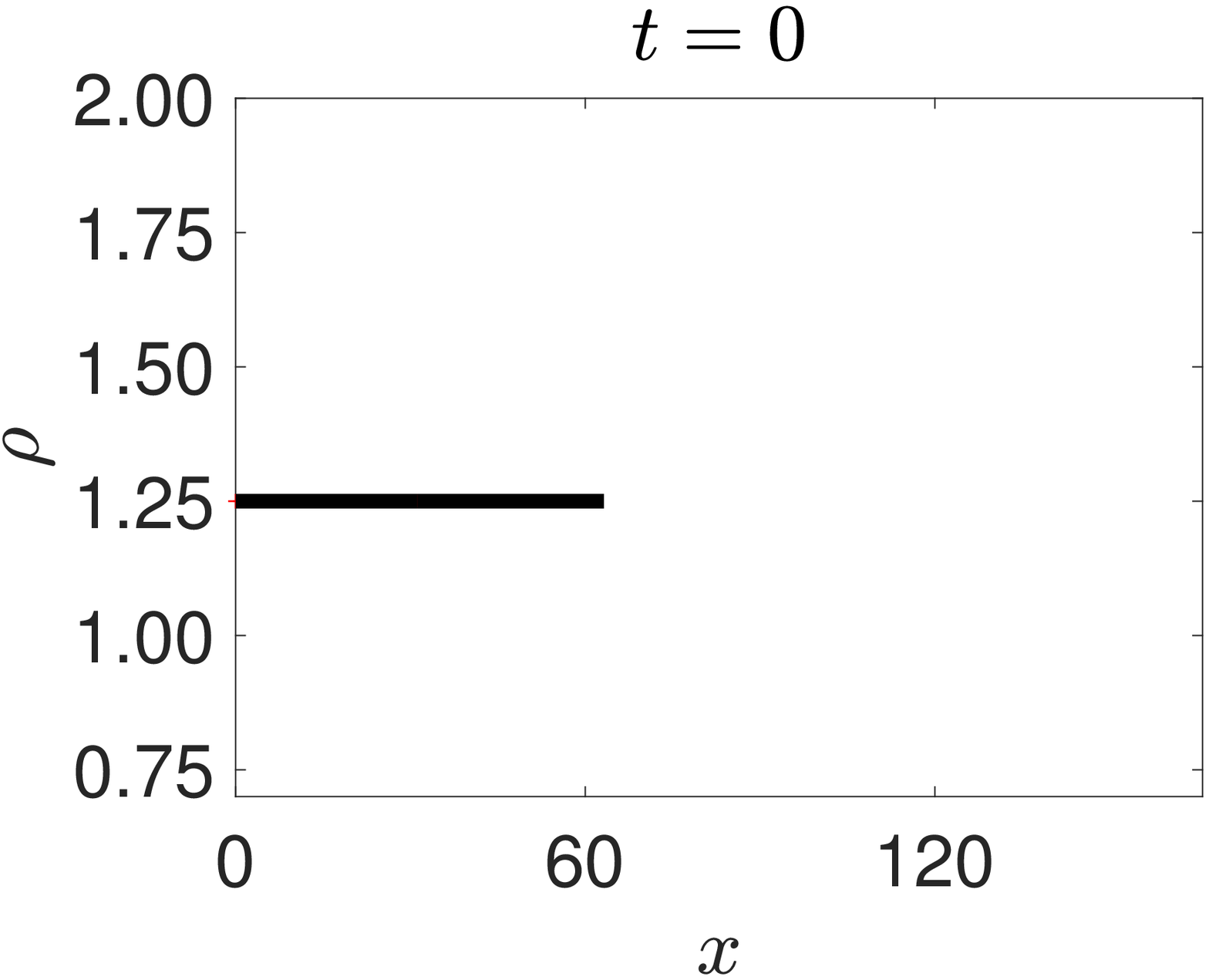}} 
{\includegraphics[scale = 0.31]{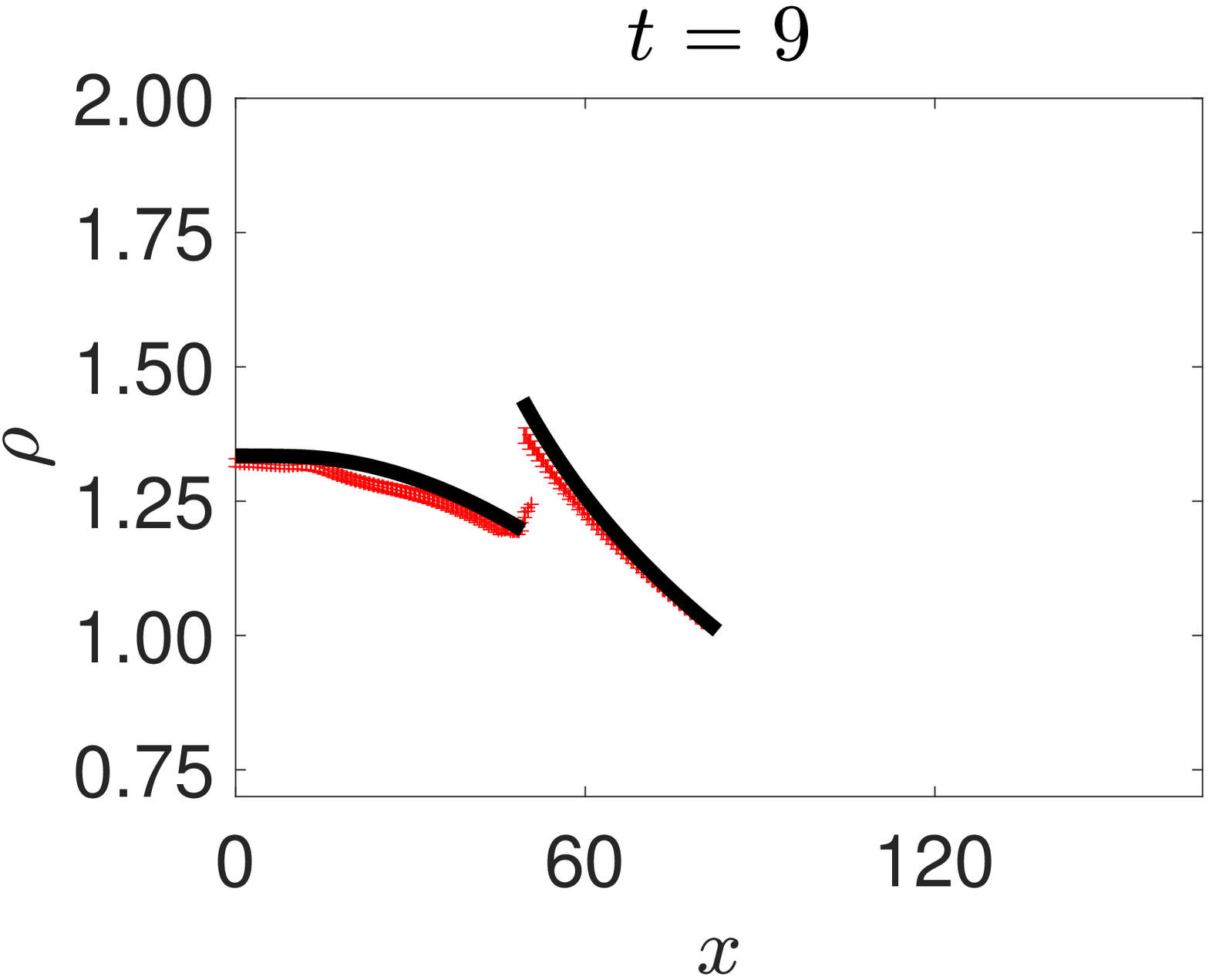}} 
{\includegraphics[scale = 0.31]{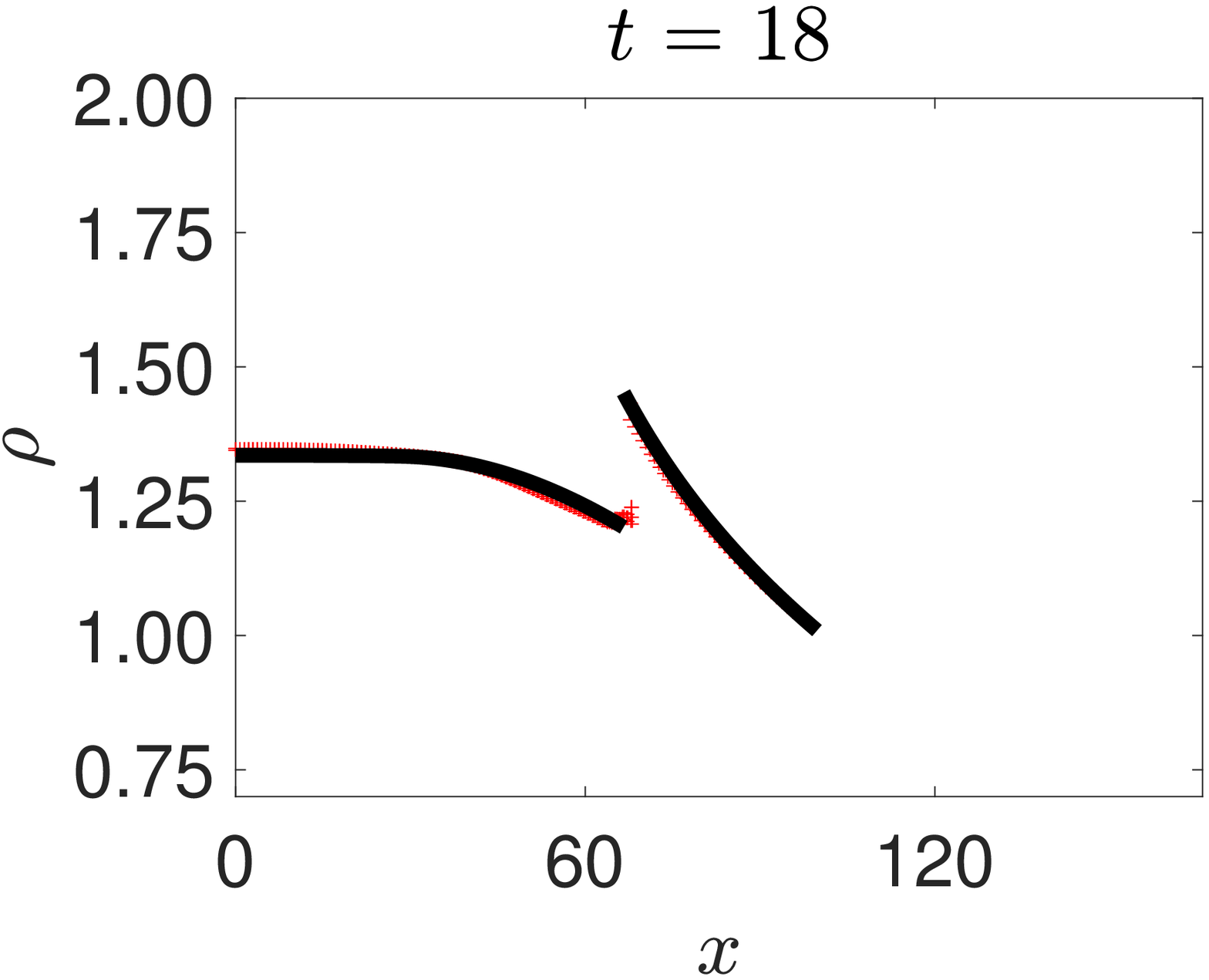}}\\
{\includegraphics[scale = 0.31]{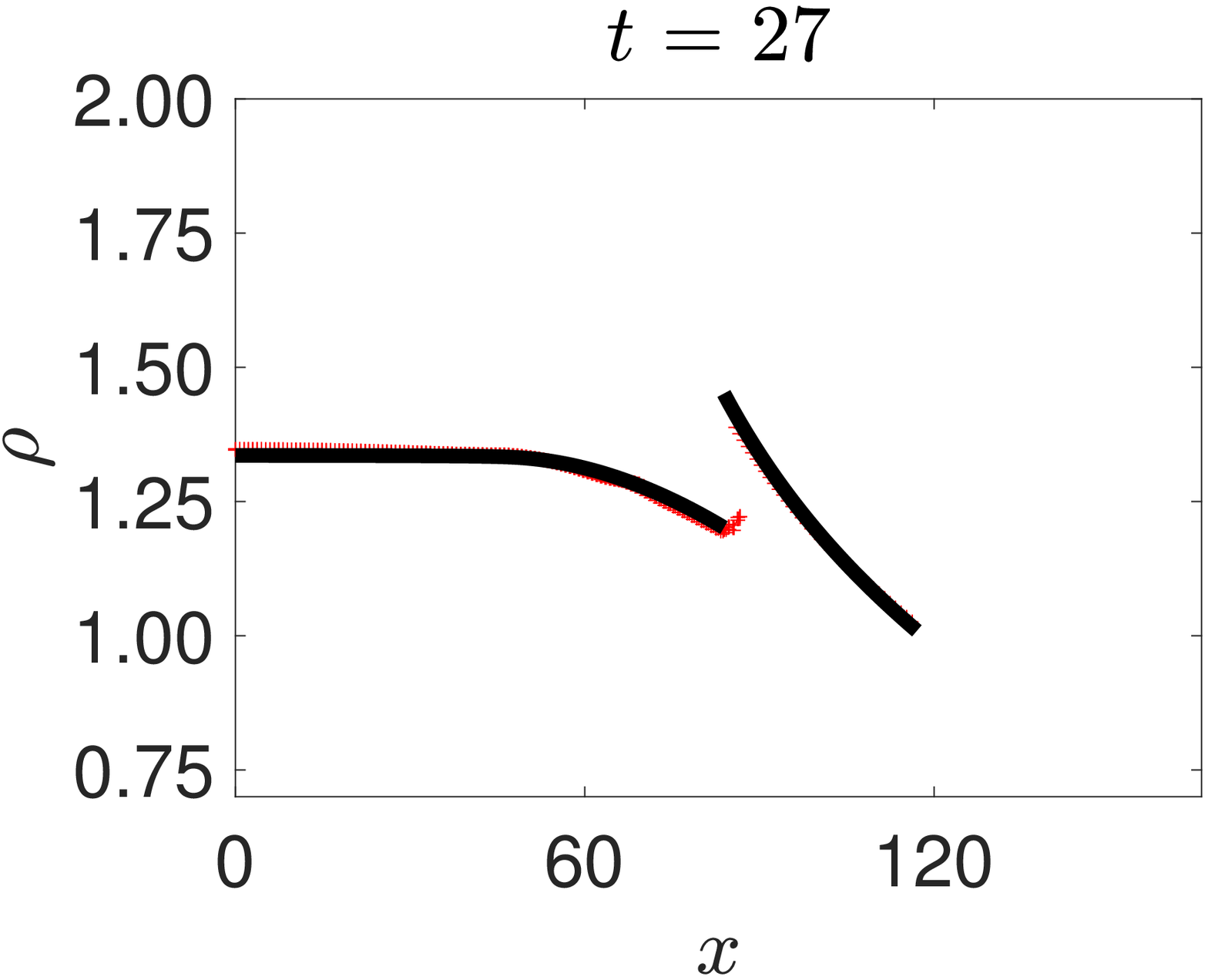}} 
{\includegraphics[scale = 0.31]{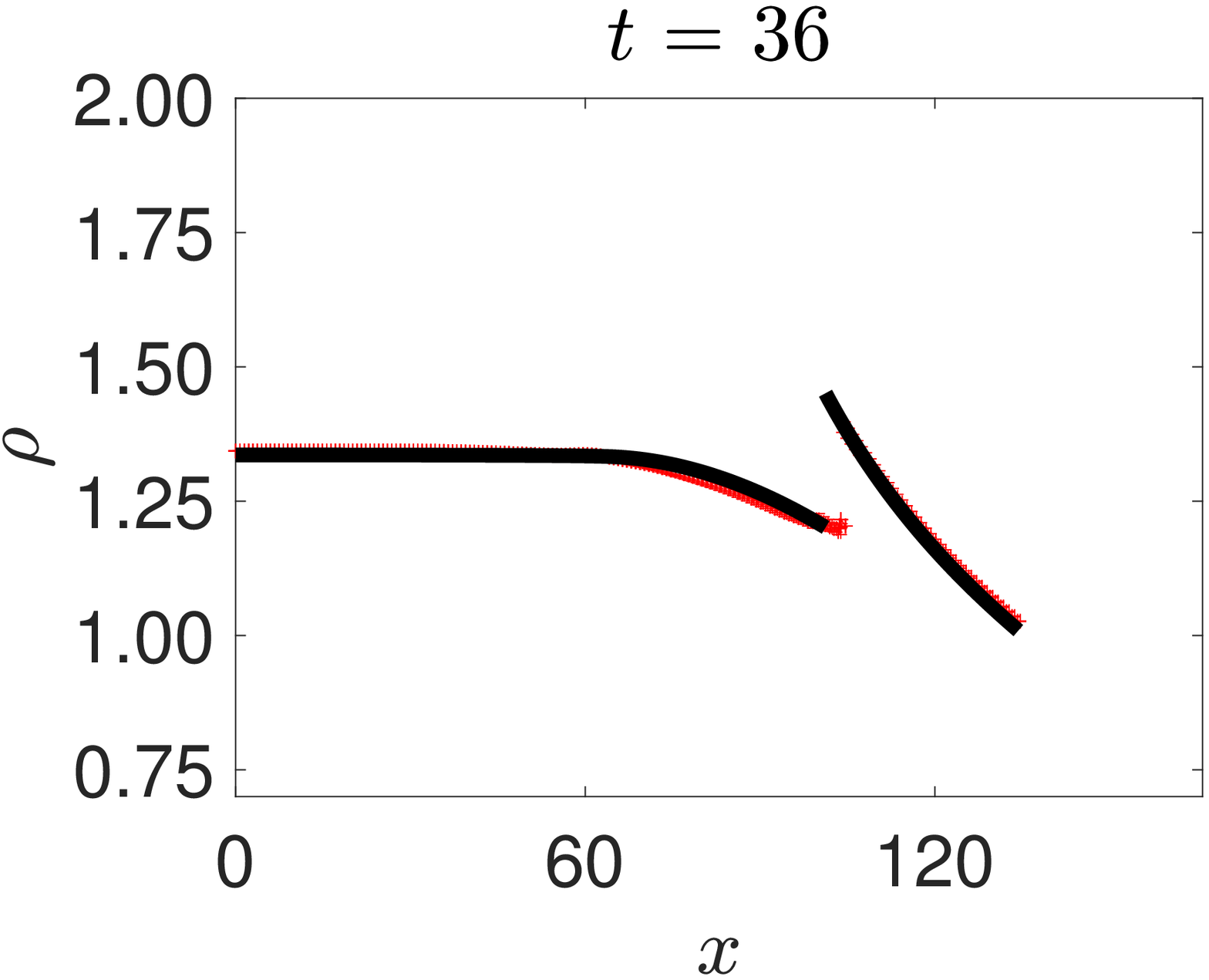}}
{\includegraphics[scale = 0.31]{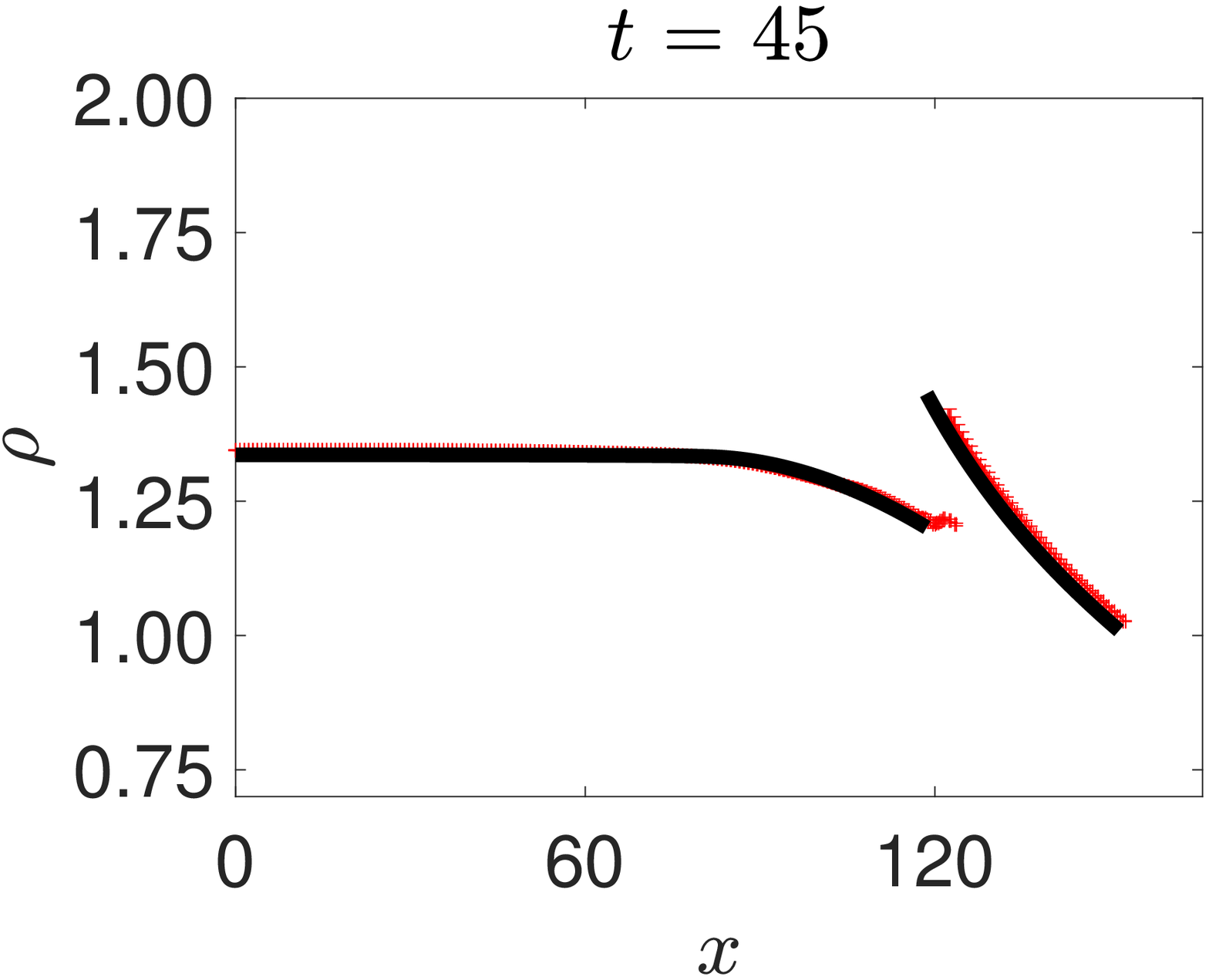}} 
\caption{{\it Comparison between the free-boundary problem and the individual-based model for $\eta_A<\eta_B$.} The cell density, $\rho(t,x)$, given by  \eqref{celldens}, is plotted against $x$ for increasing values of $t$. The cell densities $\rho_A$ and $\rho_B$ are either numerical solutions of the free-boundary problem (black lines) or approximate cell densities computed  from  simulation results for the individual-based model using~\eqref{celldensdisc} (red markers). The values of $x$ are nondimensionalised by $d^{\rm eq}$, while the values of $\rho$ are nondimensionalised by $\rho^{\rm eq}$.}
\label{LowerMotilityNonProlifTwoPopJKRModelDensity}
\end{figure}

\begin{figure}
\centering
{\includegraphics[scale = 0.31]{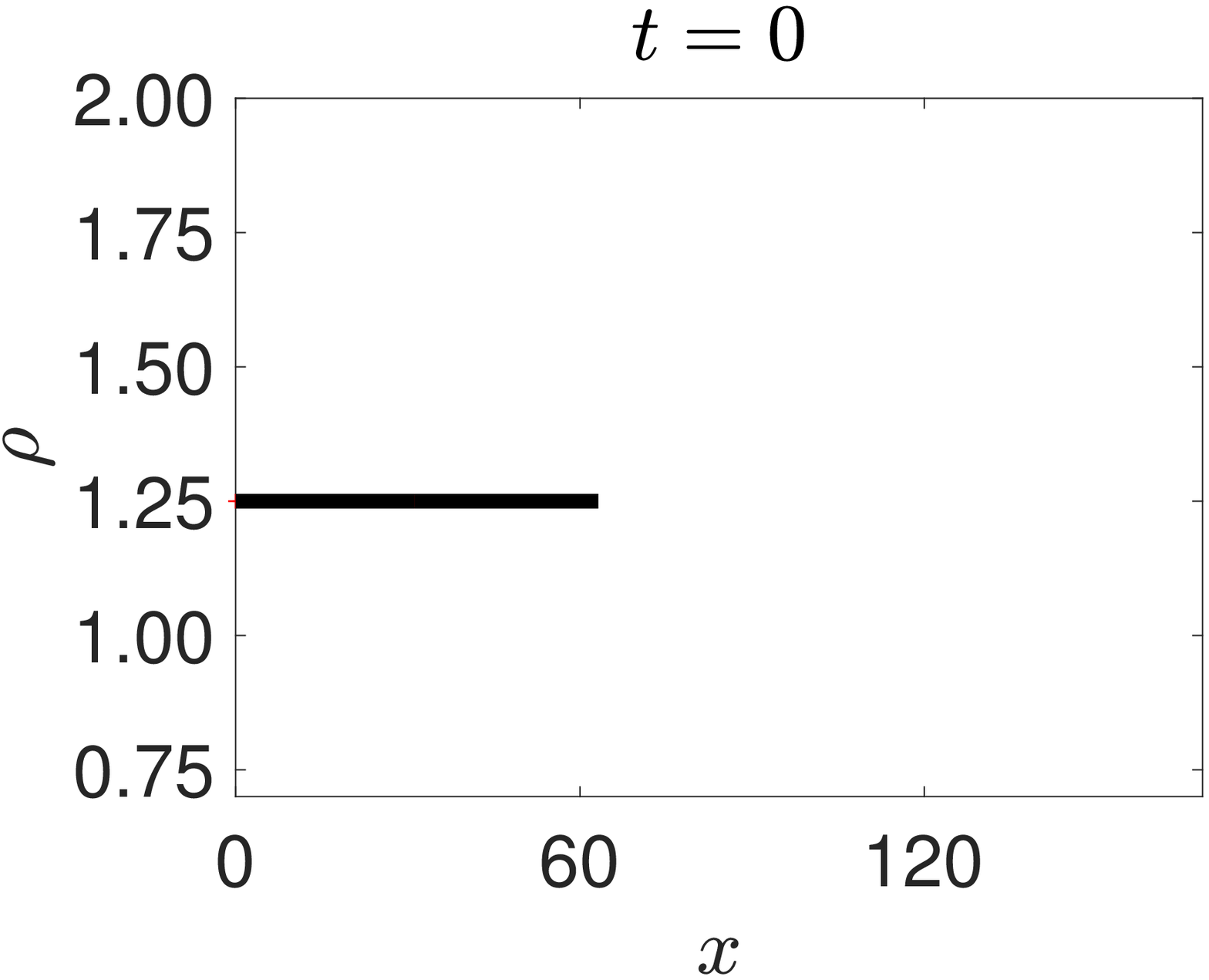}} 
{\includegraphics[scale = 0.31]{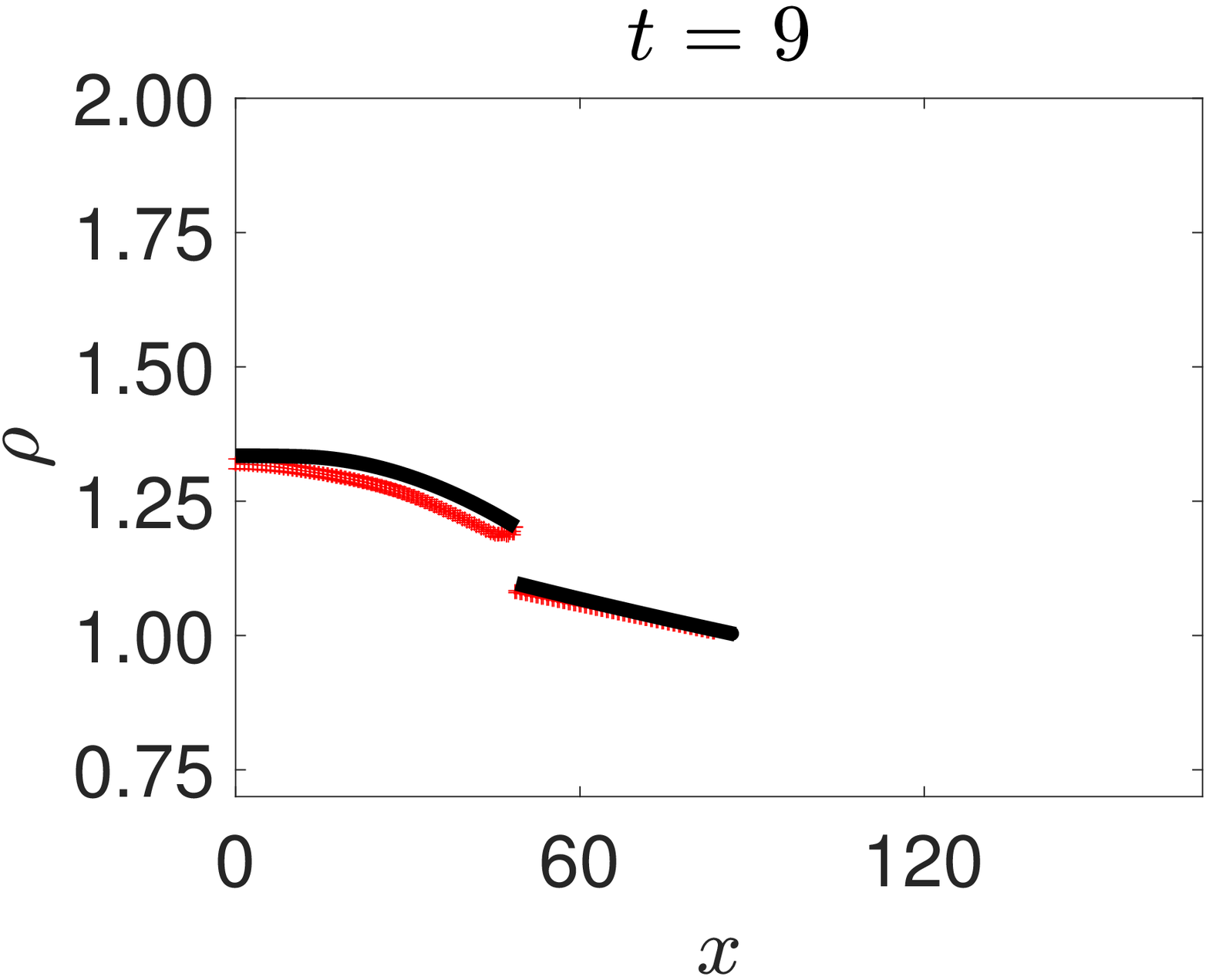}} 
{\includegraphics[scale = 0.31]{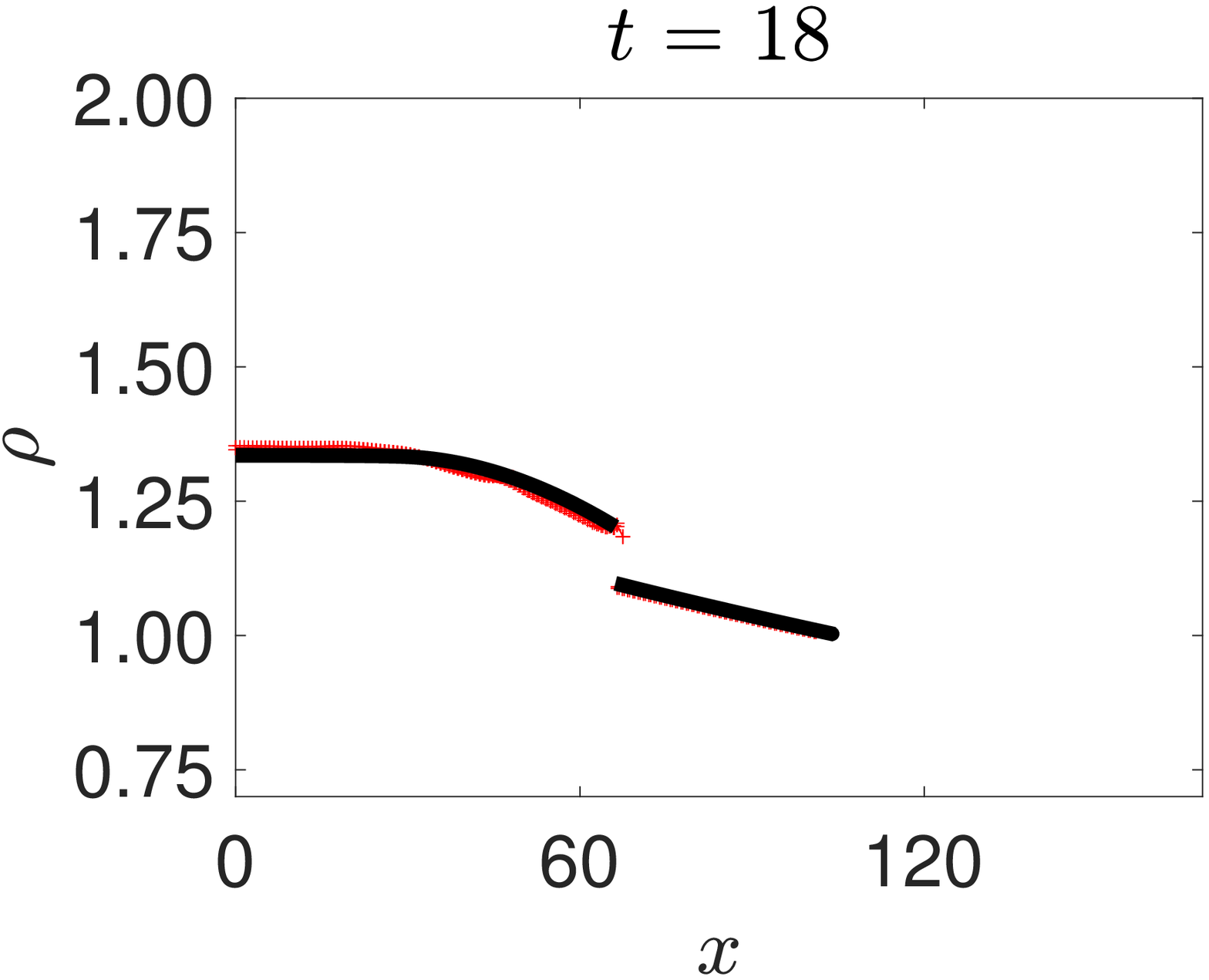}}\\
{\includegraphics[scale = 0.31]{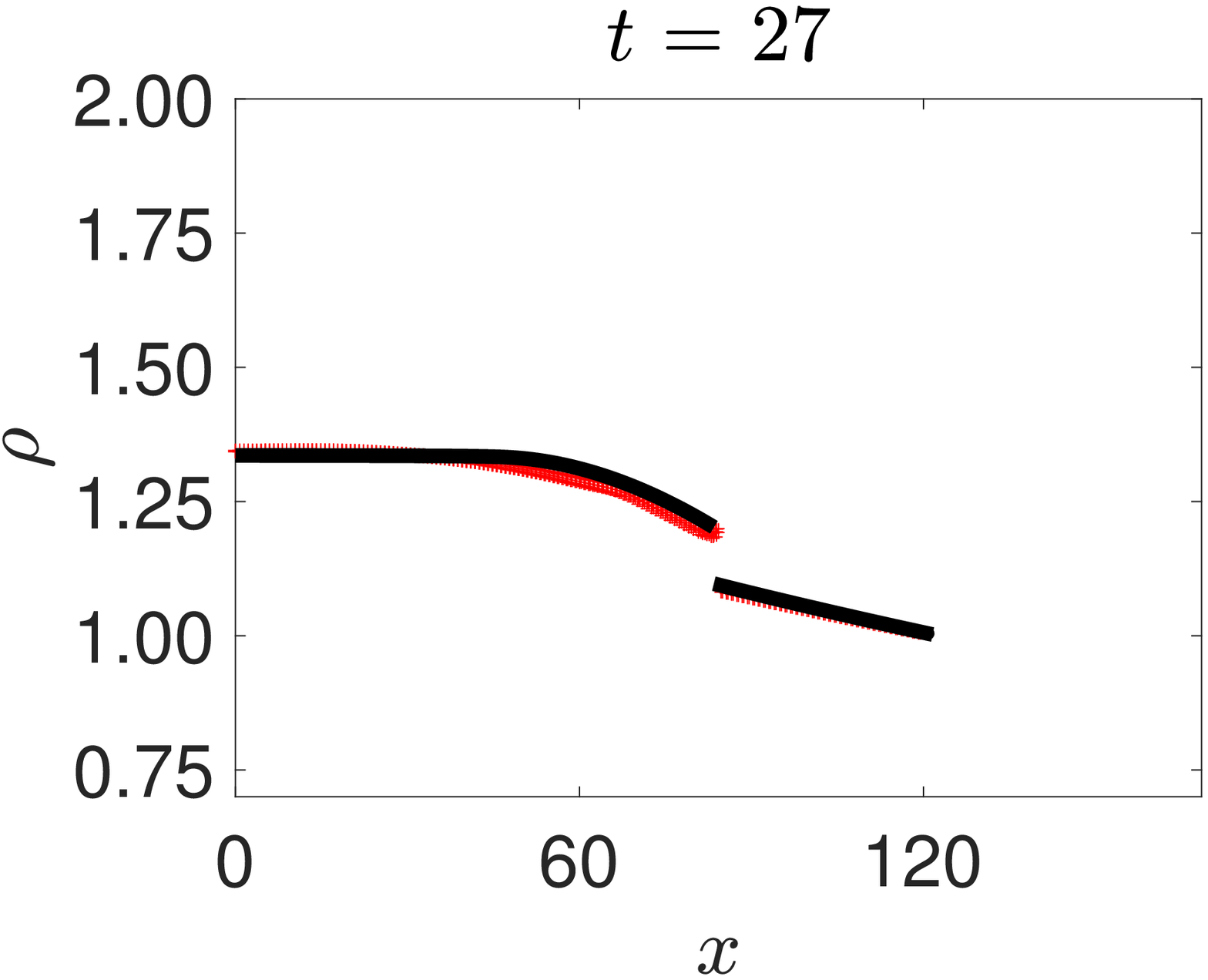}} 
{\includegraphics[scale = 0.31]{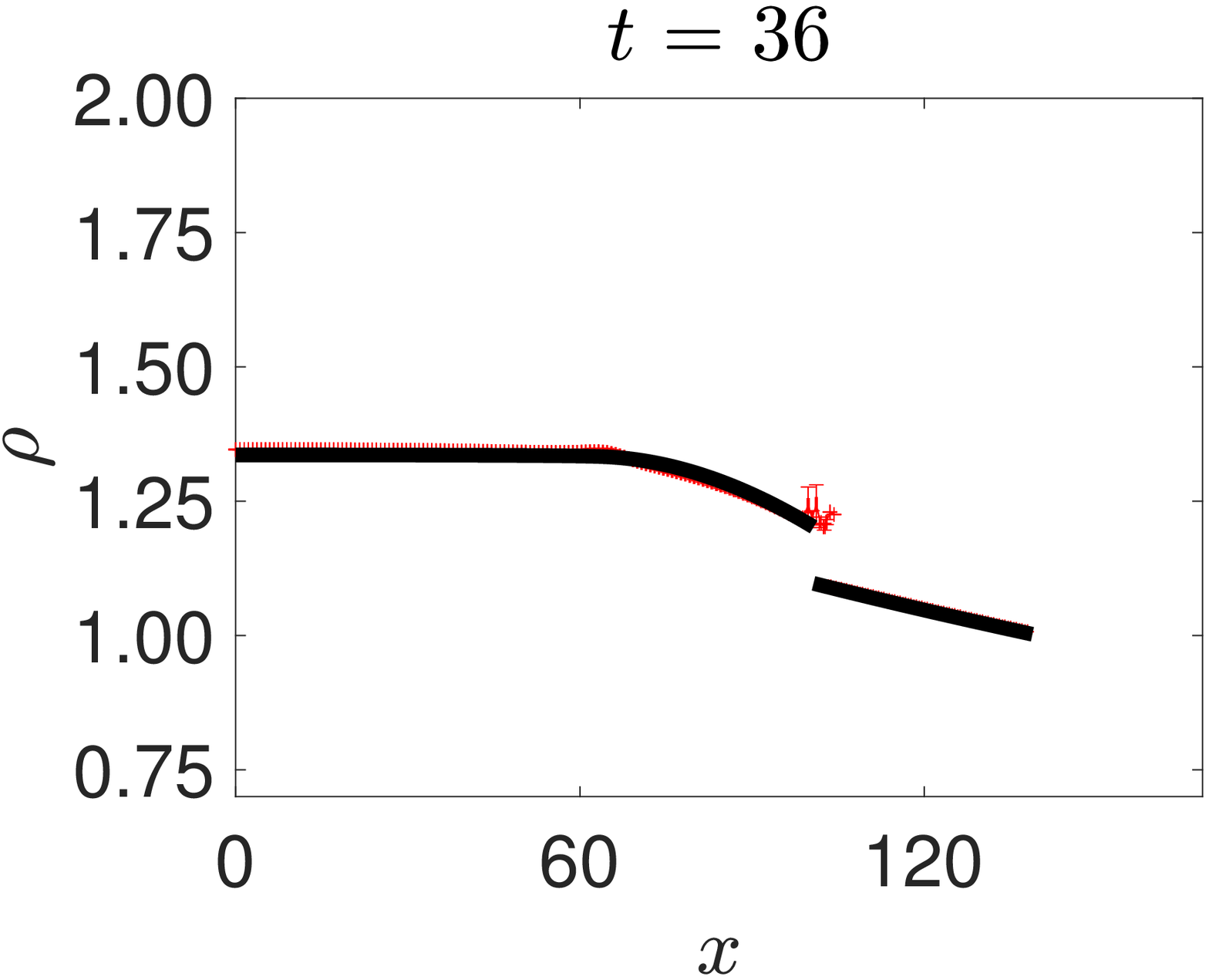}}
{\includegraphics[scale = 0.31]{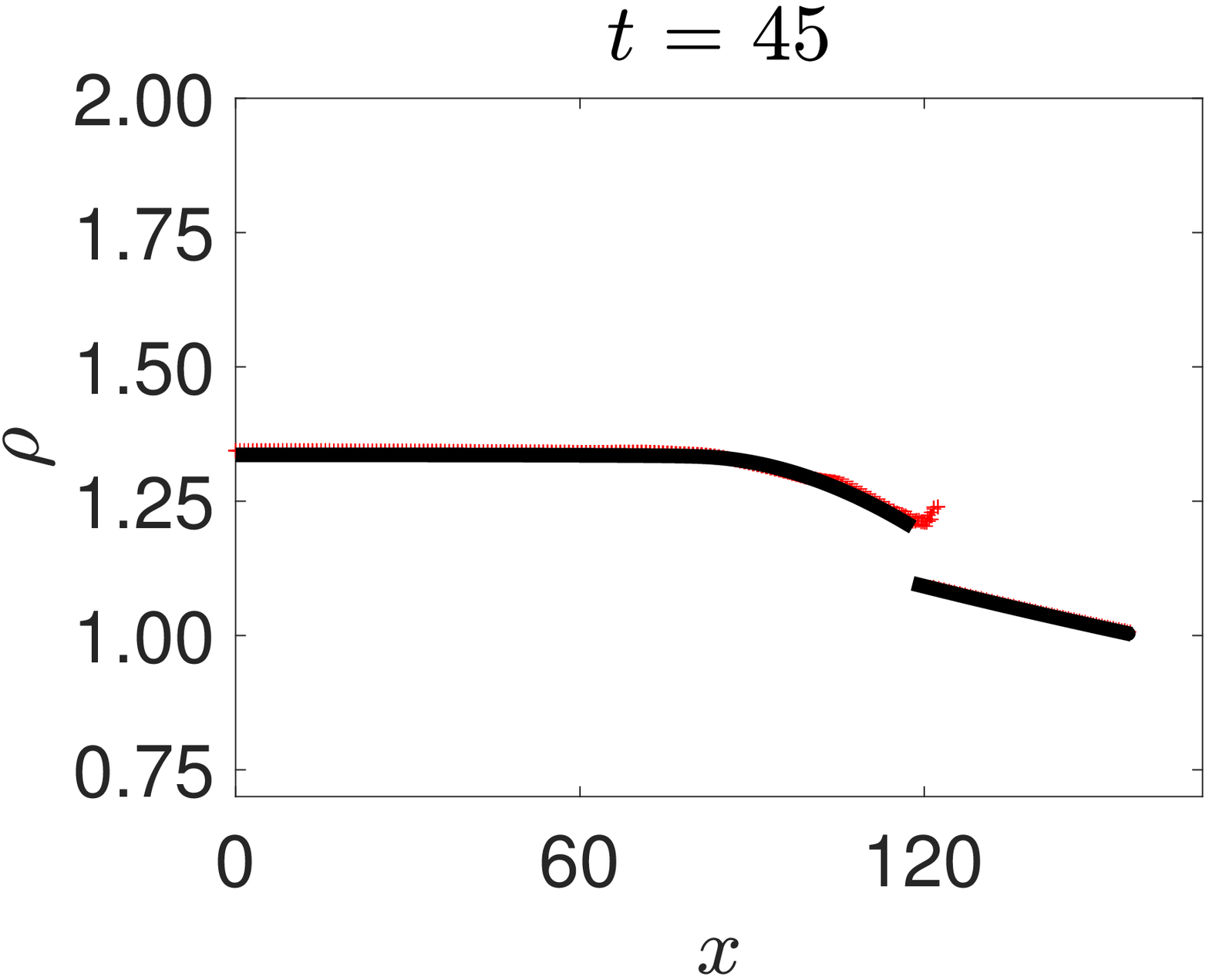}} 
\caption{{\it Comparison between the free-boundary problem and the individual-based model for $\eta_A>\eta_B$.} The cell density  $\rho(t,x)$, given by  \eqref{celldens}, is plotted against $x$ for increasing values of $t$. The cell densities $\rho_A$ and $\rho_B$  are either numerical solutions of the free-boundary problem (black lines) or approximate cell densities computed from simulation results for the individual-based model using~\eqref{celldensdisc} (red markers). The values of $x$ are nondimensionalised by $d^{\rm eq}$, while the values of $\rho$ are nondimensionalised by $\rho^{\rm eq}$.}\label{HigherMotilityNonProlifTwoPopJKRModelDensity}
\end{figure}

\section{Discussion}
\label{sec:discussion}
We presented an off-lattice individual-based  model that describes the dynamics of two contiguous cell populations with different proliferative and mechanical characteristics. 

We formally showed that this discrete model can be represented in the continuum limit as a free-boundary problem for the cell densities. We proved an existence result for the free-boundary problem and constructed travelling-wave solutions. We performed numerical simulations in the case where the cellular interaction forces are described by the celebrated JKR model of elastic contact, and we found excellent agreement between the computational simulation results for the individual-based model, the numerical solutions of the corresponding free-boundary problem and the travelling-wave analysis. Taken together, the results of numerical simulations demonstrate that the solutions of the free-boundary problem faithfully capture the qualitative and quantitative properties of the outcomes of the off-lattice individual-based model. 

In this paper, we focussed on a one-dimensional scenario where the two cell populations do not mix. It would be interesting to extend the individual-based mechanical model presented here, and the related formal method of derivation of the corresponding continuum model as well, to more realistic two-dimensional cases whereby spatial mixing between the two cell populations can occur. In this regard, an additional development of our study would be to formulate probabilistic discrete mechanical models of interacting cell populations and, using asymptotic methods analogous to those employed, for instance, in~\cite{champagnat2007invasion,fournier2004microscopic,oelschlager1989derivation,oelschlager1990large}, to perform  a rigorous derivation of their continuum counterparts. These are all lines of research that we will be pursuing in the near future.

\bibliographystyle{siamplain}
\bibliography{Biblio}

\appendix

\section{Approximate representation of the JKR force law} \label{JKR_approx}
The nonlinear function $F^{JKR}(d_{ij})$ that gives the JKR force law between the $i^{th}$ cell and the $j^{th}$ cell, with centres at distance $d_{ij}$, is implicitly defined by the following formulas \cite{drasdo2012modeling}
\beq
\label{eq:JKRimplicit}
\begin{aligned} 
&\delta_{ij} = \frac{a^2_{ij}} {R_{ij}} - \sqrt{\frac{2 \pi \gamma a_{ij}}{E_{ij}}}, \\
 & a^3_{ij} = a^3_{ij}\left(F^{JKR}_{ij}\right) = \frac {3 R_{ij}}{ 4 E_{ij}}\left[F^{JKR}_{ij} + 3 \pi \gamma R_{ij} + \sqrt{ 6 \pi \gamma R_{ij} F^{JKR}_{ij} + ( 3 \pi \gamma R_{ij})^2}\right],
\end{aligned}
\eeq
where
$$
R_{ij}^{-1}= R_i^{-1} + R_j^{-1}, \quad d_{ij} = R_i + R_j - \delta_{ij} \quad \text{and} \quad  E_{ij}^{-1}  = (1- \nu_i^2)E_i^{-1}+ (1-\nu_j^2) E_j^{-1}.
$$
In the formulas \eqref{eq:JKRimplicit}, the parameter $\gamma$ models the strength of cell-cell adhesion, $R_{i}$ stands for the radius of the $i^{th}$ cell, $E_{i}$ is the Young's modulus of the $i^{th}$ cell, and $\nu_{i}$ denotes the Poisson's ratio of the $i^{th}$ cell. Analogous considerations hold for the parameters of the $j^{th}$ cell. Moreover, $\delta_{ij}$ is the sum of the deformations undergone by the $i^{th}$ cell and the $j^{th}$ cell.

As the computational cost of numerical simulations carried out by solving implicitly for $F^{JKR}(d_{ij})$ is prohibitive, we derive an approximate representation of this function based on a third degree polynomial of the form
\begin{align}\label{FJKRApprox}
F^{JKR}(d_{ij}) \approx F(d_{ij}) = a^{ij}_{3}(d_{ij} -d^{\rm eq}_{ij})^3+a^{ij}_{2}(d_{ij}- d^{\rm eq}_{ij})^2+a^{ij}_{1}(d_{ij}-d^{\rm eq}_{ij}), 
\end{align}
with
\beq
\label{e.a123app}
a_{1}^{ij} =  F'(d^{\rm eq}_{ij}) \; d^{\rm eq}_{ij}, \quad a_{2}^{ij} = \frac 12 F''(d^{\rm eq}_{ij}) \; (d^{\rm eq}_{ij})^2, \quad a_{3}^{ij} = \frac 16  F'''(d^{\rm eq}_{ij}) \; ( d^{\rm eq}_{ij})^3.
\eeq
In the above equations, $d^{\rm eq}_{ij}$ denotes the equilibrium distance between the centres of cell $i$ and cell $j$ (\emph{i.e.} the distance $d_{ij}$ such that $F(d_{ij}) =0$ for all $d_{ij} \geq d^{\rm eq}_{ij}$). 

With this goal in mind, we look for explicit expressions of $d^{\rm eq}_{ij}$, $F'(d^{\rm eq}_{ij})$, $F''(d^{\rm eq}_{ij})$ and $F'''(d^{\rm eq}_{ij})$ in terms of the cell radii and the mechanical parameters of the cells. In the rest of this appendix we will use the abridged notation $F_{ij}$ for $F(d_{ij})$.

\paragraph{Expression for $d^{\rm eq}_{ij}$.} The equilibrium distance $d^{\rm eq}_{ij}$ can be directly computed from the formulas~\eqref{eq:JKRimplicit}. In fact, choosing $d_{ij} = d^{\rm eq}_{ij}$ and using the fact that $F(d^{\rm eq}_{ij})=0$, from the second formula in  \eqref{eq:JKRimplicit} we obtain
$$
a^3_{ij}(0) = \frac{9 \pi \gamma R_{ij}^2}{2 E_{ij}}.  
$$
Substituting this expression into the first formula in \eqref{eq:JKRimplicit} yields
$$
\begin{aligned} 
\delta_{ij}^{\rm eq}= \frac{R^{1/3}_{ij} (9\pi\gamma)^{2/3}}{2^{2/3}E_{ij}^{2/3}} - \frac{3^{1/3}2^{1/3}(\pi\gamma)^{2/3} R^{1/3}_{ij}}{E_{ij}^{2/3}}
=\frac 12 \frac{(\pi\gamma)^{2/3} (6R_{ij})^{1/3}}{E_{ij}^{2/3}}
\end{aligned} 
$$
and noting that $d^{\rm eq}_{ij} = R_i + R_j - \delta^{\rm eq}_{ij}$ we find the equilibrium distance $d^{\rm eq}_{ij}$.

\paragraph{Expression for $F'(d^{\rm eq}_{ij})$.} We substitute the second formula in  \eqref{eq:JKRimplicit} into the first formula to obtain
\beq
\label{eq:RipRj-dij}
R_i+R_j - d_{ij} = \frac 1{R_{ij}}\left( \frac{3R_{ij}}{4E_{ij}}\right)^{2/3} f(F_{ij})^{2/3} - \left(\frac{2\pi \gamma}{E_{ij}}\right)^{1/2} \left( \frac{3R_{ij}}{4E_{ij}}\right)^{1/6}
f(F_{ij})^{1/6}, 
\eeq
where 
$$
f(F_{ij}) = F_{ij} + \alpha +\sqrt{2\alpha} \sqrt{ F_{ij} + \alpha/2} \;  \text{ with } \;  \alpha = 3\pi \gamma R.
$$
Differentiating $f$ with respect to $F_{ij}$ yields 
$$
f'(F_{ij}) = 1 + \frac {\sqrt{2\alpha}} 2  \frac 1{\sqrt{F_{ij} + \alpha/2}} , \quad f''(F_{ij})  =  -  \frac {\sqrt{2\alpha}} 4  \frac 1{(F_{ij} + \alpha/2)^{3/2}}
$$
and
$$
f'''(F_{ij}) =    \frac {3\sqrt{2\alpha}} 8  \frac 1{(F_{ij} + \alpha/2)^{5/2}}.
$$
Hence, for $d_{ij} = d^{\rm eq}_{ij}$ we have
\beq
\label{eq:f0}
 f(0) =  2 \alpha= 6\pi \gamma R, \; f'(0) = 2, \; f''(0) = - \frac 1 \alpha = - \frac 1{3\pi \gamma R_{ij}},  \; f'''(0) =  \frac 3{ \alpha^2} = \frac 3{(3\pi\gamma R_{ij})^2}.
\eeq
Differentiating both sides of~\eqref{eq:RipRj-dij} with respect to $d_{ij}$ we find
$$
-1 =\left [ A\frac 2 3  (f(F_{ij}))^{-1/3}  - \frac B 6 (f(F_{ij}))^{-5/6}\right] f'(F_{ij}) F'_{ij}, 
$$
where 
$$
A = \frac 1{R}\left( \frac{3R_{ij}}{4E_{ij}}\right)^{2/3}= \frac 32  \frac 1{(6R_{ij})^{1/3} E_{ij}^{2/3}}, \quad B = \left(\frac{2\pi \gamma}{E_{ij}}\right)^{1/2} \left( \frac{3R_{ij}}{4E_{ij}}\right)^{1/6} = \frac{  (6R_{ij})^{1/6} (\pi \gamma)^{1/2}}{E_{ij}^{2/3}}.
$$
Rearranging terms in the latter equation yields
$$
F'_{ij} = - \left(f'(F_{ij}) \right)^{-1} \left [ A\frac 2 3  (f(F_{ij}))^{-1/3}  - \frac B 6 (f(F_{ij}))^{-5/6}\right] ^{-1}
$$
and, therefore, for $d_{ij} = d^{\rm eq}_{ij}$ we have
$$
F'(d^{\rm eq}_{ij}) = - \left(f'(0) \right)^{-1} \left [ A\frac 2 3  (f(0))^{-1/3}  - \frac B 6 (f(0))^{-5/6}\right] ^{-1}.
$$
Finally, noting that 
$$
- \frac{1}{2} \left [ A\frac 2 3  f(0)^{-1/3}  - \frac B 6 f(0)^{-5/6}\right]^{-1} =  \frac 6{5}  (6R_{ij} E_{ij})^{2/3}( \pi \gamma)^{1/3}
$$
we find
\beq
\label{eq:Ffirst}
F'(d^{\rm eq}_{ij}) = - \frac 3{5}  (6R_{ij} E_{ij})^{2/3}.
\eeq

\paragraph{Expression for $F''(d^{\rm eq}_{ij})$.} Differentiating twice both sides of~\eqref{eq:RipRj-dij} with respect to $d_{ij}$ yields
\begin{eqnarray*} 
0 &=& \left [ A\frac 2 3  (f(F_{ij}))^{-1/3}  - \frac B 6 (f(F_{ij}))^{-5/6}\right] f''(F_{ij}) \Big(\left(F'_{ij}\right)^2 + 
 f'(F_{ij}) F''_{ij} \Big)\\
&& + \left [ - A\frac 2 9  (f(F_{ij}))^{-4/3}  + \frac {5B} {36} (f(F_{ij}))^{-11/6}\right] 
\left( f'(F_{ij}) F'_{ij} \right)^2.
\end{eqnarray*} 
Rearranging terms in the latter equation gives 
$$
F''_{ij} = - \left(f'(F_{ij})\right)^{-1} f''(F_{ij}) \left(F'_{ij} \right)^2- \left(f'(F_{ij})\right)^2 \left(F'_{ij} \right)^3 \left [ A\frac 2 9  f(F_{ij})^{-4/3}  -  B\frac {5} {36} f(F_{ij})^{-11/6}\right].
$$
Hence, choosing $d_{ij} = d^{\rm eq}_{ij}$ and using the fact that $F(d^{\rm eq}_{ij})=0$ along with the expressions~\eqref{eq:f0} for $f'(0)$ and $f''(0)$, we find
\beq
\label{eq:Fsecond}
F''(d^{\rm eq}_{ij}) 
 = \frac 1{ 2\alpha} \left(F'(d^{\rm eq}_{ij}) \right)^2 - 4\left(F'(d^{\rm eq}_{ij})\right)^3\left [ A\frac 2 9  f(0)^{-4/3}  -  B\frac {5} {36} f(0)^{-11/6}\right].
\eeq
Noting that [cf. the expression of $f(0)$ in  \eqref{eq:f0}]
$$
 A\frac 2 9  f(0)^{-4/3}  - B \frac {5} {36} f(0)^{-11/6} = \frac 7{36}\frac 1{(6R_{ij})^{5/3} E_{ij}^{2/3} (\pi \gamma)^{4/3}} 
$$
and inserting the expression~\eqref{eq:Ffirst} of $F'(d^{\rm eq}_{ij})$ into~\eqref{eq:Fsecond} yields
$$
F''(d^{\rm eq}_{ij})   = 
\left(\frac 9{25} + 4\frac {27}{125}\frac {7}{36}\right) (6R_{ij}E_{ij})^{4/3}(\pi\gamma)^{2/3}\frac 1{6 R_{ij}\pi \gamma } = 
\frac{66}{125} \frac{E_{ij}^{4/3} (6R_{ij})^{1/3} }{(\pi \gamma)^{1/3}}.
$$

\paragraph{Expression for $F'''(d^{\rm eq}_{ij})$.} Differentiating thrice both sides of~\eqref{eq:RipRj-dij} with respect to $d_{ij}$ yields
\begin{eqnarray*} 
0 &=& \left [ A\frac 2 3  f(F_{ij})^{-1/3}  - \frac B 6 f(F_{ij})^{-5/6}\right] \left[f'''(F_{ij}) \left(F'_{ij} \right)^3 + 
3 f''(F_{ij}) F'_{ij} F''_{ij} + f'(F_{ij}) F'''_{ij} \right]\\
&& + \left [ - A\frac 2 9  f(F_{ij})^{-4/3}  + \frac {5B} {36} f(F_{ij})^{-11/6}\right] f'(F_{ij}) F'_{ij} \left[f''(F_{ij}) \left(F'_{ij}\right)^2 +
f'(F_{ij}) F''_{ij} \right] \\
&& +  \left [  A\frac {8}{27}  f(F_{ij})^{-7/3}  - B \frac {55} {216} f(F_{ij})^{-17/6}\right] 
\left(f'(F_{ij}) F'_{ij} \right)^3.
\end{eqnarray*} 
Rearranging terms in the latter equation we obtain 
\begin{eqnarray*} 
F'''_{ij} &=& \left\{3\left [  \frac {2A} 9  f(F_{ij})^{-\frac 43}  - \frac {5B} {36} f(F_{ij})^{-\frac{11}6}\right]   F'_{ij}
\left[f''(F_{ij}) \left(F'_{ij} \right)^2  + f'(F_{ij}) F''_{ij} \right]  \right.
\\  &&
\left.   -  \left [ \frac {8 A}{27}  f(F_{ij})^{-\frac 73}  -  \frac {55 B} {216} f(F_{ij})^{-\frac{17}6}\right] 
\left( f'(F_{ij}) \right)^2 \left(F'_{ij}\right)^3 \right\}\left [ A\frac 2 3  f(F_{ij})^{-\frac 13}  - \frac B 6 f(F_{ij})^{-\frac 56}\right] ^{-1}\\
&& -\left(f'''(F_{ij}) \left(F'_{ij} \right)^3  + 
3 f''(F_{ij}) F'_{ij} F''_{ij} \right)\left(f'(F_{ij}) \right)^{-1}.
\end{eqnarray*} 
Hence, choosing $d_{ij} = d^{\rm eq}_{ij}$ and using the fact that $F(d^{\rm eq}_{ij})=0$, along with the expression~\eqref{eq:Fsecond} for $F''(d^{\rm eq}_{ij})$ and the expressions~\eqref{eq:f0} for $f'(0)$, $f''(0)$ and $f'''(0)$, we find
\beq
\label{eq:Fthird}
F'''(d^{\rm eq}_{ij}) = 48 C_1^2 \left(F'(d^{\rm eq}_{ij}) \right)^5 + \left(8C_2 - \frac 6 \alpha C_1\right) \left(F'(d^{\rm eq}_{ij}) \right)^4 - \frac {3}{4 \alpha^2} \left(F'(d^{\rm eq}_{ij})\right)^3,
\eeq
where 
\begin{eqnarray*} 
C_1 &=& \left [  A\frac 2 9  f(0)^{-4/3}  - \frac {5B} {36} f(0)^{-11/6}\right]  =  \frac 7{36} \frac 1{ E_{ij}^{2/3} (6R_{ij})^{5/3} (\pi \gamma)^{4/3}}  , \\ 
C_2 &=&   A\frac {8}{27}  f(0)^{-7/3}  - B \frac {55} {216} f(0)^{-17/6}\\
&=& \frac 1 {27}  (6R_{ij}\pi \gamma)^{-7/3}\left[8  \frac 1{R_{ij}^{1/3} E_{ij}^{2/3}} \left(\frac 3 4\right)^{2/3}  
- \frac {55} 8   \frac{  (6R_{ij})^{1/6} (\pi \gamma)^{1/2}}{E_{ij}^{2/3}  (6R_{ij}\pi \gamma)^{1/2}}   \right] \\ 
&=&\frac{41}{216} \frac 1{  E_{ij}^{2/3} (6R_{ij})^{8/3}(\pi \gamma)^{7/3}}.
\end{eqnarray*}
Finally, inserting the expression~\eqref{eq:Ffirst} for $F'(d^{\rm eq}_{ij})$ into~\eqref{eq:Fthird} we obtain 
$$
F'''(d^{\rm eq}_{ij}) =  \frac{1254}{3125}  \frac{E_{ij}^2}{\pi \gamma}.
$$

\paragraph{Expressions for $d^{\rm eq}_{ij}$, $a^{ij}_{1}$, $a^{ij}_{2}$ and $a^{ij}_{3}$.} Taken together the results from above give
\beq
\label{e.deqFeq}
\begin{aligned}
& d^{\rm eq}_{ij} = R_i + R_j - \frac 12 \frac{(\pi\gamma)^{2/3} (6R_{ij})^{1/3}}{E_{ij}^{2/3}},\\
&a^{ij}_{1} = d^{\rm eq}_{ij} \, F'(d^{\rm eq}_{ij}) = - \frac 3{5}  (6R_{ij} E_{ij})^{2/3}  ( \pi \gamma)^{1/3} \, d^{\rm eq}_{ij}, \\
&a^{ij}_{2} = \frac 12 \, F''(d^{\rm eq}_{ij}) \; (d^{\rm eq}_{ij})^2 = \frac{33}{125} \frac{E_{ij}^{4/3} (6R_{ij})^{1/3} }{(\pi \gamma)^{1/3}} \, (d^{\rm eq}_{ij})^2,\\
&a^{ij}_{3} = \frac 16  \, F'''(d^{\rm eq}_{ij}) \; ( d^{\rm eq}_{ij})^3 =  \frac{209}{3125}  \frac{E_{ij}^2}{\pi \gamma} \, (d^{\rm eq}_{ij})^3.
\end{aligned}
\eeq

\paragraph{Approximate representation of $F^{JKR}(d_{ij})$ used to perform numerical simulations.} To perform numerical simulations, we assumed that all cells have the same radius $R$, Young's modulus $E$ and Poisson's ratio~$\nu$, \emph{i.e.}
$$
R_i=R, \;\; E_i=E \; \text{ and } \; \nu_i=\nu \; \text{ for all } i=1, \ldots, n.
$$
Under these assumptions, we have that
$$
R_{ij} = \frac{R}{2} \; \text{ and } \; E_{ij} = \tilde E =  \frac{E}{2 (1- \nu^2)}\; \; \text{ for all } i,j=1, \ldots, n,
$$
and the approximate representation of the JKR force law given by  \eqref{FJKRApprox} and \eqref{e.deqFeq} reads as
\beq
\label{e.FJKRl}
F^{JKR}(d_{ij}) \approx a_3(d_{ij} -d^{\rm eq})^3+a_2(d_{ij}- d^{\rm eq})^2+a_1(d_{ij}-d^{\rm eq}) \; \text{ for all } i,j=1, \ldots, n,
\eeq
with
\beq
\label{e.a123l}
\begin{aligned}
&d^{\rm eq} = 2R - \frac 12 \frac{(\pi\gamma)^{2/3} (3R)^{1/3}}{\tilde E^{2/3}}, \quad
&&a_1 = - \frac 3{5}  (3R \tilde E)^{2/3}  ( \pi \gamma)^{1/3} d^{\rm eq},  \\ 
&a_2 = \frac{33}{125} \frac{\tilde E^{4/3} (3R)^{1/3} }{(\pi \gamma)^{1/3}}( d^{\rm eq})^2,\quad
&&a_3 =  \frac{209}{3125}  \frac{\tilde E^2}{\pi \gamma}( d^{\rm eq})^3.
\end{aligned}
\eeq

\newpage
\section{Supplementary figure}
\renewcommand\thefigure{\thesection.\arabic{figure}}    
\setcounter{figure}{0}

\begin{figure}[h!]
\centering
\subfigure[]{\includegraphics[scale = 0.27]{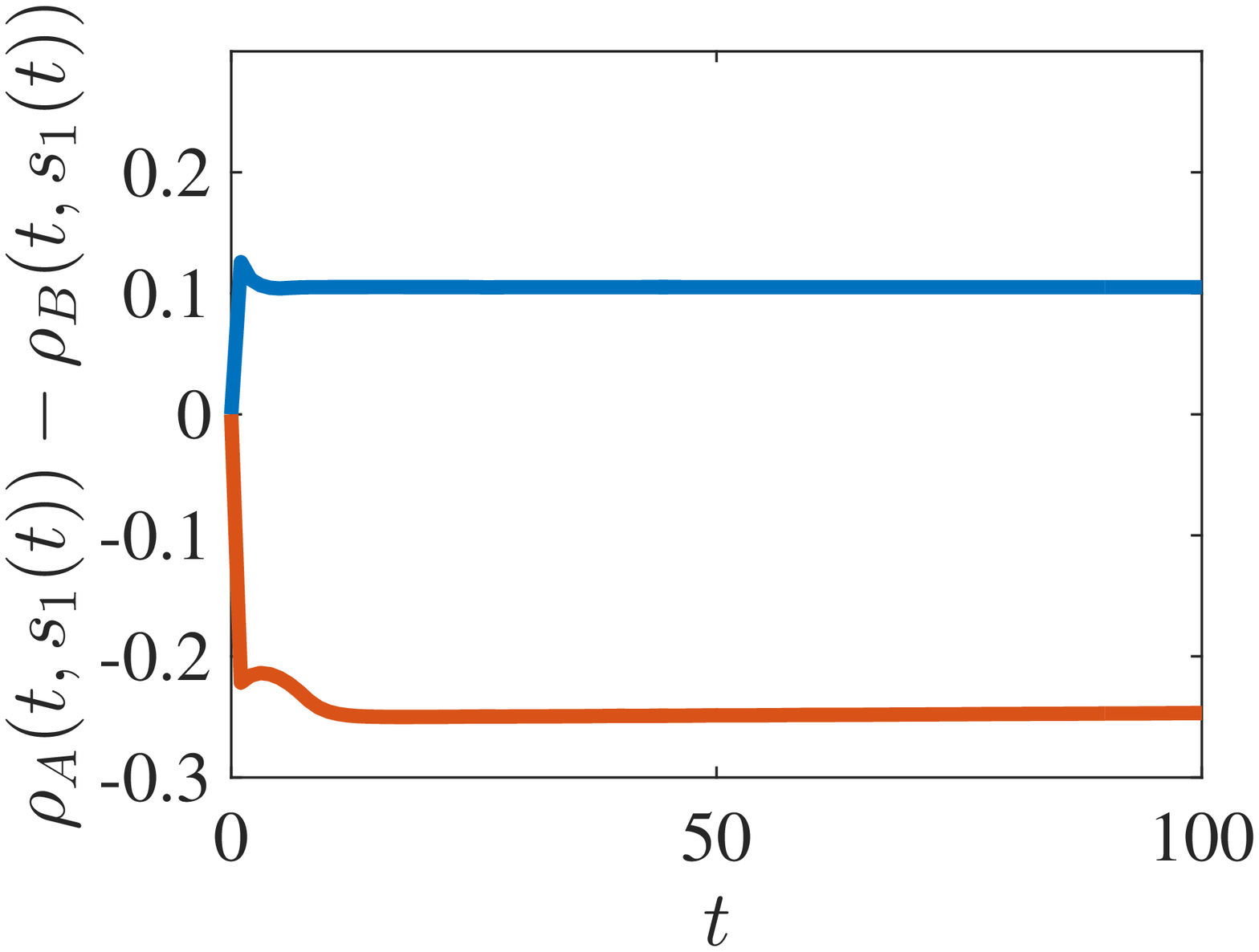}}  
\subfigure[]{\includegraphics[scale = 0.27]{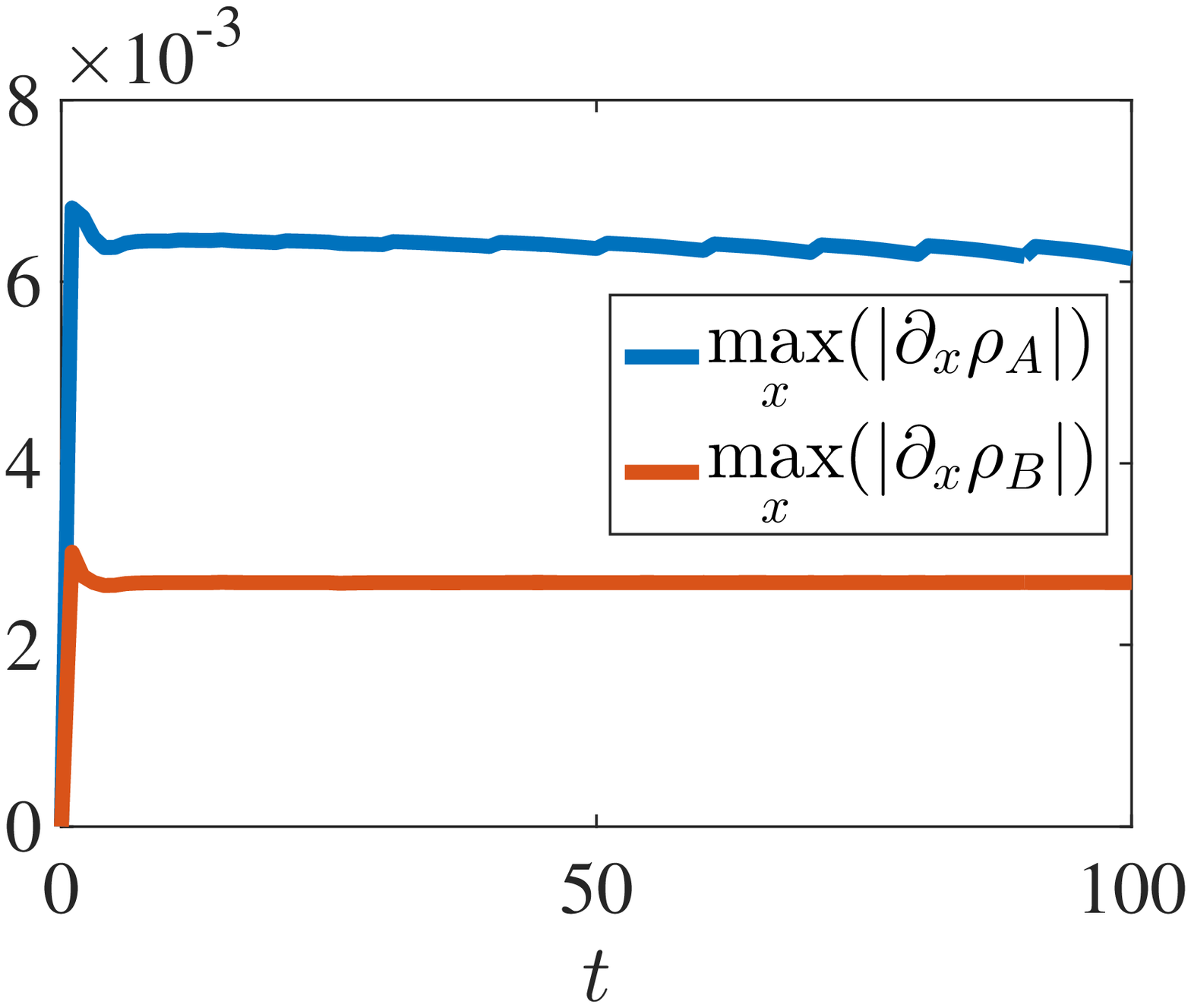}} 
\subfigure[]{\includegraphics[scale = 0.27]{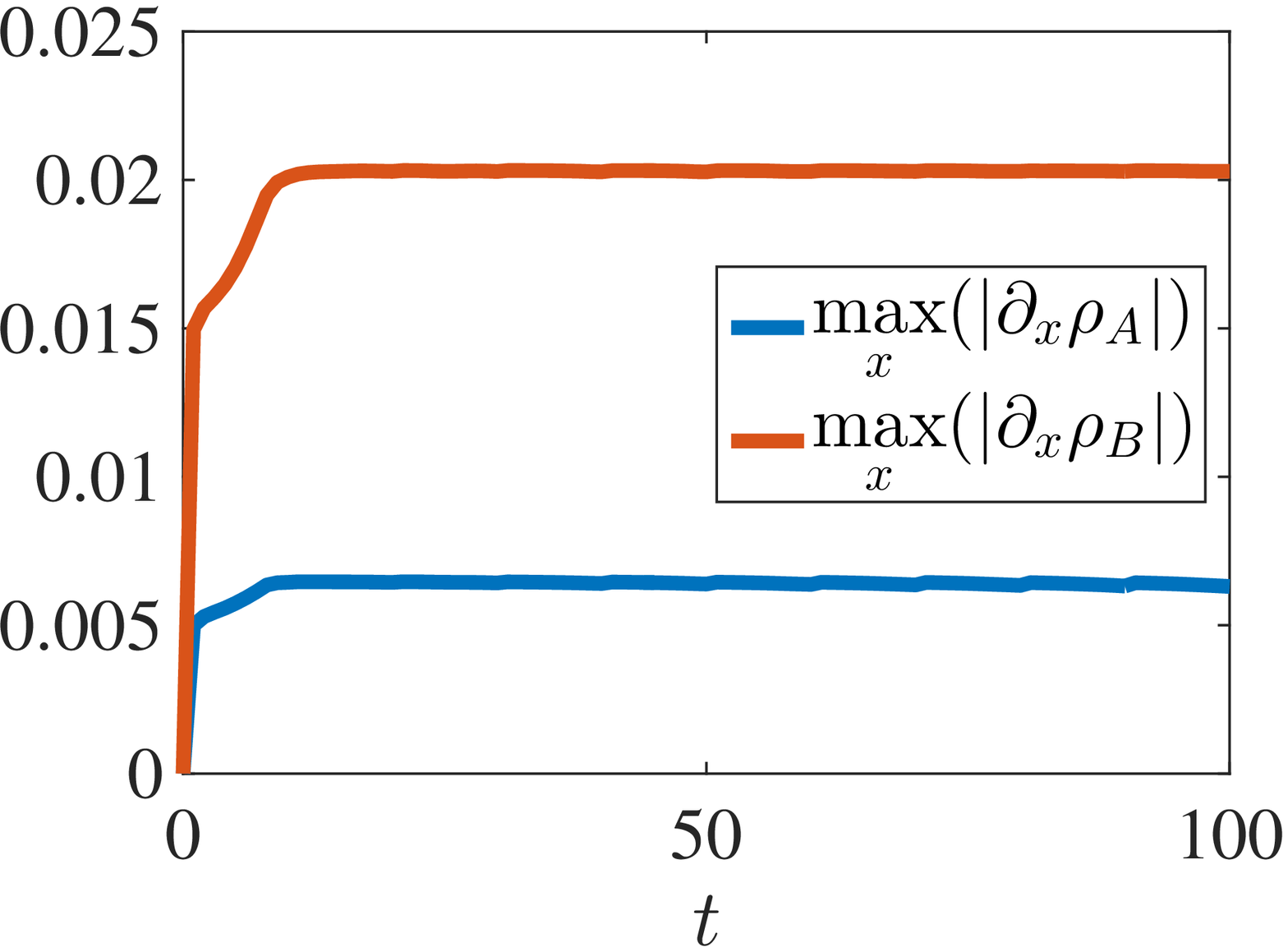}}
\caption{{\it (a)} Plot of $\rho_A(t,s_1(t)) - \rho_B(t,s_1(t))$ against $t$ in the case where $\eta_A<\eta_B$ (orange line), which corresponds to Figure \ref{LowerMotilityNonProlifTwoPopJKRModelDensity}, and in the case where $\eta_A>\eta_B$ (blue line), which corresponds to Figure \ref{HigherMotilityNonProlifTwoPopJKRModelDensity}. {\it (b)} Plot of $\displaystyle{\max_x(|\partial_x \rho_A|)}$ (blue) and $\displaystyle{\max_x(|\partial_x \rho_B|)}$ (orange) against $t$ in the case where $\eta_A<\eta_B$, which corresponds to Figure \ref{LowerMotilityNonProlifTwoPopJKRModelDensity}. {\it (c)} Plot of $\displaystyle{\max_x(|\partial_x \rho_A|)}$ (blue) and $\displaystyle{\max_x(|\partial_x \rho_B|)}$ (orange) against $t$ in the case where $\eta_A>\eta_B$, which corresponds to Figure \ref{HigherMotilityNonProlifTwoPopJKRModelDensity}. See Figures \ref{LowerMotilityNonProlifTwoPopJKRModelDensity} and \ref{HigherMotilityNonProlifTwoPopJKRModelDensity} for further details. }
 \label{Figure_5}
\end{figure}

\end{document}